\documentclass[opre,nonblindrev]{informs3}

\DoubleSpacedXI

\usepackage{endnotes}
\let\footnote=\endnote

\usepackage{natbib}
 \bibpunct[, ]{(}{)}{,}{a}{}{,}%
 \def\bibfont{\small}%
\TheoremsNumberedThrough     
\ECRepeatTheorems

\EquationsNumberedThrough    

\usepackage{subcaption}
\usepackage[utf8]{inputenc} 
\usepackage[T1]{fontenc}    
\usepackage{url}            
\usepackage{booktabs}       
\usepackage{amsfonts}       
\usepackage{nicefrac}       
\usepackage{microtype}     
\usepackage{thmtools}
\usepackage{amsmath}  
\usepackage{bm}
\usepackage{threeparttable}
\usepackage{float}
\usepackage[table]{xcolor}
\usepackage{graphicx}
\usepackage{placeins} 
\usepackage{todonotes}
\usepackage{thm-restate}

\newcommand{\prob}[1]{\mathbb{P}\prtr{#1}}

\newcommand{\expee}[2]{\mathbb{E}_{{#1}}\prts{#2}}

\newcommand{\vect}[1]{\mathbf{#1}}

\newcommand{\norm}[1]{\lVert #1 \rVert}

\newcommand{\abs}[1]{\left|{#1}\right|}
\newcommand{\ceiling}[1]{\left\lceil{#1}\right\rceil}

\newcommand{\prtc}[1]{\left\{{#1}\right\}}
\newcommand{\prts}[1]{\left[{#1}\right]}
\newcommand{\prtr}[1]{\left({#1}\right)}
\newcommand{\prta}[1]{\langle{#1}\rangle}

\newcommand{\grad}{\nabla}

\newcommand{\events}[1]{\mathbf{1}_{\mathcal{E}_{t,S}}}

\newcommand{\ProblemRare}{\ensuremath{\mathbf{P}_{\lambda}}}

\def\theARTICLETOP{} 

\makeatletter
\def\@evenhead{%
  \vbox{\hbox to \textwidth{\normalfont\small\thepage\hfill}\vskip 4pt\hrule height 0.4pt}}
\def\@oddhead{%
  \vbox{\hbox to \textwidth{\normalfont\small\hfill\thepage}\vskip 4pt\hrule height 0.4pt}}
\makeatother
\begin{document}

\RUNTITLE{Safe Start}
\TITLE{Safe Start: Configuring Optimization Algorithms for Decision-Making under Extreme Risks}


\ARTICLEAUTHORS{%
\AUTHOR{Henry Lam}
\AFF{Department of Industrial Engineering and Operations Research, Columbia University, New York, NY 10027, \EMAIL{khl2114@columbia.edu}}
\AUTHOR{Wasin Meesena}
\AFF{Department of Industrial Engineering and Operations Research, Columbia University, New York, NY 10027, \EMAIL{wm2501@columbia.edu}}
} 
\RUNAUTHOR{Lam and Meesena}

\ABSTRACT{%
We consider stochastic optimization where the goal is not only to optimize an average-case objective, but also to mitigate the occurrence of rare catastrophic events. This problem is motivated by safety-aware decision-making and AI training. We first argue that, in the presence of a simulation model, natural attempts to integrate variance reduction into optimization, even executed in a reasonable adaptive fashion, encounter fundamental challenges in guaranteeing realistic runtime when using common stochastic gradient descent algorithms. This challenge arises from the extreme sensitivity of tail-based objectives with respect to the decision variables, which renders a dichotomic failure of convergence regardless of what step size we select. We offer remedies based on a new notion of \emph{safe start} that allows for efficient finite-time error control, and show how the sampling complexity scales favorably under the combination of safe start and variance reduction. We illustrate our methodologies on examples in portfolio optimization and robust classification with neural networks.}%




\maketitle

\section{Introduction}
\label{sec:intro}

In many high-stakes problems, it is critical for decision-making to account for the occurrence of rare catastrophic events, in addition to standard average-case performances. For example, in designing human-interacting physical systems such as self-driving vehicles, the objective must prioritize the prevention of road conflicts and fatalities \citep{huang2017accelerated,o2018scalable}. Similarly, financial portfolio management needs to balance gains with hedges against extreme losses \citep{embrechts2013modelling,mcneil2015quantitative}. In simulation modeling, estimating these catastrophic events, and understanding how they occur, has been a long-standing focus under the umbrella of \emph{rare-event simulation}. Despite many established approaches in this literature, decision optimization with the goal of \emph{preventing} catastrophic events appears to be substantially open and has only very recently started to gather attention \citep{he2024adaptive,tong2022-opt-under-rare,blanchet2024efficient,deo2025achieving}. On a high level, this work attempts to systematically understand and remedy the challenges in running optimization algorithms for objectives that critically avoid rare events. As we will explain, this involves several novel fundamental issues that deviate from the established literature in the \emph{evaluation} of rare event likelihoods.

We consider the following generic ``rare-event optimization'' formulation:
\begin{equation}\label{eq:start-formulation}
     \min_{\vect{x} } F(\vect{x}) = f(\vect{x}) + \gamma p(\vect{x}).
\end{equation}
where $\vect{x}$ is the decision, $f$ the expectation of some loss function, and $p$ an extreme-risk term, i.e., it involves the tail of the underlying randomness, such as the probability of a rare event or value-at-risk at a high level. \(\gamma>0\) is a weighting parameter that balances the expected loss and the extreme risk. To provide some perspective, first, note that one can typically choose a conservative decision $\vect x$ to almost fully avoid the catastrophic event captured in $p(\vect x)$. This, however, would usually be impractical and thus, for most problems, accounting for catastrophic events means a balance between average-case performance and the extreme risk, which is exactly what \eqref{eq:start-formulation} is motivated from. Second, as the weighting parameter $\gamma$ increases, more weight is put on the extreme risk. Here, $\gamma$ should be roughly of order $1/p(\vect{x}^*)$, where $\vect{x}^*$ is an optimal solution, for formulation \eqref{eq:start-formulation} to be meaningful. This is because this is the order of $\gamma$ that would allow the two terms in \eqref{eq:start-formulation} to be of the same magnitude at $\vect{x}^*$; otherwise, it would mean that we either focus on the average-case only, or the target solution is highly conservative. To explain this balancing further, as well as understanding the choice of $\gamma$ more concretely, formulation \eqref{eq:start-formulation} could be viewed in two ways/examples:

\noindent\emph{Chance-constrained optimization:} A natural optimization approach to explicitly avoid target risks is a chance-constrained optimization \citep{ben2009robust,calafiore2005uncertain}, namely \begin{equation}\label{CCP}
     \min_{\vect{x} } f(\vect{x}) \text{\ \ subject to\ \ } p(\vect{x})\leq\tilde\gamma,
\end{equation}
where $\tilde\gamma$ is often called the tolerance level. If $p$ denotes the probability of an adversarial event, then \eqref{CCP} stipulates that the decision must explicitly control this probability to be within $\tilde\gamma$. When $\tilde\gamma$ is a low level, e.g., $0.001\%$, this means $p( \vect{x}^*)$ must be similarly small, and a heuristic calculation would reveal that $\gamma$ acts as the Lagrangian multiplier for the chance constraint in \eqref{CCP}, and is of order $1/p(\vect{x}^*)$.


\noindent\emph{Extreme quantile estimation/regression:} Suppose $\xi\in\mathbb R$ is a random variable. Then, taking $f(\vect x)=\vect x$ and $p(\vect x)= \expee{\xi}{\xi-\vect x}_+$ (where $\vect x\in\mathbb R$) in \eqref{eq:start-formulation} is precisely the quantile estimation problem, for the target level $1/\gamma$. Moreover, this objective itself becomes the conditional value-at-risk. In this formulation, $\gamma$ is the reciprocal of the tail probability of $\xi$, which in many cases has the same magnitude as $p(\vect x)$ (ignoring logarithmic terms). Moreover, the above can be generalized to quantile regression \citep{chernozhukov2017extremal,pasche2024neural}, where $\vect x$ now represents the set of parameters (e.g., linear coefficients) in a quantile regression model $f_{\vect x}(\vect z)$ with covariates $\vect z$, and $p(\vect x)= \expee{\vect z,\xi}{\xi-f_{\vect x}(\vect z)}_+$.


Our goal is to obtain a good solution for \eqref{eq:start-formulation}, under the paradigm that the underlying randomness $\xi$ is simulable and, moreover, susceptible to variance reduction. In the following, we will discuss the challenges through several interacting layers: The basic issue of crude Monte Carlo that has driven much of the rare-event simulation literature (Section \ref{sec:basic}), the ``curse of circularity'' in rare-event optimization (Section \ref{sec:curse of circularity}), and finally the challenge in sampling complexity conversion (Section \ref{sec:lack of bridge}). The last one especially concerns the (lack of) bridge between the classical efficiency notion in the rare-event literature and optimization efficiency, and is the main focus of this paper that motivates our novel ``safe start'' notion (introduced in Section \ref{sec:contributions}).


\subsection{Basic Evaluation Challenge and Remedy via Variance Reduction} \label{sec:basic}
The established literature of rare-event simulation studies the evaluation of rare-event probabilities or associated risk quantities \citep{bucklew2004introduction,juneja2006rare,blanchet2012state}. Suppose $p=P(\xi\in S)$ denotes a rare-event probability of $\xi$ falling into a region $S$ (there is no decision $\vect x$ here to make). A natural estimator would be the sample proportion of hits onto $S$. This, however, is a poor estimator when $p$ is tiny because, as readily intuited, more likely than not we would simply output 0 as the estimate due to 0 hits. That is, we fail to estimate the magnitude of $p$, unless we use an extremely large sample size. To mathematically articulate this issue, consider an estimate ``good'' if the discrepancy between the estimate and the target probability is within $\epsilon$, relative to the target probability (instead of the absolute discrepancy), i.e., $|\hat p-p|\leq\epsilon p$. Under this criterion, 0 would be a lousy estimate. Now, from the Chebyshev inequality we have $P(|\hat p-p|>\epsilon p)\leq\sigma^2/(\epsilon^2p^2n)$, where $\sigma^2$ is the per-run variance. From this, the required sample size to obtain a ``good'' estimate with $(1-\alpha)$ confidence level, i.e., $P(|\hat p-p|>\epsilon p)\leq\alpha$, can be deduced by enforcing $\sigma^2/(\epsilon^2p^2n)\leq\alpha$, thus giving the required size $\sigma^2/(\epsilon^2p^2\alpha)$. Here, the ratio between the standard deviation and mean, i.e., $\sigma/p$, is known as the relative error. This quantity controls the required sample size in that, the larger the squared relative error, the larger the sample size is proportionately. Unfortunately, for naive sample proportion, its relative error is $\sqrt{(1-p)/p}$, which blows up the required sample size when $p$ is tiny.

To tackle the above challenges, \emph{variance reduction} methods \citep{asmussen2007stochastic,glasserman2004monte} aim to drive down the variance $\sigma^2=p(1-p)$ in sample proportion, or crude Monte Carlo, to much smaller: In large deviations settings, this usually means, up to a logarithmic factor, order $p^2$ instead of $p$ as $p\to0$, which translates to a sample size that only depends on $p$ logarithmically. These variance reduction methods include importance sampling \citep{bucklew2004introduction,juneja2006rare,blanchet2012state}, multilevel splitting \citep{glasserman1999multilevel,dean2009splitting,villen2011rare}, and their variations \citep{de2005tutorial,rubinstein1997optimization,botev2013markov}.




\subsection{Curse of Circularity and Remedy via Adaptivity}\label{sec:curse of circularity}
The aforementioned evaluation challenge propagates to optimization problem \eqref{eq:start-formulation}: In order to solve \eqref{eq:start-formulation}, we first need to have a good estimate of $p(\vect x)$ in the objective function. This, however, gives rise to a ``curse of circularity'' \citep{he2024adaptive,aolaritei2025stochastic}. To explain, most rare-event variance reduction schemes, prominently importance sampling, are known to be ``double-edged swords'', in that they are highly sensitive to the problem instance, i.e., the scheme needs to be well-designed and well-tuned according to the problem configuration to achieve variance reduction, otherwise it can perform very poorly. However, by the nature of optimization, this configuration is not known in advance since the decision has not been obtained. That is, on one hand a good optimizer needs a good variance reduction scheme, and on the other hand a good variance reduction scheme requires knowledge about the optimal solution, which leads to the ``curse of circularity''. To break this curse, \cite{he2024adaptive} and \cite{aolaritei2025stochastic} use \emph{adaptivity}, by iteratively updating solutions where at each iteration the variance reduction is applied as if the current solution is optimal. In particular, stochastic gradient descent (SGD) or stochastic approximation \citep{nemirovski2009robust,beck2017first} is a natural iterative procedure to embed adaptive variance reduction. \cite{he2024adaptive} and \cite{aolaritei2025stochastic} show central limit theorems for the updated solutions that reveal, indeed, that such adaptive procedures enjoy asymptotic variance that is within-class the best possible, i.e., it is on par with the variance achieved as if the variance reduction scheme was well-tuned knowing the optimal solution in advance.

\subsection{Lack-of-Bridge to Sampling Complexity and Remedy via Safe Start}\label{sec:lack of bridge}
While central limit theorems and asymptotic variance provide positive guidance on the choice and performance of adaptive variance reduction, such results do not inform the sampling requirements to achieve a good solution. This issue is critically distinct from the evaluation task -- In evaluation, the variance of an estimator translates \emph{directly} into the required sample size, since the standard Chebyshev inequality reveals that the latter is proportional to the (squared) relative error $\sigma^2/p^2$. In optimization, the situation is substantially more complicated since there are no longer simple guiding inequalities to inform the conversion of relative error into sampling complexity. 

Table \ref{tab:rare_event_comparison} summarizes the key distinctions between estimation and optimization that we have discussed so far. In estimation, to address the inefficiency of naive Monte Carlo arising from rarity, variance reduction methods have been well-established, and their challenge is often the sensitivity of their performance to problem configuration which necessitates careful design and tuning. Nonetheless, the analysis framework using the notion of relative error is well-established and justified, since a small relative error immediately leads to a small sampling complexity thanks to the simple Chebyshev's inequality. In contrast, for optimization that has only been recently investigated, the curse of circularity arises from the sensitivity of variance reduction, this time due to the lack of knowledge on the optimal solution a priori. To this end, \cite{he2024adaptive} and \cite{aolaritei2025stochastic} propose adaptive variance reduction as a remedy, by leveraging existing good-for-evaluation rare-event estimation tools. However, their analyses are based on asymptotic variance (which in turn relates to relative error). The main challenge now is that there is no established concentration result that directly governs the translation from variance-based criteria to sampling complexity for optimization algorithms.












\begin{table}[htbp]



\scriptsize
\linespread{1.5}\selectfont
\caption{Comparison of Estimation vs. Optimization under Rare Events }
\label{tab:rare_event_comparison}
\renewcommand{\arraystretch}{2.5} 
\begin{tabular}{@{} >{\raggedright\arraybackslash}p{0.20\linewidth} p{0.37\linewidth} p{0.37\linewidth} @{}}
\toprule
\textbf{Task} & \textbf{Estimation} & \textbf{Optimization} \\
\midrule

\textbf{Success Criterion} & 
An estimator $\hat{p}$ with small relative discrepancy with the target rare-event probability $p$: \newline $\mathbb{P}(|\hat{p} - p| \leq \epsilon p) \approx 1$. & 
Output $\hat{x}$ close to the minimizer $\vect{x}^*$ of $F$: \newline $\mathbb{P}(F(\hat{x}) \leq (1 + \epsilon)F(\vect{x}^*)) \approx 1$. \\

\textbf{Sampling Complexity }\textbf{(Computation Cost)} & 
Required total number of samples simulated ($N$). & 
Required total number of samples across the entire SGD trajectory: \newline Iterations ($T$) $\times$ Batch Size ($B$). \\

\textbf{Existing Rarity-Induced Challenge and Remedy}  & 
Naive Monte Carlo is inefficient, and variance reduction can enhance sampling efficiency. However, designing good variance reduction typically needs to be analyzed case-by-case due to their ``double-edged sword" nature and sensitivity to problem configuration. &
Curse of circularity arises from the sensitivity of variance reduction scheme to the unknown optimal solution. Adaptive algorithm has been proposed as a remedy.\\

\textbf{Efficiency Notion}  & 
Relative error of estimator $\hat{p}$, $RE(\hat{p}) = \sqrt{\text{Var}(\hat{p})} / p$, is small (bounded or logarithmic in $p$ as $p\to0$). & 
Relative error of gradient estimator in the adaptive variance reduction scheme is small (Assumption \ref{assumption:second-moment}), and the \textbf{Safe Start} condition (Assumption \ref{def:safe-start}). \\


\textbf{Translation from Efficiency Notion to Sampling Complexity}  & 
 Chebyshev's inequality establishes that $N = \mathcal{O}(RE^2(\hat{p}))$. So a small (squared) relative error guarantees a proportionately small sampling complexity. &
\textbf{Our SGD analysis} establishes that a safe start initialization + efficient adaptive variance reduction for gradient estimates guarantees a small sampling complexity. \\

\bottomrule
\end{tabular}
\end{table}

\subsection{Our Contributions}\label{sec:contributions}
Our main contribution is to fill in the aforementioned significant gap, to understand what efficiency notions and algorithmic requirements drive the sampling complexity for rare-event optimization. Our theoretical results are two-fold: On the negative front, we show that, in general, embedding adaptive variance reduction into an iterative algorithm like SGD is fundamentally inadequate in achieving desirable sampling complexity, even if the variance reduction is excellent for the evaluation task. On the positive front, we show that, by using \emph{safe start} -- initializing the solution to be in a ``safe'' region where the risk term $p(\vect x)$ in \eqref{eq:start-formulation} is suitably small, we can guarantee a sub-exponential sampling complexity instead of an exploding exponential complexity, as long as we use a good-for-evaluation variance reduction. The last two entries in the ``Optimization'' column of Table \ref{tab:rare_event_comparison} highlight our main novelty: By introducing this new ``safe start'' notion as part of the algorithmic requirement, we argue, through our elaborate SGD analysis, that we can attain high sampling efficiency. Moreover, we demonstrate that this requirement is necessary, and so is the embedment of efficient adaptive variance reduction.


We briefly explain the intuition of safe start. At a high level, a reason why there is no easy translation from variance-based criteria to sampling complexity is the ultra-sensitive landscape of the objective function \eqref{eq:start-formulation} due to the rare-event penalty term. Across different solutions, rare-event probabilities can vary exponentially and, because of this, ``risky'' solutions can exhibit exponentially larger gradients than ``safe'' solutions. Consequently, it could happen that, no matter what step size we choose in the SGD, we would need an exponential number of iterations to obtain an adequately accurate solution, and this is the case even if we already embed efficient variance reduction scheme adaptively into the gradient estimation at each SGD iteration. Basically, we would either choose step size too small, in which case we have slow convergence, or risk making unstable big jumps, in which case we end up at a far-away position that again takes a lot of iterations to reach optimality. \emph{The key role of safe start is to prevent this exponential number of iterations, arising from the ultra-sensitivity of the objective landscape, from materializing.}



\subsubsection{An illustrative example.}

We give a preview of our main insights through the following example. Consider estimating the extreme $q$-quantile $x^*$ of a standard normal via the stochastic objective $\min_x \left\{ x + \frac{1}{1-q}\mathbb{E}[(Z - x)_+] \right\}$ using SGD with embedded variance reduction. Figure \ref{fig:teaser}a shows the convergence probability to $x^*$ across various SGD configurations including their constant step sizes ($\eta_0$ on the y-axis) and initializations ($x_0$ on the x-axis). Here, the convergence probability is defined as $\prob{\abs{x_{2000}-x^*}\leq 0.1}$ where $x_{2000}$ denotes the 2000-th iterate, and this probability is estimated via abundant repetition of SGD runs. In the heatmap, dark green cells represent configurations with convergence probability close to $1$, meaning they can reliably reach a tight neighborhood around $x^*$, while dark red cells indicate convergence probability close to $0$. The colors within this spectrum indicate convergence probabilities in between.

The blue dotted vertical line is the true quantile $x^*$. We see a sharp change of convergence behavior from the left to the right hand side of this line. On the left, the SGD fails to converge no matter what initialization and step size we choose (except when the configuration is very close to the vertical line). On the right hand side, \emph{all} initializations converge with high probability, by selecting a well-tuned step size. In this example, any initialization $x_0 \ge x^*$ is a safe start, since its rare-event penalty term $\mathbb{E}[(Z - x_0)_+]$ is less than that of $x^*$. That is, this example shows how enforcing a safe start can lead to efficient SGD while, in a very abrupt manner, not satisfying this condition leads to algorithmic failure.




Figure \ref{fig:teaser}b gives a closer look at the failure of SGD under an unsafe start ($x_0=0$). With a conservative step size ($\eta_0 = 10^{-5}$), the gradient updates move slowly. Conversely, increasing to $\eta_0 = 10^{-4}$ causes the trajectory to overshoot to $x_1 \approx 5.4$, followed by a painfully slow return. For even larger step sizes ($\eta_0 \geq 10^{-3}$), the updates blow up entirely, overshooting far outside the plotting window. These behaviors demonstrate the problematic dynamics of SGDs without safe start discussed earlier, that they either move slowly or exhibit unstable big jumps that again lead to prohibitively long iterations.






To close this discussion, Figure \ref{fig:teaser}c further illustrates the necessity of adaptive variance reduction to pair with safe start ($x_0=10$ in this figure). That is, both safe start and efficient adaptive variance reduction are needed to elicit optimization efficiency. In Figure \ref{fig:teaser}c, we show the trajectories of three algorithms: SGD with variance reduction (blue), standard SGD (red), and standard SGD using a $10,000\times$ larger mini-batch size compared to the first two (gray) (here ``standard SGD'' means not using variance reduction). Without variance reduction, the red trajectory explodes beyond the plot boundaries due to extreme variance, while the blue trajectory remains stably within the neighborhood of the true quantile. To make standard SGD achieve this same level of stability, the gray trajectory requires $10,000\times$ more samples in each iteration.

  
   


\begin{figure}[http]
    \centering
    \begin{minipage}{0.48\textwidth}
        \includegraphics[width=\linewidth]{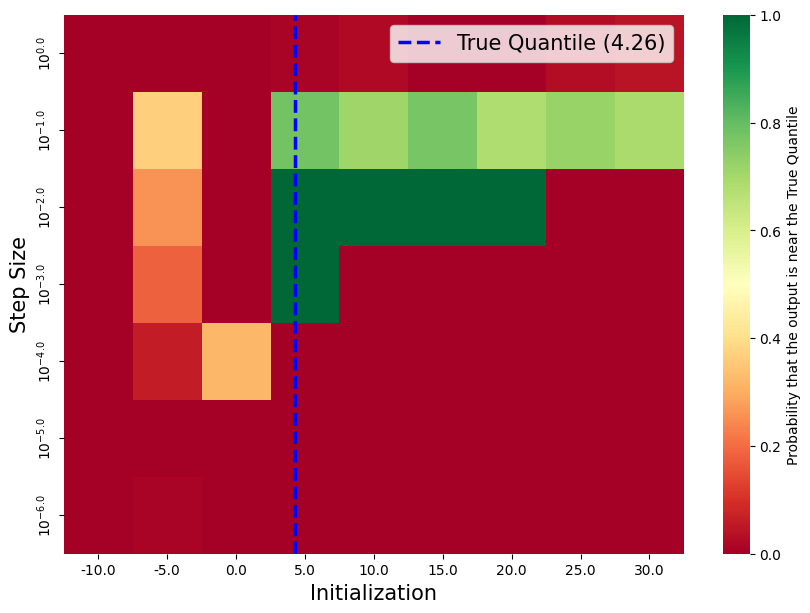}
        \centerline{(a)}
    \end{minipage}\hfill
    \begin{minipage}{0.48\textwidth}
        \includegraphics[width=\linewidth]{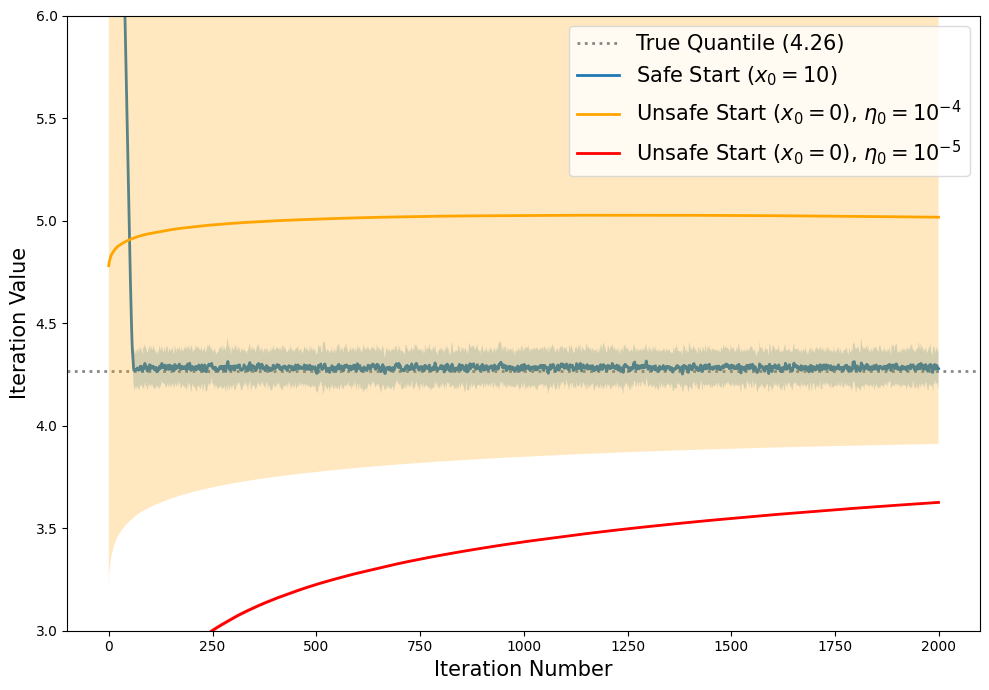}
        \centerline{(b)}
    \end{minipage}

    \vspace{0.5cm} 

    \begin{minipage}{0.48\textwidth}
        \includegraphics[width=\linewidth]{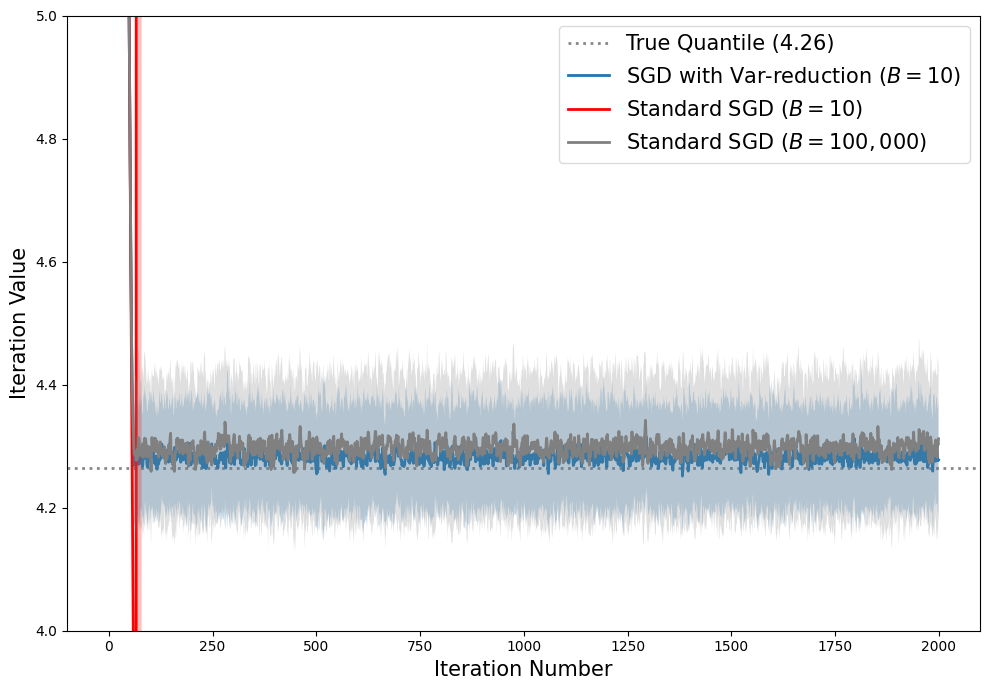}
        \centerline{(c)}
    \end{minipage}

 \caption{\textbf{SGD convergence across algorithmic configurations for extreme quantile estimation.} \textbf{(a)} Heatmap of convergence probability for different constant step size $\eta_0$ and initialization $x_0$. Vertical line signifies the true quantile $\vect{x}^*$. \textbf{(b)} Trajectories from a safe start ($x_0 = 10$) and an unsafe start ($x_0 = 0$). \textbf{(c)} Trajectories from a safe start ($x_0 = 10$) using SGD with variance reduction (blue), standard SGD (red), and another standard SGD with a $10,000\times$ larger mini-batch size than the previous two (gray). Here, $B$ denotes the mini-batch size per iteration.}
\label{fig:teaser}
\end{figure}

 
\subsection{Other Related Literature}
We close this introduction by briefly discussing other related recent works. First, \cite{deo2025achieving,deo2024importance}, similar to us, also study rare-event optimization. They use the notion of self-structuring importance sampling that exploits distributional self-similarity arising from the large deviations of optimization objectives, and use a ``transport'' map to execute the importance samplers. Unlike us, they do not need adaptive variance reduction or SGD and only require a single-stage sampling design. On the other hand, they rely on structural forms of the optimization objective, and as such their assumptions are arguably more restrictive than ours. \cite{blanchet2024efficient} study rare-event chance-constrained optimization under heavy tails. They focus on linear chance constraints and, like \cite{deo2025achieving,deo2024importance}, exploit linearity of the constraints and heavy-tail properties to derive efficient sampling schemes. \cite{tong2022-opt-under-rare} also investigate rare-event chance-constrained optimization, but under light-tailed settings. They propose the use of large deviations asymptotics to approximate the chance constraint and obtain tractable reformulations. Their approach has a spirit similar to \cite{glynn1996importance} who suggests inverting large deviations asymptotics to approximate extreme quantiles. Distinct from all these works, we focus on general objective functions and are motivated from the curse of circularity that advocates using SGDs to naturally embed adaptive variance reduction. More importantly, we investigate the sampling complexity in running SGD for rare-event optimization that provides, as far as we know, the first bridge between established notions in rare-event estimation and complexity notions in optimization, ultimately leading to the safe start concept. Lastly, we note that our framework is based on light-tailed large deviations, as opposed to heavy tail \citep{nair2022fundamentals,blanchet2012state} that would warrant a separate interesting line of work.

 
\subsection{Organization} The remainder of this paper is organized as follows. Section \ref{sec:setting} formalizes the rare-event optimization framework, detailing the assumptions on the setup, the rarity structure and the SGD procedures. Section \ref{sec:main} presents our main theoretical results and insights on safe start and its relations to adaptive variance reduction and step size selection. Section \ref{subsec:proof} discusses the key technical developments in achieving the main theoretical results. Section \ref{sec:experiment} presents numerical experiments that empirically validate our theoretical claims and demonstrate the practical implications in examples across quantitative finance and robust machine learning. Finally, the Appendix presents the proofs of all results in the paper.


\section{Problem Setting}
\label{sec:setting}

Building on \eqref{eq:start-formulation}, we introduce a ``rarity'' parameter $\lambda > 0$ that signifies the target rarity level (i.e., $\lambda\to\infty$ implies the likelihood of the rare event goes to 0). This parameter is a modeling artifact commonly used in the rare-event literature to facilitate large deviations analysis \citep{bucklew2004introduction,juneja2006rare,blanchet2012state}. With this, we consider the following formulation parameterized by the rarity level $\lambda$:
    \begin{equation*}\label{eq:main-formulation}
    (\ProblemRare) \qquad \min_{\mathbf{x} \in \mathcal{X} \subseteq \mathbb{R}^n} F(\mathbf{x}; \lambda) \quad \text{where} \quad F(\mathbf{x}; \lambda) = f(\mathbf{x}) + \gamma(\lambda) p(\mathbf{x})
    \end{equation*}


As in the introduction, we denote $\mathbf{x}^*(\lambda)$ as an optimal solution to formulation $\ProblemRare$. For convenience, we define $p^*(\lambda):=p(\mathbf{x}^*(\lambda))$, representing the extreme risk at this optimal solution when the rarity level is $\lambda$.

\subsection{Main Assumptions on Optimization Formulation}
To analyze the performance of our optimization algorithms as $\lambda$ scales, we make two assumptions. The first characterizes the asymptotic decay of the target risk, following a Large Deviation Principle (LDP) \citep{dembo2009large,budhiraja2019analysis} style.

\begin{assumption}[LDP of Target Risk]\label{assumption:ldp-optimization}
There exists a rate $I>0$ such that:
$$\lim_{\lambda\to\infty} \frac{1}{\lambda}\log p^*(\lambda) = -I$$
Moreover, the weighting parameter $\gamma(\lambda)$ is chosen to match this exponential scale, such that:
$$\limsup_{\lambda\to\infty} \frac{1}{\lambda}\log \gamma(\lambda) = I$$
\end{assumption}

 This assumption ensures a balance between average-case performance and extreme risk. If the exponential growth rate of the coefficient is strictly greater than $I$, the risk penalty $\gamma(\lambda)p(\mathbf{x}^*(\lambda))$ in $\mathbf P_\lambda$ would blow up exponentially fast. This would force the optimizer to ignore the average-case objective and output an overly conservative solution. Conversely, if the growth rate is strictly less than $I$, the risk penalty asymptotically vanishes, reducing the problem to minimizing the average-case loss. In a practically meaningful risk-averse setting, the exponential growth rate of $\gamma(\lambda)$ exactly matches the LDP decay rate $I$. This requires $\gamma(\lambda)$ to scale in the order of $1/p(\mathbf{x}^*(\lambda))$, corroborating our intuition from Section 1.

We will focus on gradient-based iterative algorithms to solve $\mathbf P_\lambda$. To this end, we make an assumption on the estimation quality of the gradient:
\begin{assumption}[Unbiased and Efficient Gradient Estimators]\label{assumption:second-moment}
At any rarity level $\lambda >0$, there exist unbiased estimators $G_f(\mathbf{x},\xi)$ and $G_p(\mathbf{x},\xi; \lambda)$ for $\nabla f(\mathbf{x})$ and $\nabla p(\mathbf{x})$, respectively. Furthermore, there exists a constant $\sigma_f \ge 0$ and a rarity-dependent quantity $\sigma_p(\lambda) \ge 0$ such that, for all $\mathbf{x}\in\mathcal{X}$, we have
$\mathbb{E}\left[\big\|G_f(\mathbf{x},\xi)-\nabla f(\mathbf{x})\big\|^2\right] \le \sigma_f^2$
and
\begin{equation}
    \mathbb{E}\left[\big\|G_p(\mathbf{x},\xi; \lambda)-\nabla p(\mathbf{x})\big\|^2\right] \le \sigma_p^2(\lambda)\|\nabla p(\mathbf{x})\|^2.\label{efficiency gradient}
\end{equation}
where $\sigma_p(\lambda)$ grows at most sub-exponentially with respect to the rarity level:
\begin{equation}
    \limsup_{\lambda \to \infty} \frac{1}{\lambda} \log \sigma_p(\lambda) \le 0. \label{variance-growth}
\end{equation}
\end{assumption}


In Assumption \ref{assumption:second-moment}, the unbiasedness of $G_f$ and $G_p$ signifies a first-order gradient oracle, and this typically can be obtained by infinitesimal perturbation analysis or pathwise differentiation \citep{ho1983infinitesimal,heidelberger1988convergence}, the likelihood ratio or score function method \citep{rubinstein1986score,reiman1989sensitivity,glynn1990likelihood}, or other approaches such as measure-valued or weak differentiation \citep{heidergott2008measure,heidergott2010gradient}. On the other hand, condition \eqref{efficiency gradient} is in line with the notion of weak efficiency, or synonymously logarithmic efficiency or asymptotic optimality in the rare-event literature \citep{l2010asymptotic,bucklew2004introduction}. Specifically, this means that in estimating a rare-event probability, say $p$, using an unbiased per-run estimator, say $Y$, the ratio between the logarithms of the per-run second moment and squared mean satisfies
$\lim_{p\to0}\log E[Y^2]/\log(p^2)=1$, or equivalently $\lim_{p\to0}\log(RE)/\log p\leq0$ where $RE$ is the relative error (note that $\lim_{p\to0}\log(RE)/\log p\geq0$ because of Jensen's inequality). Under Assumption \ref{assumption:ldp-optimization} where $p$ decays exponentially in $\lambda$, this means that the relative error grows at a rate slower than exponential in $\lambda$. In contrast, recall the introduction that naive Monte Carlo has a relative error given by $\sqrt{(1-p)/p}$, which implies the relative error blows up to $\infty$ exponentially in $\lambda$. That is, weak efficiency suppresses the growth rate of relative error from exponential in naive Monte Carlo to sub-exponential and hence significantly smaller. Condition \eqref{efficiency gradient} states the same efficiency notion for the gradient instead of rare-event probabilities, using the $L_2$-norm as a natural multivariate extension of the square. Generally, if we use importance sampling as the variance reduction method (as commonly adopted), then the efficiency for the rare-event probability estimator can be translated into that of the related gradient estimator via the likelihood ratio, though this analysis will need to be conducted case-by-case. Here our focus is on the algorithmic implications, assuming we already have in hand efficient gradient estimators.

\subsection{Configurations of Stochastic Gradient Descent}


We use SGD to solve \ProblemRare, not only because of its versatility but also because it provides a natural platform to adaptively embed variance reduction in the gradient estimator. At each iteration $t$, we draw a mini-batch of $B \ge 1$ i.i.d. samples, denoted as $\xi_{t, 1:B} = \{\xi_{t,1}, \dots, \xi_{t,B}\}$. We construct the mini-batch gradient estimator for the objective $F(x; \lambda)$  as:
$$\bar{G}(\mathbf{x}_t, \xi_{t, 1:B}) = \frac{1}{B} \sum_{i=1}^{B} \left( G_f(\mathbf{x}_t, \xi_{t,i}) + \gamma(\lambda)G_p(\mathbf{x}_t, \xi_{t,i}) \right)$$
Given the current iterate $\mathbf{x}_t \in \mathcal{X}$, the projected SGD update rule is:
$$\mathbf{x}_{t+1} = \Pi_{\mathcal{X}}\left(\mathbf{x}_t - \eta_t \bar{G}(\mathbf{x}_t, \xi_{t, 1:B})\right)$$
where $\eta_t > 0$ is the step size and $\Pi_{\mathcal{X}}$ denotes the Euclidean projection onto the convex feasible domain $\mathcal{X}$.

We denote the configuration of an SGD as $(\mathcal{A}, \mathbf{x}_0)$, where $\mathbf{x}_0$ is the initial solution and $\mathcal{A}$ denotes all other SGD specifications. This notation is convenient to highlight the role of $\mathbf{x}_0$ later. Here, $\mathcal{A}$ prescribes:
\begin{enumerate}
\item A \emph{non-adaptive} step-size sequence $\{\eta_t\}_{t\ge 0}$ where step sizes are deterministically decided in advance and not dependent on the observations. Specifically, we consider the polynomially decaying step size $\eta_t = \eta_0(1+t)^{-\alpha}$ for a decay rate $\alpha \in [0,1)$. Note that setting $\alpha=0$ yields constant step size.
\item A constant mini-batch size $B \ge 1$.
\item An output method after iteration $T$ that gives either the last-iterate $\mathbf{x}_T$ or a step-size-weighted average $\bar{\mathbf{x}}_T=\frac{\sum_{t=0}^{T-1} \eta_t \mathbf{x}_{t+1}}{\sum_{t=0}^{T-1} \eta_t}.$

\end{enumerate}
For convenience, we denote $\mathbb{A}_{\alpha,\mathrm{avg}}$ and $\mathbb{A}_{\alpha,\mathrm{last}}$ as SGD schemes that use $\alpha$ in the step size specification, and average-iterate and last-iterate outputs respectively. Note that for $T$ gradient updates, the sampling effort of the scheme $\mathcal{A}$ is exactly $T \cdot B$.




\subsection{Rarity-Aware Sampling Complexity, Efficiency and Safe Start}

We define the sampling complexity for a rare-event optimization problem \ProblemRare\ explicitly in the rarity parameter $\lambda$, as follows:
\begin{definition}[Rarity-Aware Sampling Complexity]\label{def:rarity-complexity}
For an SGD configuration $(\mathcal{A}, \mathbf{x}_0)$, its sampling complexity $T^{\varepsilon,\kappa}_{\lambda}(\mathcal{A}, \mathbf{x}_0)$ is the minimum number of stochastic gradient samples required at rarity level $\lambda$ to produce an iterate $\hat{\mathbf{x}}$ satisfying:
$$F(\hat{\mathbf{x}}; \lambda) \le (1+\varepsilon)\cdot F(\mathbf{x}^*(\lambda); \lambda)$$
with probability at least $1-\kappa$.
\end{definition}
In other words, our success criterion is defined as reaching a relative accuracy $\varepsilon>0$ with confidence $1-\kappa\in (0,1)$, in terms of the attained objective value. Next, we describe algorithmic efficiency:
\begin{definition}[Efficient Configuration]\label{def:efficient-config}
A sequence of configurations $\{(\mathcal{A}(\lambda), \mathbf{x}_0(\lambda))\}_{\lambda>0}$ is called \textit{efficient} if, for any given $\varepsilon>0$ and $1-\kappa\in(0,1)$, its sampling complexity grows sub-exponentially in $\lambda$:
$$\limsup_{\lambda\to\infty}\frac{1}{\lambda} \log T^{\varepsilon,\kappa}_{\lambda}\big(\mathcal{A}(\lambda),\mathbf{x}_0(\lambda)\big) \le 0$$
\end{definition}
Analogous to Assumption \ref{assumption:second-moment}, this definition is the generalization of the weak efficiency notion in rare-event estimation \citep{l2010asymptotic} to optimization, where now the sampling complexity refers to the number of samples needed to reach an $\varepsilon$-nearly-optimal solution. Note that the sampling complexity is dependent on $\varepsilon$ and $\kappa$, which we hide for ease of presentation and their values will be apparent in the considered contexts. Now, we formally introduce safe start:
\begin{assumption}[Safe Start]\label{def:safe-start}
Under Assumption \ref{assumption:ldp-optimization} with decay rate $I$, a sequence $\{\mathbf{x}_0(\lambda)\}_{\lambda>0}$ is a sequence of \emph{safe starts}, namely it satisfies
$$\limsup_{\lambda\to\infty}\frac{1}{\lambda}\log p\big(\mathbf{x}_0(\lambda)\big)\le -I$$
\end{assumption}
That is, safe start means that the risk level of the SGD initialization is upper bounded by that of the optimal solution, in exponential scale.

\subsection{Additional Regularity Assumptions}
\label{appendix-sec:additional_assumption}

In addition to Assumptions \ref{assumption:ldp-optimization}--\ref{def:safe-start}, we also impose additional regularity assumptions on the average-case loss function $f$, the  risk term $p$, and the sequence of initializations $\prtc{\mathbf{x}_0(\lambda)}_{\lambda >0}.$

\subsubsection{Regularity Assumptions on the Average-Case Loss.}
We first introduce the following assumption that is standard in stochastic optimization.

\begin{assumption}[Smoothness and Convexity of the Average-Case Loss]\label{assumption:smooth-connvex}
$\mathcal{X}\subseteq \mathbb{R}^n$ is a convex, closed, non-empty domain. Let the average-case loss $f$ be twice continuously differentiable, convex, and $L_f$-smooth, i.e., $\|\nabla^2 f(\mathbf{x})\| \le L_f$ for all $\mathbf{x} \in \mathcal{X}$, where $L_f > 0$. 
\end{assumption}


Convexity and Lipschitz smoothness are two common, well-understood assumptions in first-order stochastic optimization \citep{nemirovski2009robust, beck2017first}. Geometrically, convexity ensures that any local minimum is a global minimum, while Lipschitz smoothness controls the rate at which the gradient changes. We note that this smoothness is imposed only on the average-case loss term. On the other hand, the sensitivity challenge in rare-event optimization is driven by the extreme-risk term, which we will discuss soon.

We next assume a coercivity condition on the average-case loss $f$, expressed through the boundedness of its sublevel set.

\begin{assumption}[Sublevel-Boundedness]\label{assumption:level-bounded}
There exist constants $\epsilon_f > 0$ and $d_f > 0$ such that the set $\{\mathbf{x} \in \mathcal{X} \mid f(\mathbf{x}) \le \epsilon_f\}$ is non-empty and has diameter at most $d_f$; that is, any $\mathbf{x}_1, \mathbf{x}_2 \in \mathcal{X}$ with $f(\mathbf{x}_1) \le \epsilon_f$ and $f(\mathbf{x}_2) \le \epsilon_f$ satisfy $\|\mathbf{x}_1 - \mathbf{x}_2\| \le d_f$.
\end{assumption}

Without coercivity, $f$ may fail to attain a minimizer, leaving the optimization ill-posed. This condition holds automatically when the feasible domain $\mathcal{X}$ is bounded, or whenever $f$ is coercive, such as the least square loss.

By the boundedness of the $\epsilon_f$-sublevel set and the continuity of $f$, the Weierstrass Extreme Value Theorem guarantees that $f$ attains a minimizer $\tilde{\mathbf{x}}^*$. Without loss of generality, we assume the objective value at this minimum is zero, which renders $f$ non-negative.

\begin{assumption}[Zero Minimum of the Average-Case Loss]\label{assumption:lower-bounded}
The objective value of $f$ at the minimum is zero, i.e., $f(\tilde{\mathbf{x}}^*) = 0$.
\end{assumption}

\subsubsection{Regularity Assumptions on the Extreme Risk.}

We introduce an assumption on the extreme risk that is central to our analysis of rare-event optimization efficiency. Unlike the global smoothness of Assumption~\ref{assumption:smooth-connvex}, we bound the magnitude of the extreme-risk gradient and Hessian \emph{relative to the extreme risk itself}. We do not impose a global Lipschitz bound on $p$, which would make the smoothness constant of $F$ scale with the exponentially large penalty $\gamma(\lambda)$.

\begin{assumption}[Relative Gradient and Hessian Bounds]\label{assumption:scale-grad-hessian}
Let the extreme risk $p: \mathcal{X} \to [0,\infty)$ be twice continuously differentiable
on $\mathcal{X}$. There exist constants $L_{p,1}, L_{p,2} > 0$ and exponents $\beta_1,\beta_2 \ge 0$
such that for all $\mathbf{x} \in \mathcal{X}$:
\begin{equation*}
\|\nabla p(\mathbf{x})\| \le L_{p,1} \prtr{1+\norm{\mathbf{x} - \tilde{\mathbf{x}}^*}^{\beta_1}} p(\mathbf{x})
\quad \text{and} \quad
\|\nabla^2 p(\mathbf{x})\| \le L_{p,2} \prtr{1+\norm{\mathbf{x} - \tilde{\mathbf{x}}^*}^{\beta_2}} p(\mathbf{x}).
\end{equation*}
\end{assumption}

This assumption captures the typical behavior of functions exponentially dependent on $\mathbf x$. Consider an extreme risk of the form $p(\mathbf{x}) = e^{\psi(\mathbf{x})}$ for some smooth $\psi$. Its derivatives are
\begin{equation*}
\nabla p = p\,\nabla\psi, \qquad
\nabla^2 p = p\prtr{\nabla\psi\,\nabla\psi^\top + \nabla^2\psi},
\end{equation*}
so that the norms
\begin{equation*}
\|\nabla p\| = p\,\|\nabla\psi\|, \qquad
\|\nabla^2 p\| \le p\prtr{\|\nabla\psi\|^2 + \|\nabla^2\psi\|}
\end{equation*}
are controlled by the risk $p$ itself, up to factors depending only on the derivatives of $\psi$. We assume these derivatives grow at most polynomially in $\|\mathbf{x}-\tilde{\mathbf{x}}^*\|$ as in  Assumption~\ref{assumption:scale-grad-hessian}. 

We recall the CVaR risk term $p(x) = \mathbb{E}[(Z - x)_+]$ from
Section~\ref{sec:intro} for a standard normal $Z$, which illustrates this assumption.
Writing $\phi$ and $\overline{\Phi}$ for the standard normal density and tail
distribution function, we have
\begin{equation*}
p(x) = \phi(x) - x\,\overline{\Phi}(x), \qquad
p'(x) = -\overline{\Phi}(x), \qquad
p''(x) = \phi(x).
\end{equation*}
By the Mills ratio expansion, $\overline{\Phi}(x) \sim \phi(x)/x$ and
$p(x) \sim \phi(x)/x^2$ as $x\to\infty$. Therefore the relative gradient and Hessian
satisfy
\begin{equation*}
\frac{|p'(x)|}{p(x)} = \frac{\overline{\Phi}(x)}{p(x)} \sim x, \qquad
\frac{|p''(x)|}{p(x)} = \frac{\phi(x)}{p(x)} \sim x^2,
\end{equation*}
so the relative gradient and Hessian norms grow linearly and quadratically in $x$,
i.e., $\beta_1 = 1$ and $\beta_2 = 2$, aligning with
Assumption~\ref{assumption:scale-grad-hessian}.

Finally, we assume the combined objective of the average-case loss and the extreme risk remains convex, guaranteeing global optimality. The CVaR formulation above is again an example, as its combined objective is convex in $x$.

\begin{assumption}[Convexity of the Combined Objective]\label{assumption:connvex}
For any rarity level $\lambda > 0$, the combined objective $F(\cdot; \lambda)$ is convex over the convex domain $\mathcal{X}$.
\end{assumption}

\subsubsection{Regularity Assumption on Initializations.}

To isolate the sensitivity challenge specific to rare-event optimization, we rule out
the pathological case where an initialization sits exponentially far from the optimal
solution.

\begin{assumption}[Sub-Exponential Initial Distance]\label{assumption:init-distance}
The sequence of initializations $\{\mathbf{x}_0(\lambda)\}_{\lambda>0}$ has a sub-exponential initial distance, meaning that its distance to the average-case minimizer $\tilde{\mathbf{x}}^*$ grows at most sub-exponentially:
\[
\limsup_{\lambda\to\infty}\frac{1}{\lambda}\log \|\mathbf{x}_0(\lambda)-\tilde{\mathbf{x}}^*\| \le 0.
\]
\end{assumption}

We consider the distance to the average-case minimizer $\tilde{\mathbf{x}}^*$ rather than the optimum $\mathbf{x}^*(\lambda)$ to the rare-event optimization problem $\ProblemRare$. This is because $\tilde{\mathbf{x}}^*$ is fixed with respect to the rarity level $\lambda$, and thus the condition above can be verified more readily by merely solving the average-case-only formulation. 
Moreover, combining with other assumptions, we can show in Lemma \ref{lemma:subexponential-constants} that the sequence of initializations satisfying this assumption also has a sub-exponential distance to the optimal solution $\mathbf{x}^*(\lambda)$.



\section{Main Results: Safe Start with Variance Reduction is Necessary and Sufficient for Rare-Event Optimization Efficiency}
\label{sec:main}


We present our main results characterizing the conditions for optimization efficiency under extreme risk. We first establish that, given efficient gradient estimators, a safe start is sufficient to guarantee sub-exponential sampling complexity across standard SGD algorithms.

\begin{theorem}[Sufficiency of Safe Start to Attain Efficiency]\label{thm:main}
For a sequence of optimization instances $\{\ProblemRare\}_{\lambda>0}$ with $\mathcal{X}=\mathbb{R}^n$, suppose Assumptions \ref{assumption:ldp-optimization} to \ref{assumption:init-distance} hold with a sequence of safe starts $\{\mathbf{x}_0(\lambda)\}_{\lambda>0}$. Then, there exists a sequence of SGD schemes $\{\mathcal{A}(\lambda)\}_{\lambda>0}$, drawn exclusively from either $\mathbb{A}_{\alpha,\mathrm{last}}$ or $\mathbb{A}_{\alpha,\mathrm{avg}}$ for any fixed $\alpha \in [0,1)$, such that the sequence of optimization configurations $\{(\mathcal{A}(\lambda), \mathbf{x}_0(\lambda))\}_{\lambda>0}$ is efficient in the sense of Definition \ref{def:efficient-config}.
\end{theorem}
Theorem \ref{thm:main} stipulates that, if we start safe and have in place an efficient gradient estimator, then we can tune SGD, either average-iterate or last-iterate, such that the algorithm is efficient. Moreover, the step size can be constant or polynomially decaying, and specified in advance. The existence of the efficient schemes in Theorem \ref{thm:main} is constructive: for any rarity level $\lambda$, we can find a mini-batch size $B_\lambda$, an iteration horizon $T_\lambda$, and a step size constant $\eta_0$ to achieve a target relative accuracy $\varepsilon$ with a probability at least $1-\kappa$, for any prescribed $\varepsilon$ and $1-\kappa$, using sub-exponential sampling cost. 


The main argument of Theorem \ref{thm:main} relies on the geometric insight that safe start places the SGD to a well-behaved initial \emph{sublevel set}, defined in our later technical developments in Section \ref{subsubsec:safe-landscape} as $S(x_0, \lambda, c) := \{x \in \mathcal{X} \mid F(x; \lambda) \le c \cdot F(x_0; \lambda)\}$ for a constant $c \ge 1$. Both the local Lipschitz constant and the variance of the gradient estimator are controlled within $S$. These desirable properties allow the SGD to stably traverse to the optimal solution without triggering an exponential explosion in the required number of iterations. These behaviors are explained in detail for $\mathbb{A}_{\alpha,\mathrm{last}}$ and $\mathbb{A}_{\alpha,\mathrm{avg}}$ respectively through 
Proposition~\ref{lemma:last-iterate-sufficiency} in Section~\ref{subsubsec:last-iterate} and Proposition~\ref{lemma:avg-iterate-sufficiency} in
Section~\ref{subsubsec:average-iterate} in the sequel. Conversely, if we do not apply safe start, then we could encounter poor complexity due to the ultra-sensitive objective landscape, as shown next. 



\begin{theorem}[Necessity of Safe Start to Attain Efficiency]\label{thm:unsafe}
Suppose Assumptions \ref{assumption:ldp-optimization}--\ref{assumption:second-moment} and \ref{assumption:smooth-connvex}--\ref{assumption:connvex} hold. There exists a relative accuracy $\epsilon>0,$ a confidence level $1-\kappa \in(0,1)$, a sequence of optimization instances $\{\ProblemRare\}_{\lambda>0}$, and a sequence of unsafe starts $\{\mathbf{x}_0(\lambda)\}_{\lambda>0}$ satisfying Assumption \ref{assumption:init-distance}, such that for any sequence of (projected) SGD schemes $\{\mathcal{A}(\lambda)\}_{\lambda>0}$ chosen from the universal class of all monotonically diminishing step sizes ($\eta_t \ge \eta_{t+1}$ for all $t \ge 0$), the sampling complexity grows exponentially:
\[
  \limsup_{\lambda\to\infty}\frac{1}{\lambda} \log T^{\varepsilon,\kappa}_{\lambda}\big(\mathcal{A}(\lambda),\mathbf{x}_0(\lambda)\big) > 0.
\]
\end{theorem}

Theorem \ref{thm:unsafe} shows that, if we do not use safe start, then it is possible that no matter what SGD configurations we use, we run into exponential sampling complexity. Note that these configurations can include the broader classes of SGD schemes with monotonically diminishing step sizes, not only polynomially decaying.

Here, with an unsafe start, the initial gradient's magnitude is exponentially large. Consequentially,  if the step size is not exponentially small, an SGD update can cause an overshoot, with the iterate thrown exponentially away from the optimum. Returning from this exponential distance requires exponential iterations. Conversely, if the step size is exponentially small to offset with this large gradient, the algorithm progresses at a slow speed that takes exponentially many iterations to reach the optimum. Note that this dichotomic failure of convergence occurs even in a deterministic setting with zero gradient noise. This highlights the root of the issue as coming from the extreme sensitivity of the landscape geometry, rather than the noise incurred in gradient estimation. Proposition~\ref{prop:unsafe-instance} in Section \ref{subsec:nec-safe-start} provides an explicit instance that captures precisely these dynamics underlying the argument for Theorem~\ref{thm:unsafe}.


 



Finally, we show that, in addition to safe start, efficient variance-reduced gradient estimator is also crucial in guaranteeing efficiency:


\begin{theorem}[Necessity of Efficient Gradient Estimators]\label{thm:highvariance}
Suppose Assumptions \ref{assumption:ldp-optimization} and \ref{assumption:smooth-connvex}--\ref{assumption:connvex} hold, and Assumption \ref{assumption:second-moment} is relaxed so that the gradient estimator satisfies only a standard absolute variance bound $\mathbb{E}[\|G_p(\mathbf{x},\xi;\lambda)-\nabla p(\mathbf{x})\|^2] \le \mathcal{O}(\|\nabla p(\mathbf{x})\|)$. There exists a relative accuracy $\epsilon>0,$ a confidence level $1-\kappa \in(0,1)$, and  a sequence of optimization instances $\{\ProblemRare\}_{\lambda>0}$ such that, for any sequence of (projected) SGD schemes $\{\mathcal{A}(\lambda)\}_{\lambda>0}$ chosen from $\mathbb{A}_{0,\mathrm{last}}$, there exists a sequence of safe starts $\{\mathbf{x}_0(\lambda)\}_{\lambda>0}$ satisfying Assumptions \ref{def:safe-start} and \ref{assumption:init-distance} under which the sampling complexity grows exponentially:
\[
  \limsup_{\lambda\to\infty}\frac{1}{\lambda} \log T^{\varepsilon,\kappa}_{\lambda}\big(\mathcal{A}(\lambda),\mathbf{x}_0(\lambda)\big) > 0.
\]
\end{theorem}

Theorem \ref{thm:highvariance} shows that, even if we enforce safe start, without efficient gradient estimator, we can run into problem instances with poor sampling complexity no matter how we tune the SGD. This failure arises from combined effects of ultra-sensitivity of the objective and the exploding gradient noise. That is, step sizes that are too large cause SGD trajectories to overshoot, while those that are too small traverse slowly. Even if an intermediate step size is selected and initialized from a safe start, the high-variance gradient estimator prevents the convergence of SGD, unless an exponentially large batch size is used to mitigate this noise. In all these cases, the sampling requirements become exponential. This necessity of efficient gradient estimation can be viewed as the sampling-complexity counterpart of the asymptotic study in \citet{he2024adaptive} that argues the need of running adaptive variance reduction for rare-event optimization. At a high level, Theorem \ref{thm:highvariance} confirms a similar requirement as \citet{he2024adaptive} at the level of finite-time performance instead of only asymptotic convergence. Like Theorem~\ref{thm:unsafe}, we provide an explicit instance to elicit these dynamics that lead to Theorem~\ref{thm:highvariance} in Proposition \ref{prop:highvariance-instance} in Section \ref{subsec:nec-safe-start}.

In summary, Theorems~\ref{thm:unsafe} and~\ref{thm:highvariance} isolate two challenges for SGD under extreme risk: the ultra-sensitivity of the rare-event objective and the exploding relative variance of its gradient estimator. Theorem~\ref{thm:main} shows these challenges can be addressed, respectively, by a safe start and adaptive variance reduction. The analysis establishing these results is developed in the following section.



\section{Main Technical Development via Joint Containment and Descent Analysis}
\label{subsec:proof}

We present the technical developments behind the theorems above. The beginning analysis of Theorem~\ref{thm:main} rests on the initial sublevel set anchored at the initialization and its regularity properties, such as the local Lipschitz smoothness constant. These properties are derived in Section~\ref{subsubsec:safe-landscape}, where we show that a safe start keeps them controlled. Moreover, these properties continue to hold along the SGD trajectory, but only locally while an iterate lies in the initial sublevel set. At the same time, these properties are needed to guarantee proper descent of the SGD. These in turn create the need of a joint containment–descent analysis to achieve SGD convergence. We detail this development for both last-iterate and average-iterate SGDs, in Sections~\ref{subsubsec:last-iterate} and~\ref{subsubsec:average-iterate} respectively. Then, we present the optimization instances that establish Theorems~\ref{thm:unsafe} and~\ref{thm:highvariance} in Section~\ref{subsec:nec-safe-start}.


\subsection{Initial Sublevel Set and Its Regularity Properties}
\label{subsubsec:safe-landscape}

Our analysis begins by looking at the sublevel set of the combined
objective, anchored at the initialization $\mathbf{x}_0$. We define such sublevel set as follows:
\begin{definition}[Initial Sublevel Set]
For a given rarity level $\lambda > 0$, an initial solution $x_0 \in \mathcal{X}$, and a positive constant $c \ge 1$, we define the initial sublevel set $S(\mathbf{x}_0, \lambda, c)$ as:
$$S(\mathbf{x}_0, \lambda, c) := \{\mathbf{x} \in \mathcal{X} \mid F(\mathbf{x}; \lambda) \le c F(\mathbf{x}_0; \lambda)\}$$
\end{definition}

Under the assumptions in Section \ref{appendix-sec:additional_assumption}, we can derive three regularity conditions of
this sublevel set: boundedness of the set, Lipschitz smoothness of the
objective, and boundedness of the gradient estimation noise. 




\begin{lemma}[Regularity of the Initial Sublevel Set]
\label{lem:safe-region-regularity}
Fix a rarity level $\lambda > 0$, a constant $c \ge 1$, and an initialization
$\mathbf{x}_0 \in \mathcal{X}$, and let $S = S(\mathbf{x}_0,\lambda,c)$ be the initial sublevel set.
Under Assumptions~\ref{assumption:second-moment}, \ref{assumption:smooth-connvex},
\ref{assumption:level-bounded}, \ref{assumption:lower-bounded} and
\ref{assumption:scale-grad-hessian}, the following hold for all $\mathbf{x} \in S$ where the constants $B_S$, $L_S$, and
$\sigma_S$ are defined below. All other constants are as defined in the referenced
assumptions, and we denote $D_S := d_f + \dfrac{c\,d_f}{\epsilon_f}\,F(\mathbf{x}_0;\lambda)$.
\begin{enumerate}
  \item \emph{(Bounded radius to optimal solution)} 
  The distance to the optimal solution
  $\mathbf{x}^*(\lambda)$ is bounded:
  \[
    \|\mathbf{x} - \mathbf{x}^*(\lambda)\| \le B_S,
    \qquad
    B_S := 2 D_S.
  \]
  \item \emph{(Local smoothness)} The objective $F(\cdot\,;\lambda)$ is
  $L_S$-smooth over $S$, i.e.,
  \[
    \bigl\|\nabla^2 F(\mathbf{x};\lambda)\bigr\| \le L_S,
    \qquad
    L_S := L_f + c\,L_{p,2}\,\bigl(1 + D_S^{\beta_2}\bigr)\,F(\mathbf{x}_0;\lambda).
  \]
  \item \emph{(Bounded gradient noise)} The per-run gradient estimator
  \[
    G(\mathbf{x},\xi) := G_f(\mathbf{x},\xi) + \gamma(\lambda)\,G_p(\mathbf{x},\xi;\lambda)
  \]
  is conditionally unbiased for $\nabla F(\mathbf{x};\lambda)$ and satisfies
  \[
    \mathbb{E}\bigl[\|G(\mathbf{x},\xi) - \nabla F(\mathbf{x};\lambda)\|^2\bigr] \le \sigma_S^2,
    \qquad
    \sigma_S := \sigma_f + c\,L_{p,1}\,\sigma_p(\lambda)\,\bigl(1 + D_S^{\beta_1}\bigr)\,F(\mathbf{x}_0;\lambda).
  \]
\end{enumerate}
\end{lemma}

The bounded radius to optimal solution $B_S$ are derived from the sublevel-boundedness (Assumption~\ref{assumption:level-bounded})
together with the convexity of the average-case loss
(Assumption~\ref{assumption:smooth-connvex}). The bounded smoothness constant $L_S$
follows from the smoothness of the average-case loss
(Assumption~\ref{assumption:smooth-connvex}) with the relative Hessian bound on
the extreme risk (Assumption~\ref{assumption:scale-grad-hessian}). The bounded gradient-noise
bound $\sigma_S$ are from the efficiency of the adaptive variance reduction
(Assumption~\ref{assumption:second-moment}) with the relative gradient bound on
the extreme risk (Assumption~\ref{assumption:scale-grad-hessian}). 



If the initial extreme risk $p(\mathbf{x}_0)$ to be small, then it makes the combined objective value $F(\mathbf{x}_0;\lambda)$ and all three initial-sublevel-set constants $B_S$, $L_S$ and $\sigma_S$ in Lemma \ref{lem:safe-region-regularity}
small at once. That is, with safe starts, these terms grow at sub-exponential rates.

\begin{restatable}[Sub-Exponential Bounds on Initial-Sublevel-Set Constants via Safe Start]{lemma}{subexponentialconstants}
\label{lemma:subexponential-constants}
Suppose Assumptions~\ref{assumption:ldp-optimization} and
\ref{assumption:second-moment} hold for the adaptive variance reduction, and
suppose Assumption~\ref{def:safe-start} holds for a sequence of safe starts
$\prtc{\mathbf{x}_0(\lambda)}_{\lambda >0}$. Under the additional regularity
Assumptions~\ref{assumption:smooth-connvex}--\ref{assumption:init-distance}, the
initial objective value $F(\mathbf{x}_0(\lambda); \lambda)$ and the initial-sublevel-set constants $B_S(\lambda)$, $L_S(\lambda)$, and $\sigma_S(\lambda)$ (as defined in Lemma \ref{lem:safe-region-regularity}) grow sub-exponentially in $\lambda$.
\end{restatable}

The subsequent convergence analyses express the sampling complexity of SGD
through these constants. Because a safe start keeps each of them sub-exponential,
the resulting sampling complexity becomes sub-exponential as well, making SGD with
a safe start an efficient configuration.

\subsection{Convergence of Last-Iterate SGD}
\label{subsubsec:last-iterate}
We first analyze SGD in the last-iterate scheme class $\mathbb{A}_{\alpha,\mathrm{last}}$ (for $\alpha \in [0,1)$), which outputs the last iterate under a decaying step-size schedule $\eta_t \propto (t+1)^{-\alpha}$ together with mini-batch variance reduction. To this end, Lemmas \ref{lem:safe-region-regularity} and \ref{lemma:subexponential-constants} provide \emph{local} regularity guarantees, i.e., within the initial sublevel set $S$. That is, any smoothness and variance bounds that allow to drive a descent step hold only when the iterate lies in the safe region $S$. This creates a coupling we must resolve: if the current iterate descends, then, by definition of a sublevel set, it remains within the set. However, deriving that descent property in the first place requires the update to already lie within the safe region, so that we can leverage the local smoothness. Containment and descent are thus interdependent and must be established jointly at each step. Our analysis does exactly this, as we first show that a single controlled step both descends and remains in $S$ with high probability, and then extend this guarantee along the trajectory.


The outline of the proof is as follows. In Section~\ref{subsubsec:fixed-rarity-last}, we fix a rarity level and establish the joint containment and descent properties of SGD per iteration. Then, in Section~\ref{subsubsec:asymptotic-last}, we present specific configurations with  step sizes, iteration count, and mini-batch size for schemes in $\mathbb{A}_{\alpha,\mathrm{last}}$ such that with the complexity of such a configuration grows only sub-exponentially in the rarity level.

\subsubsection{Convergence at a Fixed Rarity Level}
\label{subsubsec:fixed-rarity-last}

Throughout this subsection, we fix the rarity level $\lambda$ and simplify notation by writing $F(\mathbf{x}) := F(\mathbf{x}; \lambda)$ and $F^* := F(\mathbf{x}^*(\lambda); \lambda)$. Let $\mathbf{x}_0$ be the initial solution, and let $S = S(\mathbf{x}_0, \lambda, 2)$ be the initial sublevel set. We treat the constants $L_S$, $\sigma_S^2$, $B_S$, and $d_f$ from Section~\ref{subsubsec:safe-landscape} as fixed at this level; their dependence on $\lambda$ is carried through only at the end, when we read off the sampling complexity in terms of $\lambda$.

Given that an iterate is already inside the safe region, we first show that a properly chosen step size together with controlled gradient noise guarantees, with high probability, both that the next iterate remains in $S$ and that the objective descends by at least a fixed amount.

\begin{restatable}[Safe Descent for Last-Iterate SGD]{lemma}{safedescentlast}
\label{lemma:descent-lemma-last}
Suppose Assumptions \ref{assumption:second-moment}--\ref{assumption:connvex} hold. Consider SGD with step sizes $\prtc{\eta_t}_{t=0}^{\infty}$ where $\eta_0 \le \frac{1}{L_S}$, and denote $\mathbf{x}_t$ as the $t$-th iterate. If $F(\mathbf{x}_0) > (1+\varepsilon)F^*$, $\mathbf{x}_t \in \text{int}(S)$, and the variance satisfies:
\begin{equation*}
    \sigma_S^2 \le \kappa \min \prtc{ \frac{\varepsilon F^* L_S}{4}, d_f^2 L_S^2, \frac{\varepsilon^2 (F^*)^2}{16 B_S^2} }
\end{equation*}
then with probability at least $1-\kappa$, $\mathbf{x}_{t+1} \in \text{int}(S)$. Furthermore, either $F(\mathbf{x}_{t+1}) \le (1+\varepsilon)F^*$ or:
\begin{equation*}
    F(\mathbf{x}_{t+1}) \le F(\mathbf{x}_t) - \frac{3\eta_t \varepsilon^2 (F^*)^2}{32 B_S^2}
\end{equation*}
\end{restatable}

That is, a single controlled step ensures both next-step containment and descent. This also lets us analyze descent along the entire trajectory: once a union bound secures containment at every next step, the per-step descent guarantees apply throughout the path. A sufficiently large step-size sum then accumulates these decrements to drive the objective below the target level, yielding a sufficient condition for last-iterate convergence.

\begin{restatable}[Sufficient Conditions for Last-Iterate SGD Convergence]{lemma}{lastiterateconvergence}
\label{lemma:upperbound-iteration}
Suppose Assumptions \ref{assumption:second-moment}--\ref{assumption:connvex} hold. We perform SGD for $T$ iterations with step sizes $\prtc{\eta_t}_{t=0}^{T-1}$ where $\eta_0 \le \frac{1}{2L_S}$. Assume the sum of step sizes satisfies:
\begin{equation*}
    \sum_{t=0}^{T-1}\eta_t > \frac{32 B_S^2 F(\mathbf{x}_0)}{3 \varepsilon^2 (F^*)^2}
\end{equation*}
If $F(\mathbf{x}_0) > (1+\varepsilon)F^*$ and the variance satisfies:
\begin{equation*}
    \sigma_S^2 \le \frac{\kappa}{T} \min \prtc{ \frac{\varepsilon F^* L_S}{4}, d_f^2 L_S^2, \frac{\varepsilon^2 (F^*)^2}{16 B_S^2} }
\end{equation*}
then $F(\mathbf{x}_{T}) \le (1+\varepsilon)F^*$ with probability at least $1-\kappa$.
\end{restatable}

That is, if we can tune the step sizes and mini-batch size to satisfy the above conditions, we reach the target accuracy for a fixed rarity level. We specify such a configuration in what follows.

\subsubsection{Efficiency of Last-Iterate SGD}
\label{subsubsec:asymptotic-last}

We now construct explicit SGD configurations as functions of $\lambda$ and apply the sub-exponential bounds of Lemma~\ref{lemma:subexponential-constants} to establish efficiency. Specifically, we show that these configurations in $\mathbb{A}_{\alpha,\mathrm{last}}$ satisfy the sufficient conditions of Lemma~\ref{lemma:upperbound-iteration}, reaching the target accuracy with high probability at every rarity level.

\begin{restatable}[Rarity-Scaled Configuration for Last-Iterate SGD]{lemma}{lastiterateconfiguration}
\label{lemma:last-iterate-configuration}
Under Assumptions \ref{assumption:ldp-optimization}--\ref{assumption:init-distance}, choose a decay rate $\alpha \in [0,1)$. Consider a sequence of safe starts $\prtc{\mathbf{x}_0(\lambda)}_{\lambda>0}$ where $F(\mathbf{x}_0(\lambda); \lambda) > (1+\varepsilon)F^*(\lambda)$. For each $\lambda$, we select step sizes $\eta_t(\lambda) = \frac{\eta_0(\lambda)}{(t+1)^\alpha}$ with $\eta_0(\lambda) \in \prts{\frac{1}{20L_S(\lambda)}, \frac{1}{2L_S(\lambda)}}$. We define the iteration horizon $T(\lambda)$ and the mini-batch size $B(\lambda)$ as:
\begin{align*}
    T(\lambda) &= \ceiling{ \prtr{\frac{640 B_S^2(\lambda) L_S(\lambda) F(\mathbf{x}_0(\lambda); \lambda)}{3\varepsilon^2 (F^*(\lambda))^2}}^{\frac{1}{1-\alpha}} } \\
    B(\lambda) &= \ceiling{ \frac{T(\lambda) \sigma_S(\lambda)^2}{\kappa} \prtr{\frac{4}{\varepsilon F^*(\lambda)L_S(\lambda)} + \frac{1}{d_f^2 L_S(\lambda)^2} + \frac{16 B_S(\lambda)^2}{\varepsilon^2 (F^*(\lambda))^2}} }
\end{align*}
where the initial-sublevel-set constants $L_S(\lambda), B_S(\lambda)$, and $\sigma_S(\lambda)$ are derived as in Section~\ref{subsubsec:safe-landscape}. Then, the optimization configuration $(\mathcal{A}(\lambda), \mathbf{x}_0(\lambda))$ utilizing $\mathcal{A}(\lambda) \in \mathbb{A}_{\alpha,\mathrm{last}}$ satisfies the conditions of Lemma~\ref{lemma:upperbound-iteration}, guaranteeing $F(\mathbf{x}_{T(\lambda)}; \lambda) \le (1+\varepsilon)F^*(\lambda)$ with probability at least $1-\kappa$.
\end{restatable}

The mini-batch size $B(\lambda)$ performs additional variance reduction, driving the per-step variance below the threshold of Lemma~\ref{lemma:upperbound-iteration}. Without a safe start and the variance control afforded by Assumption~\ref{assumption:second-moment}, the term $\sigma_S^2(\lambda)$ would grow exponentially in the rarity level, forcing exponential mini-batch sizes; the smoothness constant $L_S(\lambda)$ could likewise blow up exponentially, inflating the horizon $T(\lambda)$. A safe start controls both terms, keeping them sub-exponential, so that the total sampling complexity $T(\lambda)\,B(\lambda)$ is sub-exponential as well. Hence the configuration is efficient, establishing the sufficiency direction of Theorem~\ref{thm:main} for the last-iterate SGD.

\begin{restatable}[Sufficiency of Safe Starts for Last-Iterate SGD]{proposition}{lastiteratesufficiency}
\label{lemma:last-iterate-sufficiency}
For a sequence of optimization instances $\{\ProblemRare\}_{\lambda>0}$ defined in \eqref{eq:main-formulation}, suppose Assumptions \ref{assumption:ldp-optimization}--\ref{assumption:init-distance} hold with a sequence of safe starts $\{\mathbf{x}_0(\lambda)\}_{\lambda>0}$. Then, there exists a sequence of SGD schemes $\{\mathcal{A}(\lambda)\}_{\lambda>0}$, drawn from $\mathbb{A}_{\alpha,\mathrm{last}}$ (for any fixed $\alpha \in [0,1)$), such that the sequence of optimization configurations $\{(\mathcal{A}(\lambda), \mathbf{x}_0(\lambda))\}_{\lambda>0}$ is efficient.
\end{restatable}

In short, with safe starts and adaptive variance reduction, we control both the smoothness of the optimization landscape and the noise of the gradient estimator, allowing SGD to converge quickly with high probability.

\subsection{Convergence of Average-Iterate SGD}
\label{subsubsec:average-iterate}

We now turn to the average-iterate scheme, the class $\mathbb{A}_{\alpha,\mathrm{avg}}$
(for $\alpha \in [0,1)$), which outputs the weighted average of the trajectory under
a decaying step-size schedule rather than the last iterate. The
containment-descent coupling persists: local smoothness and variance control
are again available only inside the initial sublevel set $S$, so we must still certify
that the trajectory stays contained while it converges. Unlike the last-iterate case,
however, the output is an average over the whole trajectory, so a per-step descent
guarantee is no longer sufficient. We instead control containment and
descendance \emph{in aggregate}, tracking two trajectory-level quantities in parallel: the
expected optimality gap over the initial sublevel set, which governs descendance, and
the first-time escape probability, which governs containment. Combining the two yields
a high-probability guarantee for the averaged output.

\subsubsection{Convergence at a Fixed Rarity Level}
\label{subsubsec:avg-fixed-rarity}

As before, we fix the rarity level $\lambda$ and write $F(\mathbf{x}) := F(\mathbf{x}; \lambda)$ and $F^* := F(\mathbf{x}^*(\lambda); \lambda)$, with $\mathbf{x}_0$ the initial solution and $S := S(\mathbf{x}_0, \lambda, 2)$ the initial sublevel set. Recall that $L_S$, $\sigma_S$, and $B_S$ denote the local smoothness, gradient variance bound, and bounding radius over $S$ as in Lemma \ref{lem:safe-region-regularity}. Moreover, for a horizon $T$ and step sizes $\prtc{\eta_t}_{t=0}^{T-1}$, recall that the weighted-average iterate is
\begin{equation*}
    \bar{\mathbf{x}}_T := \frac{\sum_{t=0}^{T-1}\eta_t \mathbf{x}_{t+1}}{\sum_{t=0}^{T-1} \eta_t}.
\end{equation*}
 
We let $\mathbb{I}_{\mathrm{safe}, t}$ be the indicator equal to $1$ if all iterates
up to step $t$ lie in the interior of $S$, denoted $\text{int}(S)$, and $0$ otherwise.
Unlike standard iterative optimization analysis that derives a contraction of the
expected optimality gap directly, our regularity bounds hold only inside the sublevel
set. We therefore bound the expected optimality gap on the event that the iterates have
not left $S$, carrying the indicator $\mathbb{I}_{\mathrm{safe}, t}$ through the analysis.

\begin{restatable}[Next-Step Expected Optimality Gap]{lemma}{nextstepeog}
\label{lemma:next-step-eog}
Under Assumptions \ref{assumption:second-moment}--\ref{assumption:connvex}, given that an iterate $\mathbf{x}_t \in \operatorname{int}(S)$ (equivalently, $\mathbb{I}_{\mathrm{safe}, t}=1$), we perform one SGD step with step size $\eta_t \le \frac{1}{L_S}$. Writing $\mathbb{E}_{\xi_t}[\,\cdot \mid \mathbf{x}_t]$ for the expectation over the gradient noise $\xi_t$ conditional on the current iterate $\mathbf{x}_t$, the expected optimality gap within the sublevel set satisfies:
\begin{equation*}
    \mathbb{E}_{\xi_t}\!\prts{ \prtr{F(\mathbf{x}_{t+1}) - F^*}\, \mathbb{I}_{\mathrm{safe}, t+1} \mid \mathbf{x}_t }
    \le \eta_t \sigma_S^2 + \frac{1}{2\eta_t} \prtr{ \|\mathbf{x}_t - \mathbf{x}^*\|^2 - \mathbb{E}_{\xi_t}\!\prts{\|\mathbf{x}_{t+1} - \mathbf{x}^*\|^2 \mid \mathbf{x}_t} }.
\end{equation*}
\end{restatable}
Here, a smaller $\sigma_S^2$ leads to a tighter upper bound on the expected optimality gap. With proper step sizes and noise controlled across the entire trajectory, this per-step improvement accumulates to bound the averaged optimality gap below the target level.
\begin{restatable}[Trajectory-Wide Expected Optimality Gap]{lemma}{decaysublinear}
\label{lemma:decay-sub-linear}
Under Assumptions \ref{assumption:second-moment}--\ref{assumption:connvex}, we perform SGD for $T$ iterations with step sizes $\prtc{\eta_t}_{t=0}^{T-1}$ where $\eta_0 \le \frac{1}{L_S}$. If the variance satisfies $\sigma_S^2 \le \frac{L_S^2 B_S^2}{2T}$ and the step sizes satisfy $\sum_{t=0}^{T-1}\eta_t \ge \frac{B_S^2}{\varepsilon F^*}$, then over the entire trajectory:
\begin{equation*}
    \mathbb{E}\prts{\prtr{F(\bar{\mathbf{x}}_T) - F^*} \mathbb{I}_{\mathrm{safe}, T}} \le \varepsilon F^*
\end{equation*}
\end{restatable}
This bounded optimality gap alone is not enough to ensure our target accuracy with high probability: the indicator $\mathbb{I}_{\mathrm{safe}, T}$ inside the expectation means the gap could decay simply because the trajectory escapes $S$ with growing probability, rather than because it converges. We therefore control the escape probability directly, so that this decay reflects genuine progress.
\begin{restatable}[First-Time Escape Probability]{lemma}{upperboundescape}
\label{lemma:upper-bound-escape}
Under Assumptions \ref{assumption:second-moment}--\ref{assumption:connvex}, we perform SGD with step sizes $\eta_0 \le \frac{1}{2L_S}$. If $\sigma_S \le 1$, the probability of escaping the sublevel set for the first time at iteration $t+1$ is bounded by $\mathbb{P}\prtr{\mathbb{I}_{\mathrm{safe}, t+1}=0 \mid \mathbb{I}_{\mathrm{safe}, t}=1} \le \sigma_S V_S$, where:
\begin{equation*}
    V_S := \frac{2\sqrt{2L_S F(\mathbf{x}_0)} + 1}{4L_S F(\mathbf{x}_0)} + \frac{2 B_S}{\varepsilon F^*} + \frac{B_S^2}{\varepsilon^2(F^*)^2}
\end{equation*}
\end{restatable}
That is, the escape probability scales linearly with the noise level $\sigma_S$, so the same variance control that bounds the optimality gap also drives the trajectory to stay contained. Combining the bounded optimality gap with this escape bound, a sufficiently small stochastic variance guarantees that the averaged iterate reaches the target accuracy with high probability.
\begin{restatable}[Sufficient Conditions for Average-Iterate SGD Convergence]{lemma}{avgsufficientconditions}
\label{lemma:avg-sufficient-conditions}
Under Assumptions \ref{assumption:second-moment}--\ref{assumption:connvex}, we perform SGD for $T$ iterations with step sizes $\prtc{\eta_t}_{t=0}^{T-1}$ where $\eta_0 \le \frac{1}{2L_S}$. Assume the sum of step sizes satisfies:
\begin{equation*}
    \sum_{t=0}^{T-1}\eta_t \ge \frac{2B_S^2}{\kappa \varepsilon F^*}
\end{equation*}
If the variance is constrained such that:
\begin{equation*}
    \sigma_S^2 \le \min \prtc{1, \prtr{\frac{\kappa}{2V_S T}}^2, \frac{L_S^2 B_S^2}{2T}}
\end{equation*}
then the weighted-average iterate $\bar{\mathbf{x}}_T$ satisfies $F(\bar{\mathbf{x}}_T) \le (1+\varepsilon)F^*$ with probability at least $1-\kappa$.
\end{restatable}

\subsubsection{Efficiency of Average-Iterate SGD}
\label{subsubsec:avg-asymptotic}
Similar to the last-iterate case, we construct an explicit $\lambda$-dependent configuration in $\mathbb{A}_{\alpha,\mathrm{avg}}$ --- selecting the step-size schedule $\mathcal A(\lambda)$, horizon $T(\lambda)$, and mini-batch size $B(\lambda)$ so that the conditions of Lemma~\ref{lemma:avg-sufficient-conditions} hold at every rarity level.

\begin{lemma}[Rarity-Scaled Configuration for Average-Iterate SGD]\label{lemma:avg-iterate-configuration}
Under Assumptions \ref{assumption:ldp-optimization}--\ref{assumption:init-distance}, choose a decay rate $\alpha \in [0,1)$. Consider a sequence of safe starts $\prtc{\mathbf{x}_0(\lambda)}_{\lambda>0}$ where $F(\mathbf{x}_0(\lambda); \lambda) > (1+\varepsilon)F^*(\lambda)$. For each $\lambda$, we select step sizes $\eta_t(\lambda) = \frac{\eta_0(\lambda)}{(t+1)^\alpha}$ with $\eta_0(\lambda) \in \prts{\frac{1}{20L_S(\lambda)}, \frac{1}{2L_S(\lambda)}}$. We define the iteration horizon $T(\lambda)$ and the mini-batch size $B(\lambda)$ as:
\begin{align*}
    T(\lambda) &= \ceiling{\prtr{\frac{40L_S(\lambda) B_S^2(\lambda)}{\kappa \varepsilon F^*(\lambda)}}^{\frac{1}{1-\alpha}}} \\
    B(\lambda) &= \ceiling{\sigma_S^2(\lambda)\prtr{1 + \prtr{\frac{2T(\lambda)V_S(\lambda)}{\kappa}}^2 + \frac{2T(\lambda)}{L_S^2(\lambda)B_S^2(\lambda)}}}
\end{align*}
where the local sublevel set constants $L_S(\lambda), B_S(\lambda)$, and single-sample variance $\sigma_S^2(\lambda)$ are derived  as in Section~\ref{subsubsec:safe-landscape}, and the escape bound parameter $V_S(\lambda)$ is defined as in Lemma \ref{lemma:upper-bound-escape}. The average-iterate configuration utilizing $\bar{\mathbf{x}}_{T(\lambda)}(\lambda)$ satisfies the conditions of Lemma \ref{lemma:avg-sufficient-conditions}, guaranteeing $F(\bar{\mathbf{x}}_{T(\lambda)}(\lambda); \lambda) \le (1+\varepsilon)F^*(\lambda)$ with probability at least $1-\kappa$.
\end{lemma}

Under a safe start, the governing constants $L_S(\lambda)$, $B_S(\lambda)$, and $\sigma_S(\lambda)$ remain sub-exponential in $\lambda$, so the horizon and batch size --- and hence the total sample complexity $T(\lambda)\,B(\lambda)$ --- are sub-exponential as well.

\begin{restatable}[Sufficiency of Safe Starts for Average-Iterate SGD]{proposition}{avgsufficiency}
\label{lemma:avg-iterate-sufficiency}
For a sequence of optimization instances $\{\ProblemRare\}_{\lambda>0}$ defined in \eqref{eq:main-formulation}, suppose Assumptions \ref{assumption:ldp-optimization}--\ref{assumption:init-distance} hold with a sequence of safe starts $\{\mathbf{x}_0(\lambda)\}_{\lambda>0}$. Then, there exists a sequence of SGD schemes $\{\mathcal{A}(\lambda)\}_{\lambda>0}$, drawn from $\mathbb{A}_{\alpha,\mathrm{avg}}$ (for any fixed $\alpha \in [0,1)$), such that the sequence of optimization configurations $\{(\mathcal{A}(\lambda), \mathbf{x}_0(\lambda))\}_{\lambda>0}$ is efficient.

\end{restatable}

Together with Proposition~\ref{lemma:last-iterate-sufficiency}, this establishes that a safe start suffices for efficient rare-event optimization under both the last-iterate and average-iterate schemes.


\subsection{Necessity of Safe Start and Efficient Gradient Estimators}
\label{subsec:nec-safe-start}

We now provide explicit rare-event optimization instances that constitute the
necessity results in Theorems \ref{thm:unsafe} and \ref{thm:highvariance}.

\begin{proposition}[Explicit Instance for Theorem~\ref{thm:unsafe}]
\label{prop:unsafe-instance}
Consider a sequence of two-dimensional optimization instances
$\{\ProblemRare\}_{\lambda>0}$ on $\mathcal{X}=\mathbb{R}_{\geq0}^2$,
\[
  (\ProblemRare) \qquad \min_{\mathbf{x}=(u,v)\in\mathbb{R}_{\geq0}^2}
  F(\mathbf{x};\lambda)= f(\mathbf{x}) +\gamma(\lambda) p(\mathbf{x})
\]
where $f(\mathbf{x})= u^2 +v$, $p(\mathbf{x})= e^{-u-v}$, and $\gamma(\lambda)= e^{\lambda}$.
We consider a sequence of unsafe starts $\prtc{\mathbf{x}_0(\lambda)}_\lambda$ where
$\mathbf{x}_0(\lambda) = \big(\tfrac{\lambda}{4},0\big)$. For  any
$\{\mathcal{A}(\lambda)\}_{\lambda>0}$ from either last-iterate or average-iterate SGD
with monotonically diminishing step sizes,
\[
  \limsup_{\lambda\to\infty}\tfrac1\lambda
  \log T^{\varepsilon,\kappa}_{\lambda}\big(\mathcal{A}(\lambda),\mathbf{x}_0(\lambda)\big)>0,
\]
where $\varepsilon=1$ and $\kappa = \tfrac12$.
\end{proposition}

We verify first that these instances satisfy every assumption except the safe start
(Assumption~\ref{def:safe-start}). The optimal solution is
$\mathbf{x}^*(\lambda)=\big(\tfrac12,\lambda-\tfrac12\big)$. The optimal tail risk decays at
rate $\lim_{\lambda\to\infty}-\tfrac1\lambda\log p^*(\lambda)=I$ with $I=1$, and the weight
$\gamma(\lambda)=e^{\lambda}$ matches this rate
(Assumption~\ref{assumption:ldp-optimization}). We work in the exact-gradient setting,
which satisfies Assumption~\ref{assumption:second-moment} trivially with
$\sigma_f=\sigma_p(\lambda)=0$. The average-case loss $f(\mathbf{x})=u^2+v$ is convex, twice
differentiable, and $L_f$-smooth with $L_f=2$
(Assumption~\ref{assumption:smooth-connvex}). Moreover, its sublevel set
$S_{\epsilon_f}=\{\mathbf{x}\in\mathcal{X}\mid u^2+v\le 1\}\subseteq[0,1]^2$ when
$\epsilon_f=1$ is bounded (Assumption~\ref{assumption:level-bounded}), and it attains minimum at $\tilde{\mathbf{x}}^*=(0,0)$ with $f(\tilde{\mathbf{x}}^*)=0$
(Assumption~\ref{assumption:lower-bounded}). The gradient and Hessian of the risk are
controlled by the risk itself, in the sense $\|\nabla p(\mathbf{x})\|=\sqrt2\,p(\mathbf{x})$ and
$\|\nabla^2 p(\mathbf{x})\|=2\,p(\mathbf{x})$, so
Assumption~\ref{assumption:scale-grad-hessian} holds with $\beta_1=\beta_2=0$. The
combined objective $F(\cdot;\lambda)$ is convex for every $\lambda$
(Assumption~\ref{assumption:connvex}), and the initialization is sub-exponentially close
to $\tilde{\mathbf{x}}^*$, since
$\limsup_{\lambda\to\infty}\tfrac1\lambda\log\|\mathbf{x}_0(\lambda)-\tilde{\mathbf{x}}^*\|
=\limsup_{\lambda\to\infty}\tfrac1\lambda\log(\lambda/4)=0$
(Assumption~\ref{assumption:init-distance}). The safe-start condition, however, fails:
$\limsup_{\lambda\to\infty}\tfrac1\lambda\log p(\mathbf{x}_0(\lambda))=-\tfrac14>-I$.

We first provide a proof sketch for the last-iterate scheme. In this deterministic setting, the sample complexity
$T^{\varepsilon,\kappa}_{\lambda}(\mathcal{A}(\lambda),\mathbf{x}_0(\lambda))$ is basically the first stopping time $\tau=\tau(\lambda)$ at which the iterate reaches the target level
$F(\mathbf{x}_{\tau};\lambda)\le 2F(\mathbf{x}^*(\lambda);\lambda)$. We can derive the necessary condition of $\mathbf{x}_\tau= (u_\tau, v_\tau)$ as follows: since $F(\mathbf{x};\lambda)\ge v$, the
target requires $v_\tau\le 2\lambda+1.5$; and since
$F(u_\tau,v_\tau;\lambda)\ge u_\tau^2-u_\tau+\lambda+1$ by minimizing over $v$, it also
requires $u_\tau<\lambda/5$. Writing
$\tau_u(\lambda)=\min\{t\ge1: u_t<\tfrac{\lambda}{5}\}$ and
$\tau_v(\lambda)=\min\{t\ge1: v_t\le 2\lambda+1.5\}$, we have
$\tau(\lambda)\ge\max\{\tau_u(\lambda),\tau_v(\lambda)\}$. We then show that, for any step sizes, either of these terms grow exponentially in the rarity level.

We divide the analysis into two cases. If $\eta_0\ge e^{-\lambda/2}$,
the exponentially large initial risk gradient makes the first $v$-update overshoot by
$v_1=\Omega(\eta_0 e^{3\lambda/4})$, an exponentially large distance from $v^*(\lambda)$.
Each subsequent step returns $v_t$ by at most $\eta_t\le\eta_0$ as the gradient $\grad_v F(u_t, v_t;\lambda )$ is roughly $1$; thus,
$\tau_v(\lambda)=\Omega(e^{3\lambda/4})$. If instead $\eta_0< e^{-\lambda/2}$, the
$u$-coordinate crawls: from $u_0=\lambda/4$, each step contracts $u_t$ by a factor at most
$1-2\eta_t$, while the target accuracy demands $u_\tau<\lambda/5$, resulting in
$\tau_u(\lambda)>\tfrac1{10}e^{\lambda/2}$. Either way, $\tau(\lambda)$ grows
exponentially in $\lambda$.

The average-iterate scheme inherits the same bound: a weighted average of iterates that
have all stayed above the thresholds $\lambda/5$ and $2\lambda+1.5$ stays above them as
well, so the average-iterate scheme reaches the thresholds no sooner than the last iterate does. Next, we present
the instances for Theorem~\ref{thm:highvariance}.

\begin{proposition}[Explicit Instance for Theorem~\ref{thm:highvariance}]
\label{prop:highvariance-instance}
Consider a sequence of one-dimensional optimization instances $\{\ProblemRare\}_{\lambda>0}$ on $\mathcal{X}=\mathbb{R}_{\geq0}$,
\[
  (\ProblemRare)\qquad \min_{x\in\mathbb{R}_{\geq0}} F(x;\lambda)= x+\gamma(\lambda)p(x),
\]
where $p(x)=e^{-x}$ and $\gamma(\lambda)=e^{\lambda}$. Suppose $\nabla f$ is estimated exactly by $G_f\equiv1$, while $\nabla p$ is estimated by
\[
  G_p(x,\xi)=
  \begin{cases}
    -e^{-x}+e^{-x/2}, & \text{w.p. } \tfrac12,\\[2pt]
    -e^{-x}-e^{-x/2}, & \text{w.p. } \tfrac12,
  \end{cases}
  \quad\text{for\ \ } x\in\mathcal{H}_\lambda,
  \qquad
  G_p(x,\xi)=-e^{-x}\ \text{ otherwise},
\]
for a high-variance region $\mathcal{H}_\lambda=\big(\tfrac32(\lambda+1),\,\tfrac{8}{5}(\lambda+1)\big)$. For any constant-step-size last-iterate scheme $\mathcal{A}(\lambda)\in\mathbb{A}_{0,\mathrm{last}}$with a constant mini-batch size $B(\lambda)$, there exists a safe start $x_0(\lambda)\in \mathcal{X}_0(\lambda)=[\tfrac{8}{5}(\lambda+1),\,\tfrac{33}{10}(\lambda+1)+1]$ under which
\[
  \limsup_{\lambda\to\infty}\tfrac1\lambda
  \log T^{\varepsilon,\kappa}_{\lambda}\big(\mathcal{A}(\lambda),x_0(\lambda)\big)>0,
\]
where $\varepsilon=\tfrac12$ and $\kappa =\tfrac12$.
\end{proposition}

These instances satisfy every assumption, including the safe start, except 
the efficiency of the gradient estimator (Assumption \ref{assumption:second-moment}).   The optimal solution is $x^*(\lambda)=\lambda$. The optimal risk decays at rate
$\lim_{\lambda\to\infty}-\tfrac1\lambda\log p^*(\lambda)=I=1$, matched by
$\gamma(\lambda)=e^{\lambda}$ (Assumption~\ref{assumption:ldp-optimization}). The
average-case loss $f(x)=x$ is convex, twice differentiable, and $L_f$-smooth for any
$L_f>0$ (Assumption~\ref{assumption:smooth-connvex}); its sublevel set
$\{x\ge0\mid x\le\epsilon_f\}$ is bounded (Assumption~\ref{assumption:level-bounded});
and it attains minimum at $\tilde{x}^*=0$ with $f(\tilde{x}^*)=0$
(Assumption~\ref{assumption:lower-bounded}). The risk's gradient and Hessian norms are controlled by the risk itself, in the sense
$\|\nabla p(x)\|=\|\nabla^2 p(x)\|=p(x)$, so
Assumption~\ref{assumption:scale-grad-hessian} holds with $\beta_1=\beta_2=0$. Moreover, 
$F(\cdot;\lambda)$ is convex since $F''(x;\lambda)=e^{\lambda-x}>0$
(Assumption~\ref{assumption:connvex}). Every $x_0(\lambda)\in \mathcal{X}_0(\lambda)$ lies within
$\mathcal{O}(\lambda)$ of $\tilde{x}^*$, so the initialization is sub-exponentially
close (Assumption~\ref{assumption:init-distance}), and it is a safe start
(Assumption~\ref{def:safe-start}): $p(x_0(\lambda))\le e^{-\frac85(\lambda+1)}$ gives
$\limsup_{\lambda\to\infty}\tfrac1\lambda\log p(x_0(\lambda))\le-\tfrac85<-I$.

The gradient estimator, however, is not efficient in the sense of satisfying Assumption~\ref{assumption:second-moment}. The efficiency bound in that assumption requires
$\mathbb{E}\big[\|G_p(x,\xi)-\nabla p(x)\|^2\big]= e^{-x}$ to be smaller than $ \sigma_p(\lambda)^2\,\|\nabla p(x)\|^2=\sigma_p(\lambda)^2 e^{-2x}$, forcing
$\sigma_p(\lambda)^2\ge e^{x}\ge e^{\frac32(\lambda+1)}$ over $\mathcal{H}_\lambda$. This violates the sub-exponential requirement of $\sigma_p$. The noise here only satisfies  the
relaxed absolute bound of Theorem~\ref{thm:highvariance}, namely
$\mathbb{E}\big[\|G_p(x,\xi)-\nabla p(x)\|^2\big]\leq \|\nabla p(x)\|$.

 For any fixed step size $\eta$, we can choose a safe start  $x_0\in \mathcal{X}_0$ that leads to an exponential sample complexity. Let $\tau$ be the stopping time that the iterate reach the target accuracy, i.e., $F(x_\tau;\lambda)\leq 1.5 F(x^*(\lambda);\lambda)$. Because $F(x; \lambda) > x$, reaching the target accuracy requires $x_\tau \leq 1.5(\lambda+1)$. We divide the analysis into three cases.

If the step size is too small ($\eta<e^{-0.1\lambda}$), we pick $x_0=\tfrac85(\lambda+1)+1$. The iterate sits above the high-variance interval, in the noiseless region, where each step lowers $x_t$ by at most $\eta$. That is, before arriving at $x_\tau \leq 1.5(\lambda+1)$, reaching the boundary of $\mathcal{H}_\lambda$ at $\tfrac85(\lambda+1)$ alone takes at least $1/\eta>e^{0.1\lambda}$ iterations.

If the step size is too large ($\eta>\tfrac{17}{10}(\lambda+1)+1$), we again pick $x_0=\tfrac85(\lambda+1)+1$. Now a single gradient step projects the iterate $x_1$ to $0$; then, the next gradient step sends the next iterate $x_2$ to $\eta(e^{\lambda}-1)$, exponentially far from $x^*(\lambda)$. Descent from there is capped at $\eta$ per step, so returning takes $\Omega(e^{\lambda})$ iterations, the same overshoot phenomenon as in Proposition~\ref{prop:unsafe-instance}.

The remaining case is intermediate $\eta\in[e^{-0.1\lambda},\,\tfrac{17}{10}(\lambda+1)+1]$. We can show that there exists $x_0\in \mathcal{X}_0$ such that the next iterate $x_1$ lands in the high-variance interval $\mathcal{H}_\lambda$. Once inside, unless the batch is exponentially large, with probability at least one
half, the noise drives the next iterate $x_2$ either to $0$ or exponentially far past
$\tfrac85(\lambda+1)$. In the former case, the following step $x_3$ is thrown exponentially
far as well; from either position, the number of subsequent iterations needed to meet the
necessary condition $x_\tau \le 1.5(\lambda+1)$ is exponential. In every case, either $\tau$ or $B$ is exponential in $\lambda$.

\section{Experimental Results}
\label{sec:experiment}

We apply safe starts in three experimental settings to highlight and empirically validate our theoretical insights. These settings include the estimation of extreme Value-at-Risk for financial portfolios (Section \ref{subsec:exp-finance}), portfolio optimization using conditional value-at-risk (Section \ref{sec:portfolio optimization}), and robust machine learning using the MAGIC Gamma Telescope data set (Section \ref{subsec:exp-magic}).






\subsection{Value-at-Risk at Extreme Level}

\label{subsec:exp-finance}


We estimate the Value-at-Risk (VaR) at an extreme level for a derivatives portfolio composed of multiple assets.

\paragraph{Problem formulation and variance reduction.}
We adopt a setting similar to that of \cite{glasserman2000variance} and \cite{he2024adaptive}. By the Rockafellar-Uryasev CVaR representation, the VaR is estimated by solving the following optimization problem: $\min_{x \in \mathbb{R}}\prtc{x+\frac{1}{1-q} \mathbb{E}_\xi \prts{ \prtr{L(t, \xi)-x}_+ }}$
where $L(t,\xi)$ is the discounted portfolio loss at time $t$. Because the hinge-loss risk term vanishes as $x$ grows, any sufficiently large initialization is a safe start for this problem. To reduce variance, we apply importance sampling (IS) built from a state-dependent quadratic approximation of the loss \citep{glasserman2000variance}.

\paragraph{Experiment setting.}
The simulation assumes a 250-day trading year, a 10-day risk horizon ($t=0.04$ years), and a risk-free rate of $r=5\%$. The portfolio tracks ten pairwise uncorrelated underlying assets (initial price $100$, volatility $0.3$). Positions include 10 at-the-money calls and 5 at-the-money puts per asset, all with half-year expirations. The target is the $q$-VaR where $q = 1-10^{-5}$. We compare three SGD specifications: a diminishing average-iterate scheme with IS ($\mathbb{A}_{0.1,avg}^{IS}$), a constant last-iterate scheme with IS ($\mathbb{A}_{0,last}^{IS}$), and a constant last-iterate scheme without IS ($\mathbb{A}_{0,last}^{MC}$ using Monte Carlo estimations).

 We evaluate 5 initializations: $x_0 \in \{0, 200, 400,600,800\}$. Each initialization is configured with 5 different initial step sizes: $\eta_0 \in \{10^{-2}, 10^{-1}, 1, 10, 10^{2}\}$. We fix the batch size to $B=10$ and the iteration limit to $T=500$. We then evaluate the convergence probability of each configuration defined as the probability that the final output from that configuration falls within a  tolerance band of $x^* \pm 10$, which is $[372, 392].$ For each configuration, this convergence probability is estimated by evaluating over 100 independent sample paths.

\paragraph{Results of safe start.}
Figure \ref{fig:heatmaps_finance} reports the empirical convergence probabilities for the three SGD specifications. The x-axis represents initialization $x_0$, separated into safe ($x_0 \ge x^* \approx 382$) and unsafe ($x_0 < x^*$) regions, and the y-axis represents the initial step size $\eta_0$. Each cell is colored by its convergence probability, from green (near $1$) to red (near $0$). Figures \ref{fig:heatmaps_finance}a and \ref{fig:heatmaps_finance}b show that the two IS schemes reliably converge to the target VaR when started from a safe initialization with an appropriate step size, whereas unsafe initializations fail regardless of step size. Finally, Figure \ref{fig:heatmaps_finance}c shows that the standard MC scheme fails across all tested configurations, illustrating that variance reduction remains necessary even with a safe start.

\begin{figure}[htbp]
    \centering
    \begin{subfigure}{0.48\textwidth}
        \centering
        \includegraphics[width=\linewidth]{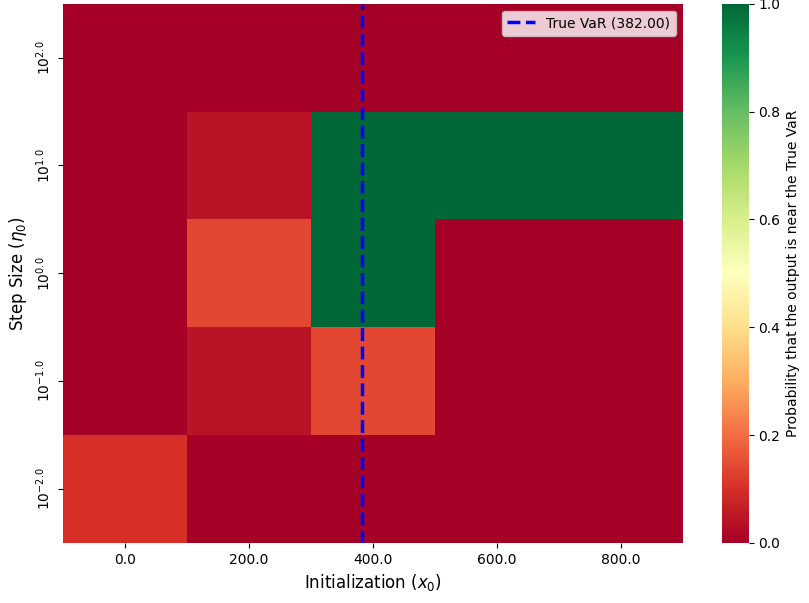}
        \caption{$\mathbb{A}_{0.1,\mathrm{avg}}^{\mathrm{IS}}$}
        \label{fig:heatmaps-a}
    \end{subfigure}\hfill
    \begin{subfigure}{0.48\textwidth}
        \centering
        \includegraphics[width=\linewidth]{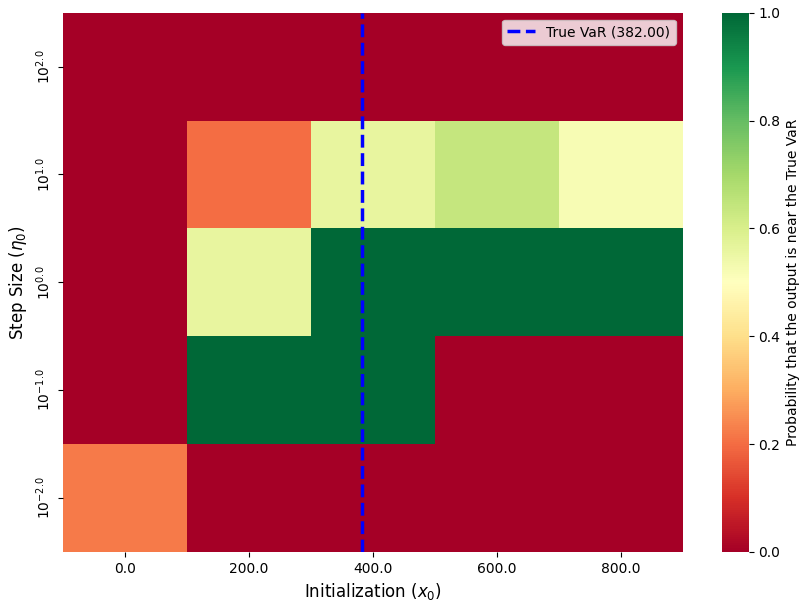}
        \caption{$\mathbb{A}_{0,\mathrm{last}}^{\mathrm{IS}}$}
        \label{fig:heatmaps-b}
    \end{subfigure}

    \vspace{0.5cm} 

    \begin{subfigure}{0.48\textwidth}
        \centering
        \includegraphics[width=\linewidth]{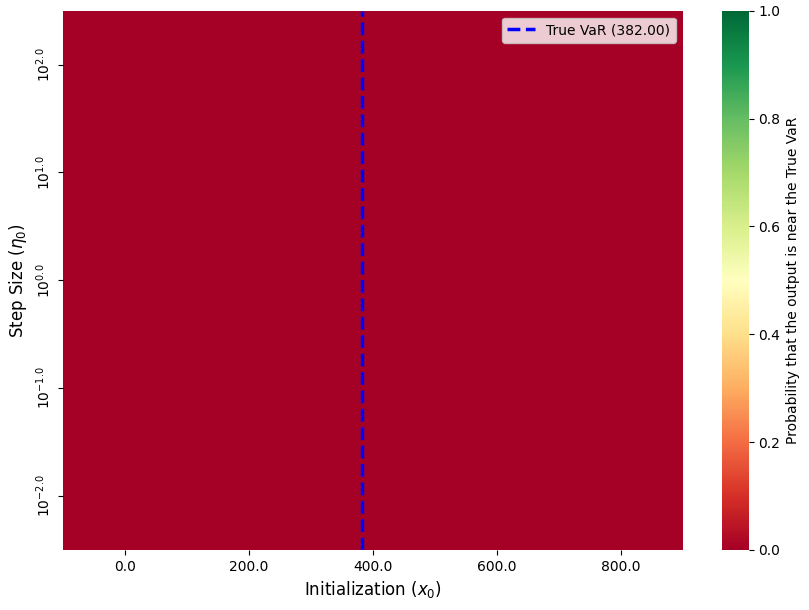}
        \caption{$\mathbb{A}_{0,\mathrm{last}}^{\mathrm{MC}}$}
        \label{fig:heatmaps-c}
    \end{subfigure}
    
  \caption{\textbf{Empirical convergence probability across configurations for portfolio VaR.} Each cell reports the empirical convergence probability at a given initialization $x_0$ (x-axis) and initial step size $\eta_0$ (y-axis), colored from green (near $1$) to red (near $0$). The dashed line marks the target $x^* \approx 382$. The three panels correspond to three SGD specifications: a diminishing average-iterate scheme with IS ($\mathbb{A}_{0.1,avg}^{IS}$), a constant last-iterate scheme with IS ($\mathbb{A}_{0,last}^{IS}$), and a constant last-iterate scheme without IS ($\mathbb{A}_{0,last}^{MC}$).}
    \label{fig:heatmaps_finance}
\end{figure}


\subsection{CVaR Portfolio Optimization}
\label{sec:portfolio optimization}


We extend the use of safe start to a multivariate setting through portfolio optimization. Here, we optimize an asset allocation to achieve minimal Conditional Value-at-Risk (CVaR). We assume the underlying asset prices follow a multivariate Gaussian distribution, allowing us to derive the minimal CVaR analytically and to verify that both a safe start and variance reduction are needed to converge to the minimal.

\paragraph{Problem formulation and variance reduction.}
Consider an $n$-dimensional random vector $R \sim \mathcal{N}(\mu, \Sigma)$ representing the returns of $n$ assets with known mean $\mu$ and covariance matrix $\Sigma$. Our goal is to find a long-only allocation $\mathbf{w} \in \mathbb{R}^n$, i.e., $\mathbf{w} \geq 0$ and $\mathbf{w}^\top \mathbf{1} = 1$, that minimizes the CVaR at level $q = 1 - 10^{-5}$. Like in the previous example, by the Rockafellar–Uryasev representation, we can formulate it as the following optimization problem:
$\min_{\mathbf{w},\, x}\ \left\{ x + \frac{1}{1-q}\, \mathbb{E}_{R}\!\left[\left(-\mathbf{w}^\top R - x\right)_+\right] \right\},$
where $-\mathbf{w}^\top R$ is the realized portfolio loss under allocation $\mathbf{w}$ and $x$ is an auxiliary variable. Under the Gaussian model, the inner CVaR admits a closed form: for a fixed allocation $\mathbf{w}$,
$\mathrm{CVaR}_q(\mathbf{w}) = -\mathbf{w}^\top \mu + \sqrt{\mathbf{w}^\top \Sigma\, \mathbf{w}}\;\frac{\phi\!\left(\Phi^{-1}(q)\right)}{1-q},$
where $\phi$ and $\Phi$ denote the standard normal density and CDF. Minimizing this over the simplex constraint is a convex program that yields the optimal allocation $\mathbf{w}^*$ and hence the exact minimal CVaR. Its VaR can also be derived as $x^* = -(\mathbf{w}^*)^\top \mu + \sqrt{(\mathbf{w}^*)^\top \Sigma\, \mathbf{w}^*}\;\Phi^{-1}(q)$.

To study SGD on this problem, we run projected gradient descent with adaptive variance
reduction. At iteration $t$, given the current iterate $(\mathbf{w}_t, x_t)$, the target event
$\{-\mathbf{w}_t^\top R \geq x_t\}$ becomes rare as the iterate approaches the optimum. We
therefore estimate the gradient by adaptive importance sampling: we sample
$R_{\theta_t} \sim \mathcal{N}\!\left(\mu - \theta_t\,\Sigma \mathbf{w}_t,\ \Sigma\right)$ with
$\theta_t = \frac{x_t + \mathbf{w}_t^\top \mu}{\mathbf{w}_t^\top \Sigma\, \mathbf{w}_t},$
chosen so that the loss $-\mathbf{w}_t^\top R_{\theta_t}$ is centered at the threshold $x_t$.

\paragraph{Experiment setting.}
We generate a random instance of $n = 10$ assets. Each asset's mean return is drawn
from $\mathcal{N}(0,1)$ and its volatility from $\mathrm{Unif}[0.8, 1.6]$, with a common
pairwise correlation of $0.4$ between distinct assets. For this instance, the minimal
CVaR is $F^* = 2.783$, attained at the optimal VaR $x^* = 2.638$.

We consider two SGD specifications: a diminishing last-iterate scheme with IS
($\mathbb{A}^{\mathrm{IS}}_{0.2,\mathrm{last}}$) and one with Monte Carlo estimation
($\mathbb{A}^{\mathrm{MC}}_{0.2,\mathrm{last}}$). Each uses a minibatch of $B = 256$ over
$T = 500$ iterations, with initial step sizes ranging from $10^{-3}$ to $10^{2}$. We
evaluate one safe and one unsafe initialization, both using the uniform allocation
$\mathbf{w}_0 = \mathbf{1}/n$ and differing only in the auxiliary variable $x_0$. The
$q$-VaR of the uniform allocation is $3.337$, which is derived analytically; the safe start sets $x_0 = 6.67$, so that
the initial risk term is of order $O(1-q)$, while the unsafe start sets $x_0$ to the mean
portfolio loss, $x_0 = -0.085$, where the risk term is large. For each configuration
we run $100$ seeds and report the empirical convergence probability, the fraction of runs
whose allocation attains a CVaR within $(1+\varepsilon)F^*$ ($\varepsilon = 0.05$).
 
\paragraph{Results of safe start.} As shown in 
Table~\ref{table:cvar_results}, projected gradient descent converges to the optimal allocation only when it is both
equipped with adaptive variance reduction and initialized at a safe start. The empirical
convergence probability is zero for every configuration with either the unsafe start or without the 
variance reduction scheme.

\begin{table}[ht]
\centering
\caption{Empirical convergence probability for CVaR portfolio optimization across step
sizes, for the two initializations (safe / unsafe) with and without variance reduction
(exponential tilting). Each entry is the fraction of $100$ seeds whose resulting
allocation attains a CVaR within $(1+\varepsilon)F^*$ ($\varepsilon = 0.05$). }
\label{table:cvar_results}
\setlength{\tabcolsep}{10pt}
\begin{tabular}{l cc cc}
\toprule
& \multicolumn{2}{c}{\textbf{With VR ($\mathbb{A}^{\mathrm{IS}}_{0.2,\mathrm{last}}$)}}
& \multicolumn{2}{c}{\textbf{Without VR ($\mathbb{A}^{\mathrm{MC}}_{0.2,\mathrm{last}}$)}} \\
\cmidrule(lr){2-3} \cmidrule(lr){4-5}
\textbf{Step Size}
& \textbf{Unsafe} & \textbf{Safe}
& \textbf{Unsafe} & \textbf{Safe} \\
\midrule
$10^{-3}$ & 0.00 & 0.00          & 0.00 & 0.00 \\
$10^{-2}$ & 0.00 & 0.00          & 0.00 & 0.00 \\
$10^{-1}$ & 0.00 & \textbf{1.00} & 0.00 & 0.00 \\
$10^{0}$  & 0.00 & \textbf{0.98} & 0.00 & 0.00 \\
$10^{1}$  & 0.00 & 0.00          & 0.00 & 0.00 \\
$10^{2}$  & 0.00 & 0.00          & 0.00 & 0.00 \\
\bottomrule
\end{tabular}
\end{table}

\subsection{Statistically Robust Classification on MAGIC Gamma Telescope}

\label{subsec:exp-magic}

This experiment shows that the safe-start notion extends to the training of robust machine learning models. We demonstrate this by training a classifier that stays accurate under input perturbations, a robustness criteria studied extensively in prior work \citep{goodfellow2014explaining, carlini2017towards}. Our focus is the high-robustness regime, where we require the prediction to both achieve high accuracy and remain accurate, with probability near one, under input perturbation .

\paragraph{Accuracy and robustness metrics.}
Consider a robust classification problem where the classifier is parameterized by the decision variable $\vect{x}$. The classifier sees a random input vector $\mathbf{Z}$ with an unobserved label $Y \in \{ 1,2,\ldots ,C\}$ for $C$ categories and classifies it as $\hat{y}(\mathbf{Z}; \vect{x}) \in \{ 1,2,\ldots ,C\}$. The pair $(\mathbf{Z}, Y)$ follows a data distribution $P_{\text{data}}$, from which we observe i.i.d. samples $\prtc{(\mathbf{z}_i, y_i)}_{i}$. We evaluate a classifier $\vect{x}$ by both its prediction accuracy and its robustness to a random input perturbation $\boldsymbol{\epsilon}$, defined as follows:
\begin{enumerate}
    \item \textbf{Prediction accuracy (PA):} $\mathbb{P}_{(\mathbf{Z},Y)}[\hat{y}(\mathbf{Z}; \vect{x}) = Y]$
    \item \textbf{Conditional robust accuracy (CRA):} $\mathbb{E}_{(\mathbf{Z},Y)}[ \mathbb{P}_{\boldsymbol{\epsilon}} (\hat{y}(\mathbf{Z}+\boldsymbol{\epsilon}; \vect{x}) = Y ) \mid \hat{y}(\mathbf{Z}; \vect{x}) = Y]$
\end{enumerate}
Here $\boldsymbol{\epsilon}$ follows some fixed noise distribution. Here, CRA measures robustness in terms of how often a correct prediction survives the perturbation. We condition on $\hat{y}(\mathbf{Z}; \vect{x}) = Y$ because robustness is only meaningful where the classifier is already correct; an input that is misclassified to begin with has no correct prediction to preserve.

\paragraph{Problem formulation.}
To achieve these two criteria simultaneously, we optimize the combined loss function:
\begin{equation}
    F(\vect{x}) := f(\vect{x}) + \gamma p(\vect{x})
\end{equation}
where $f(\vect{x}) := \mathbb{E}_{(\mathbf{Z},Y)}[ \ell(\mathbf{Z},Y; \vect{x}) ]$ is the standard cross-entropy loss and $p(\vect{x}) := \mathbb{E}_{(\mathbf{Z},Y)}[ \mathbb{I}\{ \hat{y}(\mathbf{Z};\vect{x} )=Y\} \, \mathbb{E}_{\boldsymbol{\epsilon}} [ \max_{c \neq Y} g_c(\mathbf{Z}+\boldsymbol{\epsilon}; \vect{x}) - g_Y(\mathbf{Z}+\boldsymbol{\epsilon}; \vect{x})]_+ ]$ acts as the risk term, with $g_c(\cdot; \vect{x})$ the logit output of the network for class $c$. Similar to the hinge penalty in the CVaR formulation, this term activates only when a perturbation pushes an incorrect class score above the true class score. A large penalty coefficient $\gamma$ emphasizes the robustness term, driving the classifier toward high consistency under perturbation.

\paragraph{Dataset and classification model.}
We use the MAGIC Gamma Telescope dataset \citep{asuncion2007uci}, which contains 19,020 data points. Each data point has a 10-dimensional feature vector of an atmospheric image and a label indicating whether it was induced by gamma rays (signal) or cosmic rays (background). The dataset contains $12,332$ signal images and $6,688$ background images. We train a multi-layer perceptron (MLP) with two hidden layers of 20 neurons each to distinguish the two.

\paragraph{Experiment setting and safe start.}
We split the dataset into $14{,}000/5{,}020$ points for training and testing. We set the penalty weight to $\gamma = 10^5$ and optimize for $100$ epochs with a minibatch size of $1{,}024$. We do not apply variance reduction to the risk term; instead, we estimate the perturbation expectation by crude Monte Carlo, drawing $100$ noise samples $\boldsymbol{\epsilon} \sim \mathcal{N}(\mathbf{0}, \sigma^2 \mathbf{I}_{10})$ with $\sigma = 2$ for each point in the minibatch. On the testing set, we draw $100{,}000$ noise samples per point to evaluate the CRA.

We consider two optimizers: a constant-step-size last-iterate SGD ($\mathbb{A}^{\mathrm{MC}}_{0,\mathrm{last}}$) and Adam \citep{kingma2014adam}, both using Monte Carlo gradient estimates. For each, we sweep $6$ initial step sizes: $\prtc{10^{-9},10^{-8}, \dots, 10^{-3}}$ for SGD and $\prtc{10^{-5},10^{-4}, \dots, 10}$ for Adam.

\paragraph{Constructing the safe start.}
We consider one safe initialization $\vect{x}_0^{\text{safe}}$ and one unsafe initialization $\vect{x}_0^{\text{unsafe}}$. The unsafe start is obtained by running the Adam optimizer for $400$ epochs on the average-case cross-entropy loss $f(\vect{x})$ over the training set. To construct the safe start, we take the unsafe start and modify solely the final-layer bias of the positive class, setting it to a constant ($15$ here) large enough that the model predicts almost exclusively that class. By this construction, the safe start incurs a much higher cross-entropy loss but a much smaller initial tail risk: by favoring one class, its predictions stay consistent under perturbation. Table \ref{table:magic_init} reports the initial conditions of each initialization.
\begin{table}[ht]
\centering
\caption{Conditions of the two initializations. }
\label{table:magic_init}
\begin{tabular}{l l l l l}
\toprule
Initialization  & $f(\vect{x}_0)$ & $p(\vect{x}_0)$ & PA & CRA\\
\midrule
\textbf{Unsafe start ($\vect{x}_0^{\text{unsafe}}$)}  & 0.34 & \textbf{1.50}  & 84.54\% & 64.60\% \\
\textbf{Safe start ($\vect{x}_0^{\text{safe}}$)}& 4.01 & \textbf{0.019}  & 66.53\% & 99.24\% \\
\bottomrule
\end{tabular}
\end{table}

\paragraph{Results of safe start.}
Our goal is to achieve both high standard accuracy ($>80\%$) and high conditional robustness ($\geq 99\%$). As shown in Table \ref{table:magic_results}, SGD with the unsafe start does not reach these targets at any step size we tested. In contrast, SGD with the safe start reaches them at initial step sizes $10^{-6}$ and $10^{-7}$. While the unsafe start fails under SGD, it does reach the targets at two step sizes under Adam. We attribute this to Adam's per-coordinate normalization, which damps the oversized updates behind the overshoot failure investigated  in  Theorem \ref{thm:unsafe}. Even so, the safe start remains advantageous, offering a wider band of viable step sizes.

\begin{table}[ht]
\centering
\caption{Prediction accuracy (PA) and conditional robust accuracy (CRA) across optimizers and step sizes, reported as PA / CRA. Values are the mean over 5 independent seeds, with standard deviation in parentheses. Bold indicates configurations achieving both high standard accuracy ($>80\%$) and conditional robustness ($\geq 99\%$). SGD and Adam are swept over different step-size ranges.}
\label{table:magic_results}
\resizebox{\textwidth}{!}{%
\begin{tabular}{c c c c c c}
\toprule
\multicolumn{3}{c}{\textbf{SGD} ($\mathbb{A}^{\mathrm{MC}}_{0,\mathrm{last}}$)} & \multicolumn{3}{c}{\textbf{Adam}} \\
\cmidrule(lr){1-3} \cmidrule(lr){4-6}
\textbf{Step Size} & \textbf{Unsafe ($\vect{x}_0^{\text{unsafe}}$)} & \textbf{Safe ($\vect{x}_0^{\text{safe}}$)} & \textbf{Step Size} & \textbf{Unsafe ($\vect{x}_0^{\text{unsafe}}$)} & \textbf{Safe ($\vect{x}_0^{\text{safe}}$)} \\
\midrule
$10^{-3}$ & 64.0 (0.0) / 100.0 (0.0) & 36.0 (0.0) / 100.0 (0.0) & $10^{1}$ & 64.0 (0.0) / 100.0 (0.0) & 36.0 (0.0) / 100.0 (0.0) \\
$10^{-4}$ & 64.0 (0.0) / 100.0 (0.0) & 58.4 (11.2) / 100.0 (0.0) & $10^{0}$ & 41.6 (11.2) / 100.0 (0.0) & 58.4 (11.2) / 100.0 (0.0) \\
$10^{-5}$ & 64.0 (0.0) / 100.0 (0.0) & 65.7 (3.5) / 99.9 (0.1) & $10^{-1}$ & 58.4 (11.2) / 100.0 (0.0) & 64.1 (15.4) / 99.7 (0.3) \\
$10^{-6}$ & 69.7 (17.1) / 99.2 (0.6) & \textbf{80.5 (2.2) / 99.2 (0.3)} & $10^{-2}$ & \textbf{80.9 (2.0) / 99.2 (0.3)} & \textbf{82.7 (0.3) / 99.1 (0.1)} \\
$10^{-7}$ & 79.5 (0.5) / 98.6 (0.1) & \textbf{80.5 (0.2) / 99.0 (0.1)} & $10^{-3}$ & \textbf{80.9 (0.8) / 99.1 (0.2)} & \textbf{82.8 (0.3) / 99.1 (0.1)} \\
$10^{-8}$ & 78.7 (0.1) / 97.4 (0.1) & 79.9 (0.1) / 98.1 (0.0) & $10^{-4}$ & 76.4 (0.1) / 97.3 (0.1) & \textbf{80.7 (0.1) / 99.0 (0.1)} \\
$10^{-9}$ & 77.6 (0.0) / 95.2 (0.0) & 75.1 (0.0) / 98.0 (0.0) & $10^{-5}$ & 76.1 (0.0) / 94.8 (0.0) & 72.3 (0.0) / 98.9 (0.0) \\
\bottomrule
\end{tabular}
}
\end{table}

\section{Conclusions}
While rare-event estimation has been long studied, rare-event optimization, which is important for safe AI training and decision-making, has been largely open. To this end, a major challenge arises from the breakdown of the established rare-event analysis framework that translates variance-based efficiency criteria into sampling complexity. Our paper provides a first bridge on this front, by eliciting both negative and positive results: our negative result dictates that adaptive variance reduction, when naturally embedded into SGD, does not suffice to guarantee optimization efficiency. On the other hand, our positive result reveals that, with our new notion of \emph{safe start} that initializes a solution with low risk level, a variety of adaptive variance-reduced SGD configurations become efficient. These results are drawn from our insights on the ultra-sensitivity of the objective landscape arising from the extreme-risk term, which renders SGD step size being either too small and hence exhibit slow convergence, or risking catastrophic big jumps, unless safe start is imposed to stabilize the landscape. Experiments on extreme quantile estimation, CVaR portfolio optimization, and robust neural network training corroborate our theory. These also show that safe start is simple to configure and crucial in rare-event optimization.




\bibliographystyle{informs2014} 
\bibliography{reference}

\newpage 


\ECSwitch
\ECHead{Proofs of Statements}

\section{Proofs of Lemmas \ref{lem:safe-region-regularity} and \ref{lemma:subexponential-constants}.}
\label{appendix-sec:prelim}

We divide the proof of Lemma \ref{lem:safe-region-regularity} into three parts: Lemma \ref{lemma:bounded-distance-to-optimal} in Appendix \ref{subsection:bounded}, Lemma \ref{lemma:smooth-over-sublevel} in Appendix \ref{subsec:lipschitz}, and Lemma \ref{lemma:second-moment} in Appendix \ref{subsec:second-moment}. Then, we provide the proof of Lemma \ref{lemma:subexponential-constants} in Appendix \ref{subsec:expo}.

\subsection{Boundedness.}
\label{subsection:bounded}

To prove the bounded distance to the optimal solution $\mathbf{x}^*(\lambda)$ in Lemma \ref{lemma:bounded-distance-to-optimal}, we first show that the distance to to the average-case minimizer $\tilde{\mathbf{x}}^*$ 
is bounded (Lemma \ref{lemma:bounded-sublevel}); then, we show that the optimal solution $\mathbf{x}^*(\lambda)$ exists (Lemma \ref{lemma:existence}).

\begin{lemma}[Bounded Sublevel Set]\label{lemma:bounded-sublevel}
Under Assumptions \ref{assumption:smooth-connvex}--\ref{assumption:lower-bounded}, let
$d_f, \epsilon_f,$ and $\tilde{\mathbf{x}}^*$ be as defined therein. Consider the initial sublevel set
$S := S(\mathbf{x}_0, \lambda, c)$. For all $\mathbf{x} \in S$, we have:
\begin{equation*}
    \|\mathbf{x} - \tilde{\mathbf{x}}^*\| \le D_S,
    \qquad \text{where } D_S := \frac{c d_f F(\mathbf{x}_0; \lambda)}{\epsilon_f} + d_f.
\end{equation*}
\end{lemma}

\begin{proof}{Proof.}
For any $\mathbf{x} \in S$, we consider two cases.

\textbf{Case 1:} $f(\mathbf{x}) \le \epsilon_f$. Assumption \ref{assumption:level-bounded} immediately
implies $\|\mathbf{x} - \tilde{\mathbf{x}}^*\| \le d_f \le D_S$.

\textbf{Case 2:} $f(\mathbf{x}) > \epsilon_f$. Since $f$ is convex over the convex domain $\mathcal{X}$,
it is continuous, and the segment between $\mathbf{x}$ and $\tilde{\mathbf{x}}^*$ lies in $\mathcal{X}$.
By the Intermediate Value Theorem there exists $\tau \in (0,1)$ such that for
$\mathbf{x}_\tau := \tau \tilde{\mathbf{x}}^* + (1-\tau)\mathbf{x}$ we have $f(\mathbf{x}_\tau) = \epsilon_f$,
with $1-\tau = \frac{\|\mathbf{x}_\tau - \tilde{\mathbf{x}}^*\|}{\|\mathbf{x} - \tilde{\mathbf{x}}^*\|}$.
By convexity of $f$,
\begin{equation*}
    \tau f(\tilde{\mathbf{x}}^*) + (1-\tau)f(\mathbf{x}) \ge f(\mathbf{x}_\tau) = \epsilon_f.
\end{equation*}
Under Assumption \ref{assumption:lower-bounded}, $f(\tilde{\mathbf{x}}^*) = 0$, so
\begin{equation*}
    (1-\tau)f(\mathbf{x}) \ge \epsilon_f
    \implies \frac{\|\mathbf{x}_\tau - \tilde{\mathbf{x}}^*\|}{\|\mathbf{x} - \tilde{\mathbf{x}}^*\|} f(\mathbf{x}) \ge \epsilon_f.
\end{equation*}
Because $f(\mathbf{x}_\tau) = \epsilon_f$, Assumption \ref{assumption:level-bounded} gives
$\|\mathbf{x}_\tau - \tilde{\mathbf{x}}^*\| \le d_f$, hence
$\frac{d_f}{\|\mathbf{x} - \tilde{\mathbf{x}}^*\|} f(\mathbf{x}) \ge \epsilon_f$. Since $\mathbf{x} \in S$ and
$p(\cdot) \ge 0$, $f(\mathbf{x}) \le F(\mathbf{x}; \lambda) \le c F(\mathbf{x}_0; \lambda)$, so
\begin{equation*}
    \|\mathbf{x} - \tilde{\mathbf{x}}^*\| \le \frac{c d_f F(\mathbf{x}_0; \lambda)}{\epsilon_f} \le D_S.
\end{equation*}
Combining the two cases gives the bound for all $\mathbf{x} \in S$.
\Halmos \end{proof}

\begin{lemma}[Existence of an Optimal Solution]\label{lemma:existence}
Under the assumptions of Lemma \ref{lemma:bounded-sublevel}, if $\mathcal{X}$ is
non-empty, there exists an optimal solution $\mathbf{x}^*(\lambda) \in \mathcal{X}$
that minimizes $F(\mathbf{x}; \lambda)$ over $\mathcal{X}$.
\end{lemma}

\begin{proof}{Proof.}
Consider $S := S(\mathbf{x}_0, \lambda, c)$ for an arbitrary $\mathbf{x}_0 \in \mathcal{X}$. Because $F$ is
convex and continuous, $S$ is non-empty and closed. By Lemma \ref{lemma:bounded-sublevel},
$S \subseteq \prtc{ \mathbf{x} \in \mathcal{X} \mid \|\mathbf{x} - \tilde{\mathbf{x}}^*\| \le D_S }$, so $S$
is bounded, hence compact in $\mathbb{R}^n$. By the Weierstrass Extreme Value Theorem, $F$ attains its
minimum on $S$, ensuring the existence of an optimal solution $\mathbf{x}^*(\lambda) \in \mathcal{X}$.
\Halmos \end{proof}

Finally, we bound the distance from any point in the sublevel set to the extreme-risk optimum.

\begin{lemma}[Bounded Optimality Radius]\label{lemma:bounded-distance-to-optimal}
Under the assumptions of Lemma \ref{lemma:bounded-sublevel}, let $D_S$ be as therein.
Consider the initial sublevel set $S := S(\mathbf{x}_0, \lambda, c)$. For all
$\mathbf{x} \in S$, we have $\|\mathbf{x} - \mathbf{x}^*(\lambda)\| \le B_S$, where $B_S := 2 D_S$.
\end{lemma}

\begin{proof}{Proof.}
By Lemma \ref{lemma:existence}, an optimal solution $\mathbf{x}^*(\lambda)$ exists. Since
$\mathbf{x}^*(\lambda)$ minimizes $F(\cdot; \lambda)$ and $c \ge 1$, $F(\mathbf{x}^*(\lambda); \lambda) \le
c F(\mathbf{x}_0; \lambda)$, so $\mathbf{x}^*(\lambda) \in S$. Applying the triangle inequality and
Lemma \ref{lemma:bounded-sublevel} to both $\mathbf{x}$ and $\mathbf{x}^*(\lambda)$,
\begin{equation*}
    \|\mathbf{x} - \mathbf{x}^*(\lambda)\|
    \le \|\mathbf{x} - \tilde{\mathbf{x}}^*\| + \|\tilde{\mathbf{x}}^* - \mathbf{x}^*(\lambda)\|
    \le D_S + D_S = 2D_S = B_S.
\end{equation*}
\Halmos \end{proof}

\subsection{Lipschitz Smoothness.}
\label{subsec:lipschitz}


\begin{lemma}[Lipschitz Smoothness over Sublevel Set]\label{lemma:smooth-over-sublevel}
Under the assumptions of Lemma \ref{lemma:bounded-sublevel} together with
Assumption \ref{assumption:scale-grad-hessian}, let $L_f$, $L_{p,2}$, and $\beta_2$ be as
defined therein, and let $D_S$ be as in Lemma \ref{lemma:bounded-sublevel}. Consider the
initial sublevel set $S(\mathbf{x}_0, \lambda, c)$. For all $\mathbf{x} \in S$,
\begin{align*}
    \|\nabla^2 F(\mathbf{x})\| \le L_S,
\end{align*}
 where
$L_S := L_f + c L_{p,2} \prtr{1 + D_S^{\beta_2}} F(\mathbf{x}_0; \lambda)$ is the local
Lipschitz constant of $\nabla F$ over $S$. That is, $F$ is $L_S$-smooth over $S$.
\end{lemma}

\begin{proof}{Proof.}
By the triangle inequality, we have:
\begin{align*}
    \sup_{\mathbf{x}\in S} \|\nabla^2 F(\mathbf{x})\| 
    &= \sup_{\mathbf{x}\in S} \|\nabla^2 f(\mathbf{x}) + \gamma(\lambda) \nabla^2 p(\mathbf{x})\| \\
    &\le \sup_{\mathbf{x}\in S} \|\nabla^2 f(\mathbf{x})\| + \sup_{\mathbf{x}\in S} \prtr{ \gamma(\lambda) \|\nabla^2 p(\mathbf{x})\| }
\end{align*}
Under Assumptions \ref{assumption:smooth-connvex} and \ref{assumption:scale-grad-hessian}, this becomes:
\begin{equation*}
    \sup_{\mathbf{x}\in S} \|\nabla^2 F(\mathbf{x})\| \le L_f + \sup_{\mathbf{x}\in S} \prtr{ \gamma(\lambda) L_{p,2} \prtr{1 + \norm{\mathbf{x} - \tilde{\mathbf{x}}^*}^{\beta_2}} p(\mathbf{x}) }
\end{equation*}
By Lemma \ref{lemma:bounded-sublevel}, every $\mathbf{x} \in S$ satisfies $\norm{\mathbf{x} - \tilde{\mathbf{x}}^*} \le D_S$, hence $1 + \norm{\mathbf{x} - \tilde{\mathbf{x}}^*}^{\beta_2} \le 1 + D_S^{\beta_2}$. Moreover, under Assumption \ref{assumption:lower-bounded}, $f(\mathbf{x}) \ge 0$ gives $\gamma(\lambda) p(\mathbf{x}) \le F(\mathbf{x}; \lambda)$, and because $\mathbf{x} \in S$, $F(\mathbf{x};\lambda) \le c F(\mathbf{x}_0;\lambda)$. Combining these two bounds,
\begin{equation*}
    \sup_{\mathbf{x}\in S} \prtr{ \gamma(\lambda) L_{p,2} \prtr{1 + \norm{\mathbf{x} - \tilde{\mathbf{x}}^*}^{\beta_2}} p(\mathbf{x}) } \le c L_{p,2} \prtr{1 + D_S^{\beta_2}} F(\mathbf{x}_0; \lambda).
\end{equation*}
Therefore,
\begin{equation*}
    \sup_{\mathbf{x}\in S} \|\nabla^2 F(\mathbf{x})\| \le L_f + c L_{p,2} \prtr{1 + D_S^{\beta_2}} F(\mathbf{x}_0; \lambda).
\end{equation*}
\Halmos \end{proof}

\subsection{Bounded Variance.}
\label{subsec:second-moment}

\begin{lemma}[Second-Moment Bounds on Stochastic Gradients and Noise]\label{lemma:second-moment}
Under Assumptions \ref{assumption:second-moment} and \ref{assumption:scale-grad-hessian}
together with the assumptions of Lemma \ref{lemma:bounded-sublevel}, let
$G_f(\mathbf{x},\xi)$, $G_p(\mathbf{x},\xi; \lambda)$, $\sigma_f$, $\sigma_p(\lambda)$,
$L_{p,1}$, and $\beta_1$ be as defined therein, and let $D_S$ be as in
Lemma \ref{lemma:bounded-sublevel}, for any rarity level $\lambda >0$. Consider the initial
sublevel set $S(\mathbf{x}_0, \lambda, c)$. Denote the stochastic gradient estimator and the
stochastic noise for $\nabla F(\mathbf{x})$ as $G(\mathbf{x},\xi)$ and $\mathbf{w}(\mathbf{x},\xi)$, respectively, where:
\begin{align*}
    G(\mathbf{x},\xi) &:= G_f(\mathbf{x},\xi) + \gamma(\lambda) G_p(\mathbf{x},\xi; \lambda) \\
    \mathbf{w}(\mathbf{x},\xi) &:= G(\mathbf{x},\xi) - \nabla F(\mathbf{x})
\end{align*}
Then, for all $\mathbf{x} \in S$, we have:
\begin{equation*}
    \mathbb{E}\prts{\| \mathbf{w}(\mathbf{x}, \xi) \|^2} \le \sigma_S^2(\lambda)
\end{equation*}
and 
\begin{equation*}
    \mathbb{E}\prts{\|G(\mathbf{x}, \xi)\|^2} \le \|\nabla F(\mathbf{x})\|^2 + \sigma_S^2(\lambda)
\end{equation*}
where $\sigma_S(\lambda) := \sigma_f + c \sigma_p(\lambda) L_{p,1} \prtr{1 + D_S^{\beta_1}} F(\mathbf{x}_0; \lambda)$ is a single-sample sublevel variance bound.
\end{lemma}
\begin{proof}{Proof.}
By Minkowski's inequality, we have:
\begin{align*}
    \mathbb{E}\prts{\| \mathbf{w}(\mathbf{x},\xi) \|^2} 
    &= \mathbb{E}\prts{\big\| \prtr{G_f(\mathbf{x},\xi) - \nabla f(\mathbf{x})} + \gamma(\lambda) \prtr{G_p(\mathbf{x},\xi; \lambda) - \nabla p(\mathbf{x})} \big\|^2} \\
    &\le \prtr{ \sqrt{\mathbb{E}\prts{\big\| G_f(\mathbf{x},\xi) - \nabla f(\mathbf{x}) \big\|^2}} + \gamma(\lambda) \sqrt{\mathbb{E}\prts{\big\| G_p(\mathbf{x},\xi; \lambda) - \nabla p(\mathbf{x}) \big\|^2}} }^2
\end{align*}
Under Assumption \ref{assumption:second-moment}, this simplifies to:
\begin{equation*}
    \mathbb{E}\prts{\| \mathbf{w}(\mathbf{x},\xi) \|^2} \le \prtr{ \sigma_f + \gamma(\lambda) \sigma_p(\lambda) \|\nabla p(\mathbf{x})\| }^2
\end{equation*}
By Assumption \ref{assumption:scale-grad-hessian}, $\|\nabla p(\mathbf{x})\| \le L_{p,1} \prtr{1 + \norm{\mathbf{x} - \tilde{\mathbf{x}}^*}^{\beta_1}} p(\mathbf{x})$. By Lemma \ref{lemma:bounded-sublevel}, every $\mathbf{x} \in S$ satisfies $\norm{\mathbf{x} - \tilde{\mathbf{x}}^*} \le D_S$, hence $1 + \norm{\mathbf{x} - \tilde{\mathbf{x}}^*}^{\beta_1} \le 1 + D_S^{\beta_1}$. Moreover, under Assumption \ref{assumption:lower-bounded}, $f(\mathbf{x}) \ge 0$ gives $\gamma(\lambda) p(\mathbf{x}) \le F(\mathbf{x}; \lambda)$, and because $\mathbf{x} \in S$, $F(\mathbf{x};\lambda) \le c F(\mathbf{x}_0;\lambda)$. Combining these bounds,
\begin{equation*}
    \gamma(\lambda) \|\nabla p(\mathbf{x})\| \le L_{p,1} \prtr{1 + D_S^{\beta_1}} \gamma(\lambda) p(\mathbf{x}) \le c L_{p,1} \prtr{1 + D_S^{\beta_1}} F(\mathbf{x}_0; \lambda)
\end{equation*}
Substituting this bound yields:
\begin{equation*}
    \mathbb{E}\prts{\| \mathbf{w}(\mathbf{x},\xi) \|^2} \le \prtr{ \sigma_f + c\sigma_p(\lambda) L_{p,1} \prtr{1 + D_S^{\beta_1}} F(\mathbf{x}_0; \lambda) }^2 = \sigma_S^2(\lambda)
\end{equation*}
Furthermore, since $G(\mathbf{x}, \xi)$ is a conditionally unbiased estimator, $\mathbb{E}\prts{\mathbf{w}(\mathbf{x}, \xi)^\top \nabla F(\mathbf{x})} = 0$. Thus:
\begin{align*}
\mathbb{E}\prts{\|G(\mathbf{x}, \xi)\|^2} 
    &= \mathbb{E}\prts{\|\mathbf{w}(\mathbf{x}, \xi) + \nabla F(\mathbf{x})\|^2} \\
    &= \mathbb{E}\prts{\|\mathbf{w}(\mathbf{x}, \xi)\|^2} + \|\nabla F(\mathbf{x})\|^2 + 2\mathbb{E}\prts{\mathbf{w}(\mathbf{x}, \xi)^\top \nabla F(\mathbf{x})} \\
    &\le \sigma_S^2(\lambda) + \|\nabla F(\mathbf{x})\|^2.
\end{align*}
\Halmos \end{proof}

\subsection{Proof of Lemma \ref{lemma:subexponential-constants}.}

\label{subsec:expo}



\begin{proof}{Proof. }
We first prove that the initial objective value $F(\mathbf{x}_0(\lambda); \lambda)$ grows sub-exponentially. 
Since $f$ is $L_f$-smooth and $f(\tilde{\mathbf{x}}^*) = 0$ (Assumption \ref{assumption:lower-bounded}), the quadratic upper bound and Cauchy-Schwarz inequality yield:
\begin{align*}
    f(\mathbf{x}_0(\lambda)) 
    &\le f(\tilde{\mathbf{x}}^*) + \nabla f(\tilde{\mathbf{x}}^*)^\top \prtr{\mathbf{x}_0(\lambda) - \tilde{\mathbf{x}}^*} + \frac{L_f}{2}\|\mathbf{x}_0(\lambda) - \tilde{\mathbf{x}}^*\|^2 \\
    &\le \|\nabla f(\tilde{\mathbf{x}}^*)\| \|\mathbf{x}_0(\lambda) - \tilde{\mathbf{x}}^*\| + \frac{L_f}{2}\|\mathbf{x}_0(\lambda) - \tilde{\mathbf{x}}^*\|^2
\end{align*}
Because $\tilde{\mathbf{x}}^*$ is a fixed point in $\mathcal{X}$, its gradient norm $\|\nabla f(\tilde{\mathbf{x}}^*)\|$ is a finite constant. By Assumption \ref{assumption:init-distance}, the distance $\|\mathbf{x}_0(\lambda) - \tilde{\mathbf{x}}^*\|$ grows sub-exponentially. The asymptotic exponential growth rate of a sum is bounded by the maximum growth rate of its terms:
\begin{align*}
    &\limsup_{\lambda\to\infty}\frac{1}{\lambda}\log f(\mathbf{x}_0(\lambda)) \\
    &\le \max \prtc{ \limsup_{\lambda\to\infty}\frac{1}{\lambda}\log \prtr{\|\nabla f(\tilde{\mathbf{x}}^*)\| \|\mathbf{x}_0(\lambda) - \tilde{\mathbf{x}}^*\|}, \limsup_{\lambda\to\infty}\frac{1}{\lambda}\log \prtr{\frac{L_f}{2}\|\mathbf{x}_0(\lambda) - \tilde{\mathbf{x}}^*\|^2} } \\
    &= 0
\end{align*}
Next, by Assumption \ref{assumption:ldp-optimization} and the safe start condition (Assumption \ref{def:safe-start}), the extreme risk term is bounded by:
\begin{align*}
    \limsup_{\lambda\to\infty}\frac{1}{\lambda}\log \prtr{\gamma(\lambda) p(\mathbf{x}_0(\lambda))} 
    &\le \limsup_{\lambda\to\infty}\frac{1}{\lambda}\log \gamma(\lambda) + \limsup_{\lambda\to\infty}\frac{1}{\lambda}\log p(\mathbf{x}_0(\lambda)) \\
    &\le I + (-I) = 0
\end{align*}
Since $F(\mathbf{x}_0(\lambda); \lambda) = f(\mathbf{x}_0(\lambda)) + \gamma(\lambda)p(\mathbf{x}_0(\lambda))$, the exponential growth rate of the combined objective is bounded by the maximum growth rate of its components:
\begin{align*}
    \limsup_{\lambda\to\infty}\frac{1}{\lambda}\log F(\mathbf{x}_0(\lambda); \lambda) 
    &\le \max \prtc{ \limsup_{\lambda\to\infty}\frac{1}{\lambda}\log f(\mathbf{x}_0(\lambda)), \limsup_{\lambda\to\infty}\frac{1}{\lambda}\log \prtr{\gamma(\lambda) p(\mathbf{x}_0(\lambda))} } \\
    &\le 0
\end{align*}
Finally, Lemma \ref{lem:safe-region-regularity}  gives the sublevel set constants as
\begin{equation*}
    L_S(\lambda) = L_f + c L_{p,2}\prtr{1 + D_S^{\beta_2}} F(\mathbf{x}_0(\lambda); \lambda), \quad
    \sigma_S(\lambda) = \sigma_f + c \sigma_p(\lambda) L_{p,1}\prtr{1 + D_S^{\beta_1}} F(\mathbf{x}_0(\lambda); \lambda),
\end{equation*}
and $B_S(\lambda) = 2 D_S(\lambda)$, where $D_S(\lambda) = c d_f F(\mathbf{x}_0(\lambda);\lambda)/\epsilon_f + d_f$ is linear in $F(\mathbf{x}_0(\lambda);\lambda)$. Every prefactor here is a $\lambda$-independent constant ($c, L_f, L_{p,1}, L_{p,2}, \sigma_f, d_f, \epsilon_f, \beta_1, \beta_2$), with the sole exception of $\sigma_p(\lambda)$, which grows at most sub-exponentially by Assumption \ref{assumption:second-moment}. Because $D_S$ is linear in $F(\mathbf{x}_0(\lambda);\lambda)$, the factors $D_S^{\beta_1}$ and $D_S^{\beta_2}$ render $L_S(\lambda)$ and $\sigma_S(\lambda)$ \emph{polynomial} in $F(\mathbf{x}_0(\lambda);\lambda)$ (while $B_S(\lambda)$ is linear). Since $F(\mathbf{x}_0(\lambda);\lambda)$ grows sub-exponentially, so does any polynomial of it; the extra factor $\sigma_p(\lambda)$ in $\sigma_S(\lambda)$ is itself sub-exponential, so the product remains sub-exponential. Consequently, all three constants inherit the sub-exponential growth rate of $F(\mathbf{x}_0(\lambda);\lambda)$. \Halmos

\end{proof}
\newpage
\section{Proofs of Lemmas \ref{lemma:descent-lemma-last} to \ref{lemma:last-iterate-configuration} and of Proposition \ref{lemma:last-iterate-sufficiency}.}
\label{appendix-sec:last-iterate}







\begin{proof}{Proof of Lemma \ref{lemma:descent-lemma-last}.}
Let $\mathbf{w}_t$ be the conditionally unbiased stochastic noise at $\mathbf{x}_t$, defined as $\mathbf{w}(\mathbf{x}_t,\xi) $ in Lemma \ref{assumption:second-moment}. By Markov's inequality, we have:
\begin{equation*}
    \mathbb{P}\prtr{\|\mathbf{w}_t\|^2 \ge \min \prtc{\frac{\varepsilon F^* L_S}{4}, d_f^2 L_S^2, \frac{\varepsilon^2 (F^*)^2}{16 B_S^2}}} \le \kappa
\end{equation*}
Therefore, with probability at least $1-\kappa$, the strict complement of this event holds. We condition the remainder of this proof on this bounded noise event.

The SGD update is:
\begin{equation*}
    \mathbf{x}_{t+1} := \mathbf{x}_t - \eta_t G(\mathbf{x}_t, \xi) = \mathbf{x}_t - \eta_t \prtr{\nabla F(\mathbf{x}_t) + \mathbf{w}_t}
\end{equation*}
We claim that given the bounded noise event and $\mathbf{x}_t \in \text{int}(S)$, the subsequent iterate satisfies $\mathbf{x}_{t+1} \in \text{int}(S)$. Define the line segment between $\mathbf{x}_t$ and $\mathbf{x}_{t+1}$:
\begin{equation*}
    \mathbf{x}(\tau) := \mathbf{x}_t + \tau \prtr{\mathbf{x}_{t+1} - \mathbf{x}_t}, \qquad \tau \in [0,1]
\end{equation*}
Assume for contradiction that $\mathbf{x}_{t+1} \notin \text{int}(S)$. Let $\tau^* := \inf \prtc{\tau \in [0,1] : \mathbf{x}(\tau) \in \partial S}$. Because the segment between $\mathbf{x}_t$ and $\mathbf{x}(\tau^*)$ is entirely contained within $S$, Lemma \ref{lemma:smooth-over-sublevel} guarantees $L_S$-smoothness. The quadratic upper bound yields:
\begin{align*}
    F(\mathbf{x}(\tau^*)) 
    &\le F(\mathbf{x}_t) + \prta{\nabla F(\mathbf{x}_t), \mathbf{x}(\tau^*) - \mathbf{x}_t} + \frac{L_S}{2}\|\mathbf{x}(\tau^*) - \mathbf{x}_t\|^2 \\
    &= F(\mathbf{x}_t) - \eta_t \tau^* \prtr{1 - \frac{\eta_t \tau^* L_S}{2}} \|\nabla F(\mathbf{x}_t)\|^2 \\
    &\quad - \eta_t \tau^* \prtr{1 - \eta_t \tau^* L_S} \nabla F(\mathbf{x}_t)^\top \mathbf{w}_t + \frac{\eta_t^2 (\tau^*)^2 L_S}{2} \|\mathbf{w}_t\|^2
\end{align*}
Applying the Cauchy-Schwarz inequality provides:
\begin{align}
    F(\mathbf{x}(\tau^*)) 
    &\le F(\mathbf{x}_t) - \eta_t \tau^* \prtr{1 - \frac{\eta_t \tau^* L_S}{2}} \|\nabla F(\mathbf{x}_t)\|^2 \nonumber \\
    &\quad + \eta_t \tau^* \prtr{1 - \eta_t \tau^* L_S} \|\nabla F(\mathbf{x}_t)\| \|\mathbf{w}_t\| + \frac{\eta_t^2 (\tau^*)^2 L_S}{2} \|\mathbf{w}_t\|^2 \label{eq:positive-constant-quadratic}
\end{align}
We evaluate this bound under two conditions:

\textbf{Case 1: $\|\nabla F(\mathbf{x}_t)\| \le \frac{\varepsilon F^*}{2 B_S}$.} 
Since $\tau^* \le 1$, $\eta_t \le \frac{1}{L_S}$, $1 - \frac{\eta_t \tau^* L_S}{2} \ge 0$, and $\sup_{y \in \mathbb{R}} y(1 - y L_S) \le \frac{1}{4 L_S}$, we refine \eqref{eq:positive-constant-quadratic} as:
\begin{equation*}
    F(\mathbf{x}(\tau^*)) \le F(\mathbf{x}_t) + \frac{1}{4 L_S} \|\nabla F(\mathbf{x}_t)\| \|\mathbf{w}_t\| + \frac{1}{2 L_S} \|\mathbf{w}_t\|^2
\end{equation*}
Given the bounded noise event and $B_S \ge d_f$, we substitute the maximum values for $\|\mathbf{w}_t\|$ relative to each term:
\begin{align*}
    F(\mathbf{x}(\tau^*)) 
    &\le F(\mathbf{x}_t) + \frac{1}{4 L_S} \prtr{\frac{\varepsilon F^*}{2 B_S}} \prtr{d_f L_S} + \frac{1}{2 L_S} \prtr{\frac{\varepsilon F^* L_S}{4}} \\
    &\le F(\mathbf{x}_t) + \frac{\varepsilon F^*}{4}
\end{align*}
By the convexity of $F$ and the Cauchy-Schwarz inequality:
\begin{equation*}
    \|\nabla F(\mathbf{x}_t)\| \|\mathbf{x}_t - \mathbf{x}^*\| \ge \prta{\nabla F(\mathbf{x}_t), \mathbf{x}_t - \mathbf{x}^*} \ge F(\mathbf{x}_t) - F^*
\end{equation*}
Because $\mathbf{x}_t \in \text{int}(S)$, Lemma \ref{lemma:bounded-distance-to-optimal} ensures $\|\mathbf{x}_t - \mathbf{x}^*\| \le B_S$. Therefore:
\begin{equation*}
    \prtr{\frac{\varepsilon F^*}{2 B_S}} B_S \ge F(\mathbf{x}_t) - F^* \implies F(\mathbf{x}_t) \le \prtr{1 + \frac{\varepsilon}{2}} F^*
\end{equation*}
Combining these bounds yields:
\begin{equation*}
    F(\mathbf{x}(\tau^*)) \le \prtr{1 + \frac{3\varepsilon}{4}} F^* < F(\mathbf{x}_0)
\end{equation*}
This contradicts $\mathbf{x}(\tau^*) \in \partial S$, as the boundary requires $F(\mathbf{x}) = 2F(\mathbf{x}_0)$. Thus, $\mathbf{x}_{t+1} \in \text{int}(S)$. Applying the same derivation for $\tau^* = 1$ concludes that $F(\mathbf{x}_{t+1}) \le (1 + 3\varepsilon/4) F^* < (1+\varepsilon)F^*$.

\textbf{Case 2: $\|\nabla F(\mathbf{x}_t)\| > \frac{\varepsilon F^*}{2 B_S}$.} 
The bounded noise event guarantees $\|\mathbf{w}_t\|^2 \le \frac{\varepsilon^2 (F^*)^2}{16 B_S^2} \le \frac{1}{4}\|\nabla F(\mathbf{x}_t)\|^2$, and thus $\|\mathbf{w}_t\| \le \frac{1}{2}\|\nabla F(\mathbf{x}_t)\|$. Substituting this into \eqref{eq:positive-constant-quadratic} yields:
\begin{align*}
    F(\mathbf{x}(\tau^*)) 
    &\le F(\mathbf{x}_t) - \eta_t \tau^* \prtr{1 - \frac{\eta_t \tau^* L_S}{2}} \|\nabla F(\mathbf{x}_t)\|^2 \\
    &\quad + \frac{\eta_t \tau^* \prtr{1 - \eta_t \tau^* L_S}}{2} \|\nabla F(\mathbf{x}_t)\|^2 + \frac{\eta_t^2 (\tau^*)^2 L_S}{8} \|\nabla F(\mathbf{x}_t)\|^2 \\
    &= F(\mathbf{x}_t) - \frac{\eta_t \tau^*}{2} \|\nabla F(\mathbf{x}_t)\|^2 + \frac{\eta_t^2 (\tau^*)^2 L_S}{8} \|\nabla F(\mathbf{x}_t)\|^2
\end{align*}
By the continuity of $F$ and $\mathbf{x}_t \in \text{int}(S)$, we know $\tau^* > 0$. With $\tau^* \le 1$ and $\eta_t L_S \le 1$, we simplify:
\begin{equation*}
    F(\mathbf{x}(\tau^*)) \le F(\mathbf{x}_t) - \frac{3 \eta_t \tau^*}{8} \|\nabla F(\mathbf{x}_t)\|^2 < F(\mathbf{x}_t) \le 2F(\mathbf{x}_0)
\end{equation*}
This strictly contradicts $\mathbf{x}(\tau^*) \in \partial S$. Therefore, $\mathbf{x}_{t+1} \in \text{int}(S)$. Applying this analysis for $\tau^*=1$ ensures descent:
\begin{equation*}
    F(\mathbf{x}_{t+1}) \le F(\mathbf{x}_t) - \frac{3 \eta_t}{8} \|\nabla F(\mathbf{x}_t)\|^2 \le F(\mathbf{x}_t) - \frac{3 \eta_t \varepsilon^2 (F^*)^2}{32 B_S^2}
\end{equation*}
We conclude that with probability at least $1-\kappa$, $\mathbf{x}_{t+1} \in \text{int}(S)$, and either $F(\mathbf{x}_{t+1}) \le (1+\varepsilon)F^*$ or $F(\mathbf{x}_{t+1}) \le F(\mathbf{x}_t) - \frac{3\eta_t \varepsilon^2 (F^*)^2}{32 B_S^2}$.
\Halmos \end{proof}



\begin{proof}{Proof of Lemma \ref{lemma:upperbound-iteration}.}
By Lemma \ref{lemma:descent-lemma-last}, if $\mathbf{x}_t \in \text{int}(S)$, then with probability at least $1 - \frac{\kappa}{T}$, the subsequent iterate satisfies $\mathbf{x}_{t+1} \in \text{int}(S)$ and either reaches the target accuracy $F(\mathbf{x}_{t+1}) \le (1+\varepsilon)F^*$ or strictly descends by $F(\mathbf{x}_{t+1}) \le F(\mathbf{x}_t) - \frac{3\eta_t \varepsilon^2 (F^*)^2}{32 B_S^2}$.

By construction, $\mathbf{x}_0 \in \text{int}(S)$. Applying a union bound over the $T$ iterations guarantees that the trajectory never escapes the sublevel set and the descent condition holds at every step with probability at least $1-\kappa$. We condition the remainder of the proof on this joint success event.

If the trajectory reaches the target accuracy at any step $\tau \le T$, the descent condition ensures that all subsequent iterates either remain at the target accuracy or descend further. Consequently, the final iterate satisfies $F(\mathbf{x}_{T}) \le (1+\varepsilon)F^*$.

Assume for contradiction that the trajectory never reaches the target accuracy. Then, the strict descent bound must hold at every iteration. Telescoping the objective decrease from $t=0$ to $T-1$ yields:
\begin{align*}
    F(\mathbf{x}_{T}) \le F(\mathbf{x}_0) - \sum_{t=0}^{T-1} \frac{3\eta_t \varepsilon^2 (F^*)^2}{32 B_S^2}
\end{align*}
By the assumed condition on the sum of step sizes, we have:
\begin{equation*}
    \sum_{t=0}^{T-1} \frac{3\eta_t \varepsilon^2 (F^*)^2}{32 B_S^2} = \frac{3 \varepsilon^2 (F^*)^2}{32 B_S^2} \sum_{t=0}^{T-1} \eta_t > F(\mathbf{x}_0)
\end{equation*}
Substituting this strict lower bound into the telescoping sum implies $F(\mathbf{x}_{T}) < 0$. This contradicts the non-negativity of $F$ established in Assumption \ref{assumption:lower-bounded}. Therefore, the trajectory must achieve $F(\mathbf{x}_{T}) \le (1+\varepsilon)F^*$.
\Halmos \end{proof}




\begin{proof}{Proof of Lemma \ref{lemma:last-iterate-configuration}.}
Fix $\lambda > 0$. We first note that the sum of step sizes satisfies:
\begin{align*}
    \sum_{t=0}^{T(\lambda)-1}\eta_t(\lambda) 
    &\ge \frac{1}{20L_S(\lambda)}\sum_{t=0}^{T(\lambda)-1}\frac{1}{(t+1)^\alpha} \\
    &\ge \frac{1}{20L_S(\lambda)}\frac{T(\lambda)}{T(\lambda)^\alpha} \\
    &= \frac{1}{20L_S(\lambda)} T(\lambda)^{1-\alpha} \\
    &\ge \frac{32 B_S(\lambda)^2 F(\mathbf{x}_0(\lambda); \lambda)}{ 3\varepsilon^2 (F^*(\lambda))^2}
\end{align*}

Next, by standard mini-batching, the variance bound for the $B(\lambda)$-minibatch SGD update reduces to:
\begin{equation*}
    \sigma(\lambda)^2 := \frac{\sigma_S(\lambda)^2}{B(\lambda)}
\end{equation*}
where $\sigma_S(\lambda)^2$ is the single-sample sublevel variance bound. Using the fact that $\frac{1}{\min(a,b,c)} \le \frac{1}{a} + \frac{1}{b} + \frac{1}{c}$, substituting the defined $B(\lambda)$ guarantees:
\begin{equation*}
    \sigma(\lambda)^2 \le \frac{\kappa}{T(\lambda)}\min \prtc{\frac{\varepsilon F^*(\lambda)L_S(\lambda)}{4}, d_f^2 L_S(\lambda)^2, \frac{\varepsilon^2 (F^*(\lambda))^2}{16 B_S(\lambda)^2}}
\end{equation*}

Thus, all conditions of Lemma \ref{lemma:upperbound-iteration} are satisfied. We conclude that $F(\mathbf{x}_{T(\lambda)}; \lambda) \le (1+\varepsilon)F^*(\lambda)$ with probability at least $1-\kappa$.
\Halmos \end{proof}



\begin{proof}{Proof of Proposition \ref{lemma:last-iterate-sufficiency}.}
For each rarity level $\lambda > 0$, if $F(\mathbf{x}_0(\lambda); \lambda) \le (1+\varepsilon)F^*(\lambda)$, the initial solution already satisfies the target accuracy, resulting in a sample complexity of zero. 

Assume instead that $F(\mathbf{x}_0(\lambda); \lambda) > (1+\varepsilon)F^*(\lambda)$. By Lemma \ref{lemma:last-iterate-configuration}, there exists an SGD scheme $\mathcal{A}(\lambda) \in \mathbb{A}_{\alpha,\mathrm{last}}$ that achieves the target accuracy with a sample complexity of $T(\lambda)B(\lambda)$. Lemma \ref{lemma:subexponential-constants} establishes that the initial objective value $F(\mathbf{x}_0(\lambda); \lambda)$ and all sublevel set constants grow sub-exponentially in $\lambda$. Because the denominators in the definitions of $T(\lambda)$ and $B(\lambda)$ are bounded away from zero, the total sample complexity inherits this sub-exponential upper bound:
\begin{equation*}
    \limsup_{\lambda\to\infty} \frac{1}{\lambda} \log \prtr{T(\lambda) B(\lambda)} \le 0
\end{equation*}
Therefore, the sequence of configurations $\prtc{(\mathcal{A}(\lambda), \mathbf{x}_0(\lambda))}_{\lambda>0}$ is efficient.

Note that while it is unknown \textit{a priori} whether the initialization already meets the target accuracy, one can simply evaluate both the initial solution and the final iterate to guarantee this bound.
\Halmos \end{proof}
\newpage

\section{Proofs of Lemmas \ref{lemma:next-step-eog} to \ref{lemma:avg-iterate-configuration}, of Proposition \ref{lemma:avg-iterate-sufficiency}, and of Theorem \ref{thm:main}.}
%
\label{appendix-sec:avg-iterate}

\begin{proof}{Proof of Lemma \ref{lemma:next-step-eog}.}
In what follows, we condition on $\mathbb{I}_{\mathrm{safe}, t}=1$. Denote by $G(\mathbf{x}_t, \xi_t)$ the stochastic gradient evaluated at $\mathbf{x}_t$. We first establish the following almost-sure bound:
\begin{equation}\label{eq:smooth-sublevel-t-e}
    \prtr{F(\mathbf{x}_{t+1}) - F^*} \mathbb{I}_{\mathrm{safe}, t+1} 
    \le \prtr{F(\mathbf{x}_t) - F^*} - \eta_t \nabla F(\mathbf{x}_t)^\top G(\mathbf{x}_t, \xi_t) + \frac{\eta_t^2 L_S}{2} \|G(\mathbf{x}_t, \xi_t)\|^2
\end{equation}
We evaluate this bound under two conditions:

\textbf{Case 1: $\mathbb{I}_{\mathrm{safe}, t+1}=1$.} 
Both $\mathbf{x}_t$ and $\mathbf{x}_{t+1}$ lie inside the locally smooth region $S$. Applying the $L_S$-smoothness quadratic upper bound directly yields \eqref{eq:smooth-sublevel-t-e}.

\textbf{Case 2: $\mathbb{I}_{\mathrm{safe}, t+1}=0$.} 
Since the left-hand side of \eqref{eq:smooth-sublevel-t-e} evaluates to $0$, we only need to show that the right-hand side is non-negative. By the continuity of $F$, there exists a point $\mathbf{x}' := \mathbf{x}_t - \tau \eta_t G(\mathbf{x}_t, \xi_t)$ for some $\tau \in (0,1]$ that lies exactly on the boundary of $S$. Because both $\mathbf{x}_t$ and $\mathbf{x}'$ are within $S$, the smoothness condition implies:
\begin{equation*}
    F(\mathbf{x}') - F(\mathbf{x}_t) \le \eta_t \tau \prtr{ -\nabla F(\mathbf{x}_t)^\top G(\mathbf{x}_t, \xi_t) + \frac{\tau \eta_t L_S}{2} \|G(\mathbf{x}_t, \xi_t)\|^2 }
\end{equation*}
Because $\mathbf{x}_t \in \text{int}(S)$ and $\mathbf{x}' \in \partial S$, we know $F(\mathbf{x}') > F(\mathbf{x}_t)$, which implies the right-hand side is strictly positive. Dividing by $\tau \eta_t > 0$ and using the fact that $\tau \le 1$, we obtain:
\begin{equation*}
    -\nabla F(\mathbf{x}_t)^\top G(\mathbf{x}_t, \xi_t) + \frac{\eta_t L_S}{2} \|G(\mathbf{x}_t, \xi_t)\|^2 \ge 0
\end{equation*}
Adding the non-negative term $F(\mathbf{x}_t) - F^*$ to this inequality confirms the right-hand side of \eqref{eq:smooth-sublevel-t-e} is strictly positive, satisfying the bound.

Having established \eqref{eq:smooth-sublevel-t-e}, we take the conditional expectation with respect to $\xi_t$. Because $G(\mathbf{x}_t, \xi_t)$ is unbiased and its second moment is bounded by Lemma \ref{assumption:second-moment}, we have:
\begin{align*}
    &\mathbb{E}_{\xi_t} \prts{ \prtr{F(\mathbf{x}_{t+1}) - F^*} \mathbb{I}_{\mathrm{safe}, t+1} } \\
    &\le \prtr{F(\mathbf{x}_t) - F^*} - \eta_t \|\nabla F(\mathbf{x}_t)\|^2 + \frac{\eta_t^2 L_S}{2} \prtr{\sigma_S^2 + \|\nabla F(\mathbf{x}_t)\|^2}
\end{align*}
Using the step size condition $\eta_t \le \frac{1}{L_S}$, we simplify this to:
\begin{equation}\label{eq:bound-gap-eog}
    \mathbb{E}_{\xi_t} \prts{ \prtr{F(\mathbf{x}_{t+1}) - F^*} \mathbb{I}_{\mathrm{safe}, t+1} } 
    \le \prtr{F(\mathbf{x}_t) - F^*} - \frac{\eta_t}{2} \|\nabla F(\mathbf{x}_t)\|^2 + \frac{\eta_t \sigma_S^2}{2}
\end{equation}

Next, we bound the gradient norm using the distance to the optimal solution. Expanding the expected distance yields:
\begin{align*}
    \mathbb{E}_{\xi_t}\prts{\|\mathbf{x}_{t+1} - \mathbf{x}^*\|^2} 
    &= \mathbb{E}_{\xi_t}\prts{\|\mathbf{x}_t - \mathbf{x}^* - \eta_t G(\mathbf{x}_t, \xi_t)\|^2} \\
    &\le \|\mathbf{x}_t - \mathbf{x}^*\|^2 - 2\eta_t \nabla F(\mathbf{x}_t)^\top (\mathbf{x}_t - \mathbf{x}^*) + \eta_t^2 \prtr{\sigma_S^2 + \|\nabla F(\mathbf{x}_t)\|^2}
\end{align*}
By the convexity of $F$, we have $\nabla F(\mathbf{x}_t)^\top (\mathbf{x}_t - \mathbf{x}^*) \ge F(\mathbf{x}_t) - F^*$. Substituting this and rearranging the terms to isolate $\frac{1}{2\eta_t} \|\nabla F(\mathbf{x}_t)\|^2$ provides:
\begin{equation*}
    \frac{\eta_t}{2} \|\nabla F(\mathbf{x}_t)\|^2 \ge \frac{1}{2\eta_t} \prtr{ \mathbb{E}_{\xi_t}\prts{\|\mathbf{x}_{t+1} - \mathbf{x}^*\|^2} - \|\mathbf{x}_t - \mathbf{x}^*\|^2 } + \prtr{F(\mathbf{x}_t) - F^*} - \frac{\eta_t \sigma_S^2}{2}
\end{equation*}
Substituting this lower bound into \eqref{eq:bound-gap-eog} cancels the $F(\mathbf{x}_t) - F^*$ term exactly, yielding:
\begin{equation*}
    \mathbb{E}_{\xi_t} \prts{ \prtr{F(\mathbf{x}_{t+1}) - F^*} \mathbb{I}_{\mathrm{safe}, t+1} } 
    \le \eta_t \sigma_S^2 + \frac{1}{2\eta_t} \prtr{ \|\mathbf{x}_t - \mathbf{x}^*\|^2 - \mathbb{E}_{\xi_t} \prts{\|\mathbf{x}_{t+1} - \mathbf{x}^*\|^2} }
\end{equation*}
This concludes the proof.
\Halmos \end{proof}



 \begin{proof}{Proof of Lemma \ref{lemma:decay-sub-linear}.} 
For any $t \le T-1$, we evaluate the conditional expectation bounded by the safe indicator. If $\mathbb{I}_{\mathrm{safe}, t}=0$, all indicator terms evaluate to zero almost surely, yielding:
\begin{align*}
    \mathbb{E}_{\xi_t}\prts{ \prtr{F(\mathbf{x}_{t+1}) - F^*} \mathbb{I}_{\mathrm{safe}, t+1} } 
    &\le \eta_t \sigma_S^2 + \frac{\|\mathbf{x}_t - \mathbf{x}^*\|^2 \mathbb{I}_{\mathrm{safe}, t}}{2\eta_t} - \frac{\mathbb{E}_{\xi_t}\prts{\|\mathbf{x}_{t+1} - \mathbf{x}^*\|^2 \mathbb{I}_{\mathrm{safe}, t+1}}}{2\eta_t}
\end{align*}
Alternatively, if $\mathbb{I}_{\mathrm{safe}, t}=1$, we apply Lemma \ref{lemma:next-step-eog}. Because $\mathbb{I}_{\mathrm{safe}, t+1} \le \mathbb{I}_{\mathrm{safe}, t}$, we can upper-bound the subtracted term by replacing $\mathbb{I}_{\mathrm{safe}, t}$ with $\mathbb{I}_{\mathrm{safe}, t+1}$:
\begin{align*}
    \mathbb{E}_{\xi_t}\prts{ \prtr{F(\mathbf{x}_{t+1}) - F^*} \mathbb{I}_{\mathrm{safe}, t+1} } 
    &\le \eta_t \sigma_S^2 + \frac{\|\mathbf{x}_t - \mathbf{x}^*\|^2 \mathbb{I}_{\mathrm{safe}, t}}{2\eta_t} - \frac{\mathbb{E}_{\xi_t}\prts{\|\mathbf{x}_{t+1} - \mathbf{x}^*\|^2 \mathbb{I}_{\mathrm{safe}, t}}}{2\eta_t} \\
    &\le \eta_t \sigma_S^2 + \frac{\|\mathbf{x}_t - \mathbf{x}^*\|^2 \mathbb{I}_{\mathrm{safe}, t}}{2\eta_t} - \frac{\mathbb{E}_{\xi_t}\prts{\|\mathbf{x}_{t+1} - \mathbf{x}^*\|^2 \mathbb{I}_{\mathrm{safe}, t+1}}}{2\eta_t}
\end{align*}
Taking the total expectation over the entire trajectory up to step $t$ unifies these two cases:
\begin{equation*}
    \mathbb{E}\prts{ \prtr{F(\mathbf{x}_{t+1}) - F^*} \mathbb{I}_{\mathrm{safe}, t+1} } 
    \le \eta_t \sigma_S^2 + \frac{\mathbb{E}\prts{\|\mathbf{x}_t - \mathbf{x}^*\|^2 \mathbb{I}_{\mathrm{safe}, t}}}{2\eta_t} - \frac{\mathbb{E}\prts{\|\mathbf{x}_{t+1} - \mathbf{x}^*\|^2 \mathbb{I}_{\mathrm{safe}, t+1}}}{2\eta_t}
\end{equation*}
Multiplying by $\eta_t$ and summing from $t=0$ to $T-1$ forms a telescoping sum:
\begin{align*}
    \sum_{t=0}^{T-1} \eta_t \mathbb{E}\prts{ \prtr{F(\mathbf{x}_{t+1}) - F^*} \mathbb{I}_{\mathrm{safe}, t+1} } 
    &\le \sigma_S^2 \sum_{t=0}^{T-1} \eta_t^2 + \frac{\mathbb{E}\prts{\|\mathbf{x}_0 - \mathbf{x}^*\|^2 \mathbb{I}_{\mathrm{safe}, 0}}}{2} - \frac{\mathbb{E}\prts{\|\mathbf{x}_T - \mathbf{x}^*\|^2 \mathbb{I}_{\mathrm{safe}, T}}}{2} \\
    &\le \sigma_S^2 T \eta_0^2 + \frac{B_S^2}{2}
\end{align*}
where the final inequality follows by dropping the non-positive subtracted term, upper-bounding the diminishing step sizes by $\eta_0$, and substituting the bounding radius $B_S$ via Lemma \ref{lemma:bounded-distance-to-optimal}. 

Because the objective gap is non-negative and $\mathbb{I}_{\mathrm{safe}, T} \le \mathbb{I}_{\mathrm{safe}, t+1}$ for all $t \le T-1$, we can substitute the final safe indicator into the left-hand side. Dividing by the sum of step sizes and applying the convexity of $F$ to the weighted-average iterate $\bar{\mathbf{x}}_T$ yields:
\begin{align*}
    \frac{1}{\sum_{t=0}^{T-1} \eta_t} \sum_{t=0}^{T-1} \eta_t \mathbb{E}\prts{ \prtr{F(\mathbf{x}_{t+1}) - F^*} \mathbb{I}_{\mathrm{safe}, t+1} } 
    &\ge \frac{\mathbb{E}\prts{ \sum_{t=0}^{T-1} \eta_t \prtr{F(\mathbf{x}_{t+1}) - F^*} \mathbb{I}_{\mathrm{safe}, T} }}{\sum_{t=0}^{T-1} \eta_t} \\
    &\ge \mathbb{E}\prts{\prtr{F(\bar{\mathbf{x}}_T) - F^*} \mathbb{I}_{\mathrm{safe}, T}}
\end{align*}
Combining these inequalities and substituting the assumptions $\eta_0 \le \frac{1}{L_S}$ and $\sigma_S^2 \le \frac{L_S^2 B_S^2}{2T}$, the numerator simplifies perfectly:
\begin{equation*}
    \mathbb{E}\prts{\prtr{F(\bar{\mathbf{x}}_T) - F^*} \mathbb{I}_{\mathrm{safe}, T}} 
    \le \frac{\prtr{\frac{L_S^2 B_S^2}{2T}} T \prtr{\frac{1}{L_S^2}} + \frac{B_S^2}{2}}{\sum_{t=0}^{T-1} \eta_t} 
    = \frac{B_S^2}{\sum_{t=0}^{T-1} \eta_t}
\end{equation*}
Finally, applying the condition that the sum of step sizes satisfies $\sum_{t=0}^{T-1} \eta_t \ge \frac{B_S^2}{\varepsilon F^*}$ completes the proof:
\begin{equation*}
    \mathbb{E}\prts{\prtr{F(\bar{\mathbf{x}}_T) - F^*} \mathbb{I}_{\mathrm{safe}, T}} \le \frac{B_S^2}{\frac{B_S^2}{\varepsilon F^*}} = \varepsilon F^*
\end{equation*}
\Halmos \end{proof}



\begin{proof}{Proof of Lemma \ref{lemma:upper-bound-escape}.} 
Throughout this proof, we condition on the event that $\mathbb{I}_{\mathrm{safe}, t}=1$, meaning $\mathbf{x}_t \in \text{int}(S)$, and omit this notation from the conditional probabilities for brevity. Our goal is to derive an upper bound on the escape probability $\mathbb{P}\prtr{F(\mathbf{x}_{t+1}) \ge 2 F(\mathbf{x}_0)}$ in terms of $\sigma_S$. 

Denote by $G(\mathbf{x}_t, \xi_t) = \nabla F(\mathbf{x}_t) + \mathbf{w}_t$ the stochastic gradient evaluated at $\mathbf{x}_t$, where $\mathbf{w}_t$ is the stochastic noise. Assume for the moment that the trajectory escapes, meaning $F(\mathbf{x}_{t+1}) \ge 2F(\mathbf{x}_0)$. Using a similar technique to the previous lemma, we define $\mathbf{x}' := \mathbf{x}_t - \tau \eta_t G(\mathbf{x}_t, \xi_t)$ for some $\tau \in (0,1]$ as the precise point where the segment between $\mathbf{x}_t$ and $\mathbf{x}_{t+1}$ intersects the boundary of $S$. Applying the $L_S$-smoothness quadratic upper bound over the sublevel set yields:
\begin{align*}
    2F(\mathbf{x}_0) = F(\mathbf{x}') 
    \le F(\mathbf{x}_t) - \eta_t \tau \nabla F(\mathbf{x}_t)^\top G(\mathbf{x}_t, \xi_t) + \frac{\eta_t^2 \tau^2 L_S}{2}\|G(\mathbf{x}_t, \xi_t)\|^2
\end{align*}
Because $\mathbf{x}_t$ is in the interior, we know $F(\mathbf{x}_t) < 2F(\mathbf{x}_0)$. Rearranging the terms and using $\tau \le 1$, we can upper-bound the objective difference:
\begin{align*}
    2F(\mathbf{x}_0) - F(\mathbf{x}_t) 
    &\le \eta_t \tau \prtr{-\nabla F(\mathbf{x}_t)^\top G(\mathbf{x}_t, \xi_t) + \frac{\eta_t \tau L_S}{2}\|G(\mathbf{x}_t, \xi_t)\|^2} \\
    &\le \eta_t \prtr{-\nabla F(\mathbf{x}_t)^\top G(\mathbf{x}_t, \xi_t) + \frac{\eta_t L_S}{2}\|G(\mathbf{x}_t, \xi_t)\|^2}
\end{align*}
Therefore, the event of escaping the sublevel set implies this inequality holds. We can thus upper-bound the escape probability:
\begin{align*}
    \mathbb{P}\prtr{F(\mathbf{x}_{t+1}) \ge 2F(\mathbf{x}_0)} 
    &\le \mathbb{P}\prtr{2F(\mathbf{x}_0) - F(\mathbf{x}_t) \le \eta_t \prtr{-\nabla F(\mathbf{x}_t)^\top G(\mathbf{x}_t, \xi_t) + \frac{\eta_t L_S}{2}\|G(\mathbf{x}_t, \xi_t)\|^2}}
\end{align*}
We expand the inner product and use the Cauchy-Schwarz inequality $\|G(\mathbf{x}_t, \xi_t)\|^2 \le 2\prtr{\|\nabla F(\mathbf{x}_t)\|^2 + \|\mathbf{w}_t\|^2}$ to rewrite the right-hand side:
\begin{align*}
    &\mathbb{P}\prtr{F(\mathbf{x}_{t+1}) \ge 2F(\mathbf{x}_0)} \\
    &\le \mathbb{P}\prtr{2F(\mathbf{x}_0) - F(\mathbf{x}_t) \le \eta_t \prtr{-\|\nabla F(\mathbf{x}_t)\|^2 - \nabla F(\mathbf{x}_t)^\top \mathbf{w}_t + \eta_t L_S\prtr{\|\nabla F(\mathbf{x}_t)\|^2 + \|\mathbf{w}_t\|^2}}}
\end{align*}
Using the assumption that step sizes satisfy $\eta_t L_S \le \frac{1}{2}$ and applying the Cauchy-Schwarz inequality to the noise cross-term:
\begin{align*}
    &\mathbb{P}\prtr{F(\mathbf{x}_{t+1}) \ge 2F(\mathbf{x}_0)} \\
    &\le \mathbb{P}\prtr{2F(\mathbf{x}_0) - F(\mathbf{x}_t) \le \eta_t \prtr{-\frac{1}{2}\|\nabla F(\mathbf{x}_t)\|^2 + \|\nabla F(\mathbf{x}_t)\| \|\mathbf{w}_t\| + \frac{1}{2}\|\mathbf{w}_t\|^2}} \\
    &\le \mathbb{P}\prtr{2F(\mathbf{x}_0) - F(\mathbf{x}_t) + \frac{\eta_t}{2}\|\nabla F(\mathbf{x}_t)\|^2 \le \eta_t \prtr{\|\nabla F(\mathbf{x}_t)\| \|\mathbf{w}_t\| + \frac{1}{2}\|\mathbf{w}_t\|^2}}
\end{align*}
Applying Markov's inequality and bounds on the stochastic noise variance from Lemma \ref{assumption:second-moment}, we obtain:
\begin{align}
    \mathbb{P}\prtr{F(\mathbf{x}_{t+1}) \ge 2F(\mathbf{x}_0)} 
    &\le \frac{\eta_t \mathbb{E}\prts{\|\nabla F(\mathbf{x}_t)\| \|\mathbf{w}_t\| + \frac{1}{2}\|\mathbf{w}_t\|^2}}{2F(\mathbf{x}_0) - F(\mathbf{x}_t) + \frac{\eta_t}{2}\|\nabla F(\mathbf{x}_t)\|^2} \nonumber \\
    &\le \frac{\eta_t \sigma_S \prtr{\|\nabla F(\mathbf{x}_t)\| + \frac{\sigma_S}{2}}}{2F(\mathbf{x}_0) - F(\mathbf{x}_t) + \frac{\eta_t}{2}\|\nabla F(\mathbf{x}_t)\|^2} \label{eq:markov-prob}
\end{align}

This upper bound effectively balances two regimes. If $\mathbf{x}_t$ is deep inside the interior, the term $2F(\mathbf{x}_0) - F(\mathbf{x}_t)$ in the denominator is large. Conversely, if $\mathbf{x}_t$ nears the boundary, the objective gap is larger, which by convexity forces the gradient norm $\|\nabla F(\mathbf{x}_t)\|$ in the denominator to be large. We formalize this by separating the analysis into two cases based on whether $2F(\mathbf{x}_0) - F(\mathbf{x}_t) \ge F(\mathbf{x}_0)$.

\textbf{Case 1: $2F(\mathbf{x}_0) - F(\mathbf{x}_t) \ge F(\mathbf{x}_0)$.} 
This condition implies $F(\mathbf{x}_t) \le F(\mathbf{x}_0)$. We can bound the escape probability by dropping the non-negative gradient term in the denominator:
\begin{align*}
    \mathbb{P}\prtr{F(\mathbf{x}_{t+1}) \ge 2F(\mathbf{x}_0)} 
    &\le \frac{\eta_t \sigma_S \prtr{\|\nabla F(\mathbf{x}_t)\| + \frac{\sigma_S}{2}}}{F(\mathbf{x}_0)}
\end{align*}
Because $\mathbf{x}_t$ is within the $L_S$-smooth region, its gradient is bounded by the objective gap: $\|\nabla F(\mathbf{x}_t)\|^2 \le 2L_S (F(\mathbf{x}_t) - F^*) \le 2L_S F(\mathbf{x}_0)$. Substituting this and utilizing $\eta_t \le \frac{1}{2L_S}$ gives:
\begin{align*}
    \mathbb{P}\prtr{F(\mathbf{x}_{t+1}) \ge 2F(\mathbf{x}_0)} 
    &\le \frac{\sigma_S \prtr{2\sqrt{2L_S F(\mathbf{x}_0)} + \sigma_S}}{4L_S F(\mathbf{x}_0)}
\end{align*}

\textbf{Case 2: $2F(\mathbf{x}_0) - F(\mathbf{x}_t) < F(\mathbf{x}_0)$.} 
This condition implies $F(\mathbf{x}_t) > F(\mathbf{x}_0)$. We bound the escape probability by dropping the strictly positive $2F(\mathbf{x}_0) - F(\mathbf{x}_t)$ term in the denominator:
\begin{align*}
    \mathbb{P}\prtr{F(\mathbf{x}_{t+1}) \ge 2F(\mathbf{x}_0)} 
    &\le \frac{\eta_t \sigma_S \prtr{\|\nabla F(\mathbf{x}_t)\| + \frac{\sigma_S}{2}}}{\frac{\eta_t}{2}\|\nabla F(\mathbf{x}_t)\|^2} \\
    &= \frac{2\sigma_S}{\|\nabla F(\mathbf{x}_t)\|} + \frac{\sigma_S^2}{\|\nabla F(\mathbf{x}_t)\|^2}
\end{align*}
By the convexity of $F$ and the fact that $\mathbf{x}_t \in \text{int}(S)$, Lemma \ref{lemma:bounded-distance-to-optimal} guarantees $\|\mathbf{x}_t - \mathbf{x}^*\| \le B_S$. Therefore, $\|\nabla F(\mathbf{x}_t)\| \ge \frac{F(\mathbf{x}_t) - F^*}{\|\mathbf{x}_t - \mathbf{x}^*\|} \ge \frac{F(\mathbf{x}_0) - F^*}{B_S}$. Because $F(\mathbf{x}_0) > (1+\varepsilon)F^*$ by assumption, we can bound the gradient strictly away from zero:
\begin{align*}
    \mathbb{P}\prtr{F(\mathbf{x}_{t+1}) \ge 2F(\mathbf{x}_0)} 
    &\le \frac{2\sigma_S B_S}{F(\mathbf{x}_0) - F^*} + \frac{\sigma_S^2 B_S^2}{\prtr{F(\mathbf{x}_0) - F^*}^2} \\
    &\le \frac{2\sigma_S B_S}{\varepsilon F^*} + \frac{\sigma_S^2 B_S^2}{\varepsilon^2 (F^*)^2}
\end{align*}
\begin{align*}
    \mathbb{P}\prtr{F(\mathbf{x}_{t+1}) \ge 2F(\mathbf{x}_0)} 
    &\le \frac{2\sigma_S B_S}{\varepsilon F^*} + \frac{\sigma_S^2 B_S^2}{\varepsilon^2 (F^*)^2}
\end{align*}

Finally, we unify both cases using the law of total probability. Because $\sigma_S \le 1$, we can factor out $\sigma_S$ and replace $\sigma_S^2$ with $\sigma_S$ to form a strict upper bound:
\begin{align*}
    \mathbb{P}\prtr{F(\mathbf{x}_{t+1}) \ge 2F(\mathbf{x}_0)} 
    &\le \sigma_S \prtr{\frac{2\sqrt{2L_S F(\mathbf{x}_0)} + 1}{4L_S F(\mathbf{x}_0)} + \frac{2 B_S}{\varepsilon F^*} + \frac{B_S^2}{\varepsilon^2(F^*)^2}} \\
    &= \sigma_S V_S
\end{align*}
This confirms the upper bound on the escape probability.
\Halmos \end{proof}



\begin{proof}{Proof of Lemma \ref{lemma:avg-sufficient-conditions}.}
We can lower-bound the probability of reaching the target accuracy by conditioning on the event that the trajectory remains within the safe sublevel set:
\begin{align*}
    \mathbb{P}\prtr{F(\bar{\mathbf{x}}_T) \le (1+\varepsilon) F^*} 
    &\ge \mathbb{P}\prtr{F(\bar{\mathbf{x}}_T) \le (1+\varepsilon) F^* \text{ and } \mathbb{I}_{\mathrm{safe}, T}=1} \\
    &= \mathbb{P}\prtr{\mathbb{I}_{\mathrm{safe}, T}=1} - \mathbb{P}\prtr{\prtr{F(\bar{\mathbf{x}}_T) - F^*} \mathbb{I}_{\mathrm{safe}, T} > \varepsilon F^*}
\end{align*}

First, we bound the probability that the entire trajectory remains safe. Applying a union bound over the $T$ iterations and substituting the single-step escape bound from Lemma \ref{lemma:upper-bound-escape}, we have:
\begin{align*}
    \mathbb{P}\prtr{\mathbb{I}_{\mathrm{safe}, T}=1} 
    &\ge 1 - \sum_{t=0}^{T-1} \mathbb{P}\prtr{\mathbb{I}_{\mathrm{safe}, t+1}=0 \mid \mathbb{I}_{\mathrm{safe}, t}=1} \\
    &\ge 1 - T \sigma_S V_S
\end{align*}
By the assumed variance constraint, we know $\sigma_S \le \frac{\kappa}{2V_S T}$. Substituting this yields:
\begin{equation*}
    \mathbb{P}\prtr{\mathbb{I}_{\mathrm{safe}, T}=1} \ge 1 - \frac{\kappa}{2}
\end{equation*}

Next, we bound the probability of failing to reach the target accuracy despite remaining in the sublevel set. Because the step sizes satisfy $\sum_{t=0}^{T-1}\eta_t \ge \frac{2B_S^2}{\kappa \varepsilon F^*}$, the expected optimality gap established in Lemma \ref{lemma:decay-sub-linear} is bounded by $\frac{\kappa \varepsilon F^*}{2}$. Applying Markov's inequality provides:
\begin{align*}
    \mathbb{P}\prtr{\prtr{F(\bar{\mathbf{x}}_T) - F^*} \mathbb{I}_{\mathrm{safe}, T} > \varepsilon F^*} 
    &\le \frac{\mathbb{E}\prts{\prtr{F(\bar{\mathbf{x}}_T) - F^*} \mathbb{I}_{\mathrm{safe}, T}}}{\varepsilon F^*} \\
    &\le \frac{ \kappa \varepsilon F^* / 2 }{\varepsilon F^*} \\
    &= \frac{\kappa}{2}
\end{align*}

Substituting these two bounds back into the original probability decomposition completes the proof:
\begin{equation*}
    \mathbb{P}\prtr{F(\bar{\mathbf{x}}_T) \le (1+\varepsilon) F^*} \ge \prtr{1 - \frac{\kappa}{2}} - \frac{\kappa}{2} = 1 - \kappa
\end{equation*}
\Halmos \end{proof}


\begin{proof}{Proof of Lemma \ref{lemma:avg-iterate-configuration}.}
Fix $\lambda > 0$. Similar to the proof of Lemma \ref{lemma:last-iterate-configuration}, we first verify that the sum of step sizes satisfies the necessary lower bound:
\begin{align*}
    \sum_{t=0}^{T(\lambda)-1}\eta_t(\lambda) 
    &\ge \frac{1}{20L_S(\lambda)}\sum_{t=0}^{T(\lambda)-1}\frac{1}{(t+1)^\alpha} \\
    &\ge \frac{1}{20L_S(\lambda)}\frac{T(\lambda)}{T(\lambda)^\alpha} \\
    &= \frac{1}{20L_S(\lambda)} T(\lambda)^{1-\alpha} \\
    &\ge \frac{2 B_S^2(\lambda)}{\kappa \varepsilon F^*(\lambda)}
\end{align*}

Next, by standard mini-batching, the variance bound for the $B(\lambda)$-minibatch SGD update reduces to:
\begin{equation*}
    \sigma^2(\lambda) := \frac{\sigma_S^2(\lambda)}{B(\lambda)}
\end{equation*}
where $\sigma_S^2(\lambda)$ is the single-sample sublevel variance bound. Using the inequality $\frac{1}{\min(a,b,c)} \le \frac{1}{a} + \frac{1}{b} + \frac{1}{c}$ (and noting that the constants in the definition of $B(\lambda)$ strictly upper-bound the required fractions), substituting the defined $B(\lambda)$ guarantees:
\begin{equation*}
    \sigma^2(\lambda) \le \min \prtc{1, \prtr{\frac{\kappa}{2V_S(\lambda)T(\lambda)}}^2, \frac{L_S^2(\lambda)B_S^2(\lambda)}{2T(\lambda)}}
\end{equation*}

Thus, all conditions of Lemma \ref{lemma:avg-sufficient-conditions} are satisfied. We conclude that the weighted-average iterate achieves $F(\bar{\mathbf{x}}_{T(\lambda)}(\lambda); \lambda) \le (1+\varepsilon)F^*(\lambda)$ with probability at least $1-\kappa$.
\Halmos \end{proof}


\begin{proof}{Proof of Proposition \ref{lemma:avg-iterate-sufficiency}.}
For each rarity level $\lambda > 0$, if $F(\mathbf{x}_0(\lambda); \lambda) \le (1+\varepsilon)F^*(\lambda)$, the initial solution already satisfies the target accuracy, resulting in a sample complexity of zero. 

Assume instead that $F(\mathbf{x}_0(\lambda); \lambda) > (1+\varepsilon)F^*(\lambda)$. By Lemma \ref{lemma:avg-iterate-configuration}, there exists an SGD scheme $\mathcal{A}(\lambda) \in \mathbb{A}_{\alpha,\mathrm{avg}}$ that achieves the target accuracy with a sample complexity of $T(\lambda)B(\lambda)$. The proof closely mirrors that of Lemma \ref{lemma:last-iterate-sufficiency}, with the additional requirement to show that the escape bound parameter $V_S(\lambda)$ grows at most sub-exponentially in $\lambda$.

Recall from Lemma \ref{lemma:upper-bound-escape} that $V_S(\lambda)$ is a function of the initial objective value $F(\mathbf{x}_0(\lambda); \lambda)$, the optimal objective value $F^*(\lambda)$, and the sublevel set constants $L_S(\lambda)$ and $B_S(\lambda)$. By Lemma \ref{lemma:subexponential-constants}, all of these numerators grow sub-exponentially, while the denominators are strictly bounded away from zero. Therefore, $V_S(\lambda)$ grows sub-exponentially.

Because $T(\lambda)$ and $B(\lambda)$ are constructed from these sub-exponentially growing terms and denominators bounded away from zero, the total sample complexity inherits this asymptotic upper bound:
\begin{equation*}
    \limsup_{\lambda\to\infty} \frac{1}{\lambda} \log \prtr{T(\lambda) B(\lambda)} \le 0
\end{equation*}
We conclude that the sequence of configurations $\prtc{(\mathcal{A}(\lambda), \mathbf{x}_0(\lambda))}_{\lambda>0}$ is efficient.
\Halmos \end{proof}

\begin{proof}{Proof of Theorem~\ref{thm:main}.}
The theorem follows from both Proposition~\ref{lemma:last-iterate-sufficiency} for $\mathbb{A}_{\alpha,\mathrm{last}}$ and Proposition~\ref{lemma:avg-iterate-sufficiency} for $\mathbb{A}_{\alpha,\mathrm{avg}}$.
\Halmos
\end{proof}

\newpage

\section{Proofs of Proposition \ref{prop:unsafe-instance} and of Theorem \ref{thm:unsafe}.}
%
\label{appendix-sec:nec-safe-start}

\begin{proof}{Proof of Proposition \ref{prop:unsafe-instance}.}
Consider the two-dimensional optimization instance where the decision variable is denoted
as $\mathbf{x} = (u,v) \in \mathcal{X} = \mathbb{R}_{\geq 0}^2$:

\begin{equation*}
    (\ProblemRare) \qquad \min_{\mathbf{x} \in \mathbb{R}_{\geq 0}^2} F(\mathbf{x};\lambda)
    \;=\; u^2 \;+\; v \;+\; \gamma(\lambda) p(\mathbf{x})
\end{equation*}

where the risk term is $p(\mathbf{x})=e^{-u-v}$ and the rarity scaling is
$\gamma(\lambda) = e^\lambda$. If a gradient update falls outside $\mathcal{X}$, we project
it back. The average-case loss $f(\mathbf{x}) = u^2 + v$ satisfies all standard regularity
assumptions detailed in Section \ref{appendix-sec:additional_assumption}.

To show the failure comes from the landscape geometry rather than from noise, we analyze a
deterministic setting where the gradient estimates for $f$ and $p$ are exact.

We consider a projected gradient descent algorithm with a monotonically diminishing
step-size schedule $\{\eta_t\}_{t=0}^{\infty}$. By solving the first-order optimality
conditions of $\ProblemRare$, we obtain the optimal solution
$\mathbf{x}^*(\lambda) = (u^*(\lambda), v^*(\lambda))$:

\begin{equation*}
    \mathbf{x}^*(\lambda) = \left(\frac{1}{2}, \lambda - \frac{1}{2}\right).
\end{equation*}

This yields a minimum objective value of $F(\mathbf{x}^*(\lambda);\lambda) = \lambda +
\frac{3}{4}$ and a target extreme risk of $p(\mathbf{x}^*(\lambda)) = e^{-\lambda}$. Thus,
the decay rate $I$ defined in Assumption \ref{assumption:ldp-optimization} is $1$.

We select a sequence of initial solutions $\{\mathbf{x}_0(\lambda)\}_{\lambda>0}$ where
$\mathbf{x}_0(\lambda) = (u_0(\lambda), v_0(\lambda)) = \left(\frac{\lambda}{4}, 0\right)$.
We first verify that this sequence is \emph{unsafe} as per Definition~\ref{def:safe-start}.
The initial extreme risk scale is:

\begin{equation*}
    \limsup_{\lambda\to\infty}\frac{1}{\lambda}\log p\big(\mathbf{x}_0(\lambda)\big)
    = -\frac{1}{4} > -1 = -I.
\end{equation*}

Moreover, this initialization satisfies Assumption \ref{assumption:init-distance}, since it
stays geometrically close to the average-case minimizer $\tilde{\mathbf{x}}^* = (0,0)$:

\begin{equation*}
    \limsup_{\lambda\to\infty}\frac{1}{\lambda} \log \|\mathbf{x}_0(\lambda) -
    \tilde{\mathbf{x}}^*\| = \limsup_{\lambda\to\infty}\frac{1}{\lambda}
    \log\left(\frac{\lambda}{4}\right) = 0.
\end{equation*}

In the following analysis, we fix a rarity level $\lambda \geq 40$ and a step-size schedule
$\{\eta_t\}_{t=0}^{\infty}$. Let $\mathbf{x}_t = (u_t, v_t)$ be the iterate at step $t$. We
aim for a target relative accuracy $\varepsilon = 1$. In this deterministic setting, the
sample complexity $T_{\lambda}^{1,\kappa}(\mathcal{A}(\lambda), \mathbf{x}_0(\lambda))$
equals the stopping time $\tau(\lambda)$, the minimum number of iterations required for the
algorithm's output $\hat{\mathbf{x}}_\tau$ to reach the target accuracy. That is, $\tau$
must satisfy:

\begin{equation*}
    F(\hat{\mathbf{x}}_\tau;\lambda) \le 2 F^*(\lambda) = 2\lambda + 1.5.
\end{equation*}

We first analyze the sample complexity of an SGD scheme outputting the last iterate
($\mathbb{A}_{\mathrm{dim},\mathrm{last}}$), where $\hat{\mathbf{x}}_\tau = \mathbf{x}_\tau$.
We split into two cases on the initial step size $\eta_0$.

\paragraph{Case 1: Overshoot due to a large step size ($\eta_0 \ge e^{-\lambda/2}$).}
At initialization, the $v$-gradient is dominated by the risk term:

\begin{equation*}
    \nabla_v F(u_0, v_0;\lambda) = 1 - e^{\lambda - \lambda/4 - 0} = 1 - e^{3\lambda/4}.
\end{equation*}

The first update yields $v_1 = v_0 - \eta_0 \nabla_v F(u_0, v_0;\lambda) = \eta_0
e^{3\lambda/4} - \eta_0$. For all subsequent steps, since $\nabla_v F(u_t, v_t;\lambda) = 1
- e^{\lambda-u_t-v_t} \le 1$, each iteration reduces the $v$-coordinate by at most $\eta_t
\le \eta_0$. Thus, for any $t \ge 1$:

\begin{equation*}
    v_t \ge v_1 - \sum_{i=1}^{t-1} \eta_i \ge \eta_0 e^{3\lambda/4} - \eta_0 t.
\end{equation*}

This holds at the stopping time, so $v_\tau \ge \eta_0 e^{3\lambda/4} - \eta_0 \tau$. Since
$F(\mathbf{x};\lambda) \ge v$ for all $\mathbf{x} \in \mathcal{X}$, reaching the target
accuracy requires $v_\tau \le 2\lambda + 1.5$. This gives the necessary condition:

\begin{equation*}
    2\lambda + 1.5 \ge \eta_0 e^{3\lambda/4} - \eta_0 \tau \implies \tau \ge
    \frac{e^{3\lambda/4}}{\eta_0} \left(\eta_0 - \frac{2\lambda + 1.5}{e^{3\lambda/4}}\right).
\end{equation*}

Because $\eta_0 \ge e^{-\lambda/2}$, the subtracted term inside the parenthesis is
negligible for $\lambda \ge 40$. This forces $\tau \ge \Omega(e^{3\lambda/4})$, so the
sample complexity grows exponentially in $\lambda$.

\paragraph{Case 2: Insufficient progress due to a small step size ($\eta_0 < e^{-\lambda/2}$).}
By the first-order optimality condition, we lower bound the objective at any output
$\mathbf{x}_\tau$:

\begin{equation*}
    F(u_\tau, v_\tau;\lambda) \ge \min_{v} F(u_\tau, v;\lambda) = F(u_\tau, \lambda -
    u_\tau;\lambda) = u_\tau^2 - u_\tau + \lambda + 1.
\end{equation*}

To achieve $F(\mathbf{x}_\tau;\lambda) \le 2\lambda + 1.5$, we need $u_\tau^2 - u_\tau \le
\lambda + 0.5$. For $\lambda \ge 40$, this gives the necessary condition $u_\tau 
\frac{\lambda}{5}$.

For the $u$-coordinate, the gradient is upper bounded by $\nabla_u F(u_t, v_t;\lambda) =
2u_t - e^{\lambda - u_t - v_t} \le 2u_t$. Therefore the updates satisfy:

\begin{equation*}
    u_{t+1} = u_t - \eta_t \nabla_u F(u_t, v_t;\lambda) \ge u_t - \eta_t(2u_t) = u_t(1 -
    2\eta_t).
\end{equation*}

Since $\eta_t \le \eta_0 < e^{-\lambda/2}$, for $\lambda \ge 40$ we have $2\eta_t \ll 1$, so
$1-2\eta_t > 0$. Applying this recursively from $u_0 = \frac{\lambda}{4}$:

\begin{equation*}
    u_t \ge u_0 \prod_{i=0}^{t-1} (1 - 2\eta_i) \ge u_0 \left(1 - \sum_{i=0}^{t-1}
    2\eta_i\right) \ge u_0(1 - 2t\eta_0).
\end{equation*}

At the stopping time $\tau$, this gives $u_\tau \ge \frac{\lambda}{4}(1 - 2\tau\eta_0)$.
Combining with the necessary condition for success:

\begin{equation*}
    \frac{\lambda}{5} > \frac{\lambda}{4}(1 - 2\tau\eta_0) \implies 1 - 2\tau\eta_0 
    \frac{4}{5} \implies \tau > \frac{1}{10\eta_0}.
\end{equation*}

Because $\eta_0 < e^{-\lambda/2}$, this forces $\tau > \frac{1}{10} e^{\lambda/2}$. In both
cases, the sample complexity for the last iterate grows exponentially in $\lambda$.

\paragraph{Extension to Average Iterate ($\mathbb{A}_{\mathrm{dim},\mathrm{avg}}$).}
We now extend this to the weighted average iterate scheme, where the output is
$\hat{\mathbf{x}}_\tau = \frac{1}{\sum_{t=1}^\tau \eta_t} \sum_{t=1}^\tau \eta_t
\mathbf{x}_t$. The argument is the same. Because the output is a convex combination of the
prior iterates, each coordinate is bounded below by the smallest of the individual iterates
up to time $\tau$.

If the step size is large ($\eta_0 \ge e^{-\lambda/2}$), the bounds from Case 1 give
$\hat{v}_\tau \ge \min_{t \in [1, \tau]} v_t \ge \eta_0 e^{3\lambda/4} - \eta_0 \tau$. The
same necessary condition $\hat{v}_\tau \le 2\lambda + 1.5$ yields the same exponential lower
bound on $\tau$.

If the step size is small ($\eta_0 < e^{-\lambda/2}$), the bounds from Case 2 give
$\hat{u}_\tau \ge \min_{t \in [1, \tau]} u_t \ge \frac{\lambda}{4}(1 - 2\tau\eta_0)$. The
condition $\hat{u}_\tau < \frac{\lambda}{5}$ forces $\tau > \frac{1}{10\eta_0} >
\frac{1}{10} e^{\lambda/2}$. So averaging does not help the unsafe start, and the sample
complexity stays exponentially large for both schemes.

With an unsafe start, SGD either overshoots the target or makes slow progress. Either way,
for any sequence of SGD schemes from $\mathbb{A}_{\mathrm{dim},\mathrm{avg}}$ or
$\mathbb{A}_{\mathrm{dim},\mathrm{last}}$, the sample complexity grows exponentially as the
rarity level $\lambda$ increases.
\Halmos \end{proof}

\begin{proof}{Proof of Theorem~\ref{thm:unsafe}.}
Proposition~\ref{prop:unsafe-instance} provides an explicit instance to prove this theorem.
\Halmos
\end{proof}
\newpage
\section{Proofs of Proposition \ref{prop:highvariance-instance} and of Theorem \ref{thm:highvariance}.}
\label{appendix-sec:nec-weak-estimator}

\begin{proof}{Proof of Proposition \ref{prop:highvariance-instance}.}
Consider a one-dimensional stochastic optimization instance where the decision variable is $x \in \mathcal{X} = \mathbb{R}_{\geq 0}$:
\begin{equation*}
    (\ProblemRare) \qquad \min_{x \in \mathbb{R}_{\geq 0}} F(x;\lambda) \;=\; x + \gamma(\lambda) p(x),
\end{equation*}
with risk term $p(x)=e^{-x}$ and rarity scaling $\gamma(\lambda) = e^\lambda$. Solving the first-order optimality condition for $\ProblemRare$ gives the optimal solution and value
\begin{equation*}
    x^*(\lambda) = \lambda \quad \text{ and } \quad F(x^*(\lambda);\lambda) = \lambda + 1.
\end{equation*}
Let $x_t$ be the $t$-th iterate of projected SGD with a fixed step size $\eta$ from $x_0$:
\begin{equation*}
    x_{t+1}:= \max \left\{0, x_t -\eta \bar{G}(x_t, \xi_{t, 1:B}) \right\},
\end{equation*}
where the batched gradient is
\begin{equation*}
\bar{G}(x_t, \xi_{t, 1:B}) = \frac{1}{B} \sum_{i=1}^{B} \left( G_f(x_t, \xi_{t,i}) + e^\lambda G_p(x_t, \xi_{t,i}) \right).
\end{equation*}

We take a noiseless average-case estimator $G_f(x, \xi) = 1$, and build an extreme-risk estimator $G_p(x, \xi)$ that meets the standard variance bound $\mathbb{E}[\|G_p(x,\xi)+e^{-x}\|^2] \le e^{-x}$ for all $x \geq 0$ but is not weakly efficient. Inside a \textit{high-variance region} $\mathcal{H}_\lambda := (1.5(\lambda+1), 1.6(\lambda+1))$, the estimator is
\begin{equation*}
  G_p(x,\xi) =   
  \begin{cases}
  -e^{-x} + e^{-x/2} & \text{with probability } 0.5 \\
  -e^{-x} - e^{-x/2}& \text{with probability } 0.5 
  \end{cases},
\end{equation*}
and outside $\mathcal{H}_\lambda$ it is noiseless, $G_p(x,\xi) = -e^{-x}$ almost surely.

Let $\tau$ be the stopping time at which SGD reaches a $1.5$-optimal solution,
\begin{equation*}
    \tau := \min_{t} \left\{t \geq 1 \mid F(x_{t};\lambda) \leq 1.5(\lambda +1)\right\}.
\end{equation*}
Because $F(x; \lambda) > x$, reaching this sublevel set requires $x_\tau \leq 1.5(\lambda+1)$.

For the initializations $\{x_0(\lambda)\}_{\lambda>0}$ we draw from
\begin{equation*}
x_0(\lambda) \in \mathcal{X}_0(\lambda ) = [1.6(\lambda+1), 3.3(\lambda+1)+1].    
\end{equation*}
These are safe starts (Assumption~\ref{def:safe-start}): $p(x_0(\lambda)) \leq e^{-1.6\lambda}$ bounds the initial risk below the decay rate $I=1$, and the distance to the average-case minimizer $\tilde{x}^* = 0$ grows only linearly, meeting the sub-exponential initial distance requirement (Assumption \ref{assumption:init-distance}).

We show that for any step size $\eta,$ there is an initialization $x_0 \in \mathcal{X}_0$ under which the sample complexity is exponential in $\lambda$. We first show a preliminary result of when an iterate is projected to zero.

\paragraph{Recovery from zero.}
Suppose an iterate lands at $x_t = 0$ with $\eta > 3.4(\lambda+1)e^{-\lambda}$. The gradient at zero is noiseless, so $x_{t+1} = \eta(e^\lambda - 1)$, and for large $\lambda$ this exceeds $1.6(\lambda+1)$, placing the iterate well beyond the high-variance region. There the noiseless gradient lowers $x$ by at most $\eta$ per step, so returning to the boundary $x \leq 1.6(\lambda +1)$ takes at least
\begin{equation}\label{eq:start-at-zero}
\frac{x_{t+1} - 1.6(\lambda+1)}{\eta} = \frac{\eta(e^\lambda - 1) - 1.6(\lambda+1)}{\eta} = \Omega(e^\lambda)
\end{equation}
iterations.

We split the proof into three step-size regimes.

\paragraph{Case 1: Large step size ($\eta > 1.7(\lambda+1)+1$).}
Take $x_0 = 1.6(\lambda+1)+1 \in \mathcal{X}_0$. For large $\lambda$ the noiseless update gives
$$ x_1 = \max\{0, 1.6(\lambda+1)+1 - \eta(1 - e^{\lambda - x_0})\} = 0, $$
and recovery from zero is exponential by \eqref{eq:start-at-zero}.

\paragraph{Case 2: Small step size ($\eta < e^{-0.1\lambda}$).}
Take $x_0 = 1.6(\lambda+1)+1 \in \mathcal{X}_0$. Since $x_t > 1.6(\lambda+1)$ throughout, the iterate stays in the noiseless region and drops by at most $\eta$ per step. Reaching the boundary $1.6(\lambda+1)$ alone takes
\begin{equation*}
    \frac{x_0 - 1.6(\lambda+1)}{\eta} = \frac{1}{\eta} > e^{0.1\lambda}
\end{equation*}
iterations.

\paragraph{Case 3: Intermediate step size ($\eta \in [e^{-0.1\lambda}, 1.7(\lambda+1)+1]$).}
The map $x \mapsto x - \eta(1 - e^{\lambda - x})$ is continuous, so there is an $x_0 \in [1.6(\lambda +1), 3.3(\lambda +1)]$ whose first step lands in the high-variance region, $x_1 \in \mathcal{H}_\lambda = (1.5(\lambda+1), 1.6(\lambda+1))$. Once $x_1 \in \mathcal{H}_\lambda$, the high-variance noise takes effect. With batch size $B$, we track the trajectory by the sign of the empirical noise in $\bar{G}$.

\begin{itemize}
    \item \textbf{Case 3.1: Positive noise dominates.} If the batch average carries excess positive noise, $\bar{G}(x_1, \xi_{1, 1:B}) \geq 1 + \frac{e^\lambda}{B}(-e^{-x_1} + e^{-x_1/2})$, then
    \begin{equation*}
        x_2 \leq \max \left\{0, x_1 + \frac{\eta e^{\lambda-x_1} - \eta e^{\lambda-x_1/2}}{B}\right\}.
    \end{equation*}
    Since $x_1/2 < 0.8(\lambda +1)$, the term $e^{\lambda-x_1/2}$ dominates $e^{\lambda-x_1}$ for large $\lambda$, and the update projects to $x_2 = 0$ unless $B$ grows exponentially to offset it. Once at zero, $\eta \geq e^{-0.1\lambda} > 3.4(\lambda+1)e^{-\lambda}$ puts us in the recovery-from-zero case, which is exponential. So $T$ or $B$ is exponential.
    
    \item \textbf{Case 3.2: Negative noise dominates.} If the batch average carries excess negative noise, $\bar{G}(x_1, \xi_{1, 1:B}) \leq 1 + \frac{e^\lambda}{B}(-e^{-x_1} - e^{-x_1/2})$, then
    \begin{align*}
        x_2 &= \max \left\{0, x_1 - \eta \left(1 - \frac{e^{\lambda -x_1} + e^{\lambda -x_1/2}}{B}\right)\right\} \\
        &\geq 1.5(\lambda+1) + \eta \left(\frac{e^{0.2\lambda-0.8}}{B} - 1\right).
    \end{align*}
    Unless $B$ is exponentially large, the positive exponential term dominates and $x_2 > 1.6(\lambda+1)$. From there descent is capped at $\eta$, so returning to the boundary takes $\Omega(e^{0.1\lambda})$ iterations. Again $T$ or $B$ is exponential.
    
    \item \textbf{Case 3.3: Balanced noise.} The noise cancels across the batch with probability at most $0.5$, and this probability falls as $B$ grows, so this outcome cannot deliver convergence with probability $\ge 1 - \kappa$.
\end{itemize}

Therefore, for any step-size sequence there is a safe-start initialization $x_0 \in \mathcal{X}_0$ under which reaching the target accuracy with confidence $1-\kappa$ costs an exponential number of samples, whether through the iteration count $T$ or the batch size $B$.
\Halmos \end{proof}

\begin{proof}{Proof of Theorem~\ref{thm:highvariance}.}
Proposition~\ref{prop:highvariance-instance} provides an explicit instance to prove this theorem.
\Halmos
\end{proof}


\end{document}